%% file: ex_article.tex
\documentclass[onefignum,onetabnum]{siamart190516}

\usepackage{booktabs}
\input{ex_shared}

\ifpdf
\hypersetup{
  pdftitle={A Noise-Tolerant Quasi-Newton Algorithm for Unconstrained Optimization},
  pdfauthor={Hao-Jun Michael Shi, Yuchen Xie, Richard H. Byrd and J. Nocedal}
}
\fi

\algnewcommand\algorithmicinput{\textbf{Input:}}
\algnewcommand\Input{\item[\algorithmicinput]}
\algnewcommand\algorithmicoutput{\textbf{Output:}}
\algnewcommand\Output{\item[\algorithmicoutput]}

\DeclareMathOperator{\tr}{tr}

\begin{document}

\maketitle

\begin{abstract}
    This paper describes an extension of the BFGS and L-BFGS methods for the minimization of a nonlinear function subject to errors. This work is motivated by applications that contain computational noise, employ low-precision arithmetic, or are subject to statistical noise. The classical BFGS and L-BFGS methods can fail in such circumstances because the updating procedure can be corrupted and the line search can behave erratically. The proposed method addresses these difficulties and ensures that the BFGS update is stable by employing a lengthening procedure that spaces out the points at which gradient differences are collected. A new line search, designed to tolerate errors, guarantees that the Armijo-Wolfe conditions are satisfied under most reasonable conditions, and works in conjunction with the lengthening procedure. The proposed methods are shown to enjoy convergence guarantees for strongly convex functions. Detailed implementations of the methods are presented, together with encouraging numerical results. 
\end{abstract}



\section{Introduction}

Quasi-Newton methods, such as BFGS and L-BFGS, are used widely in practice because they require only first-order information and are yet able to construct useful quadratic models that make them faster and easier to use than the classical gradient method. However, in the presence of errors in the function and gradient evaluations, these methods may behave erratically. In this paper, we show how to design practical noise-tolerant versions of BFGS and L-BFGS  that  retain the robustness of their classical counterparts. The main challenge is to ensure that the updating process and the line search are not dominated by noise.  

This paper builds upon the theoretical results of Xie et al. \cite{xie2020analysis} who show that by incorporating a {lengthening procedure}, the BFGS method enjoys global convergence guarantees to a neighborhood of the solution for strongly convex functions. However, the algorithm proposed in \cite{xie2020analysis} is not practical as it requires knowledge of the strong convexity parameter $m$ of the objective function, which is normally not known. An overestimate of $m$ may lead to an unstable iteration, whereas an underestimate can  slow down convergence. The  quasi-Newton algorithms proposed in this paper compute the lengthening parameter adaptively without the need for exogenous function information; they are designed for solving general nonlinear optimization problems and are supported by a convergence analysis for strongly convex objectives. A distinctive feature of our approach is the use of a new line search procedure that works in conjunction with the lengthening technique introduced in this paper.

The problem under consideration is
\begin{equation}
    \label{eq:obj func}
    \min_{x \in \mathbb{R}^d}~\phi(x),
\end{equation}
where $\phi: \mathbb{R}^d \rightarrow \mathbb{R}$ is a smooth function. This minimization must be performed while observing only inaccurate function and gradient information, i.e., by observing
\begin{equation}
    \label{eq:noisy eval}
    f(x) = \phi(x) + \epsilon(x), \qquad g(x) = \nabla \phi(x) + e(x),
\end{equation}
where the scalar $\epsilon$ and the vector $e$ model the errors. We will consider the setting where the errors are bounded and the bounds are either known or estimated through an auxiliary procedure, such as \texttt{ECNoise} or pointwise sample variance estimation \cite{more2011estimating}. Specifically, we assume $|\epsilon(x)| \leq \epsilon_f$ and $\|e(x)\|_2 \leq \epsilon_g$ for all $x \in \mathbb{R}^d$, and that the algorithm has access to $\epsilon_f$ and $\epsilon_g$.

Problems of this type arise in many practical applications, including when the noise is computational or adversarial.  For example, in PDE-constrained optimization, the objective function often contains computational noise created by an inexact linear system solver  \cite{more2011estimating}, adaptive grids  \cite{petsc}, or other internal computations. In those applications, the optimization method may not be able to control the size of the errors. In other cases, errors are due to stochastic noise, which can be caused, for example, by an intermediate Monte Carlo simulation \cite{caflisch1998monte}. In these cases, errors may be controllable via Monte Carlo sampling. Error in the gradient can also be inherited from noise in the function within derivative-free optimization while employing gradient approximations based on finite-differencing, interpolation, or smoothing \cite{berahas2019linear,berahas2019theoretical,GillMurrSaunWrig83,GillMurrWrig81,more2012estimating,nesterov2017random}. In this case, the gradient errors can be controlled by the choice of the finite-difference interval, but can only be diminished to a certain extent under the presence of function noise. We note, however, that there are applications where noise is not bounded or where the bounds $\epsilon_f, \epsilon_g$ depend on $x$, in which case the methods proposed here cannot be directly applied.

The fact that the BFGS and L-BFGS methods can be unstable in the presence of noise  is due to the nature of the BFGS updating procedure. One simple way to illustrate this is by recalling that the Hessian approximation is updated based on  observed gradient differences:
\[   
   g(x + p)- g(x)=   \nabla \phi (x + p) + e(x+p)- \nabla \phi(x) - e(x), \quad p \in \mathbb{R}^d.
\]
If $\|p\|$ is very small, the gradients of $\phi$ could cancel out leaving only noise differences. Thus, the standard BFGS method may falter even 
before the iterates approach the region where noise dominates. Although one could argue that the situation just described is unlikely in practice,  it shows that convergence guarantees cannot be established in this case. 

To provide more concrete numerical evidence for the need to bolster the BFGS method, we illustrate in  
Figure~\ref{fig:BFGS on noisy ARWHEAD} 
the solution of the \texttt{ARWHEAD} problem \cite{gould2015cutest} in which independent random noise uniformly distributed on $[-10^{-3} , 10^{-3}]$ is introduced to each component of the gradient.  
One can observe a very large increase in the condition number of the BFGS matrix that is unseen when noise is removed. This shows that the Hessian approximation is corrupted, and an examination of the run indicates that the  line search gives rise to tiny steps once this has occurred. The \texttt{ARWHEAD} problem is chosen because it is easily solved yet clearly illustrates the instability of the BFGS matrix under the presence of noisy updates; we revisit this example in \S \ref{onlyg}.

\begin{figure}[htp]
    \centering
    \includegraphics[width=0.45\linewidth]{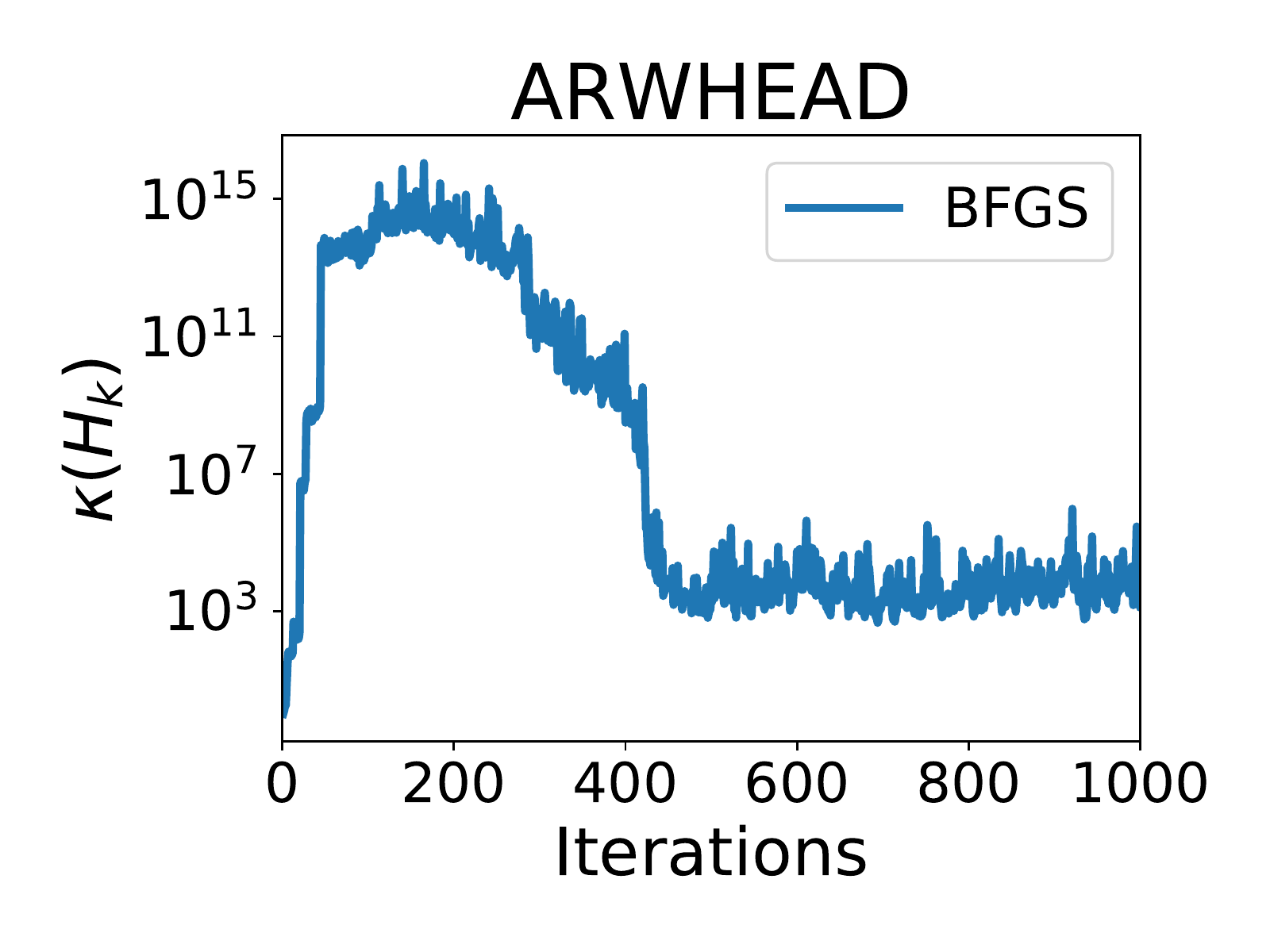}
    \caption{The condition number of the BFGS matrix $\kappa(H_k)$ against the number of iterations on the ARWHEAD problem with added noise.
    }
    \label{fig:BFGS on noisy ARWHEAD}
\end{figure}

The literature of the BFGS method with inaccurate gradients includes the implicit filtering method of Kelley et al.~\cite{choi2000superlinear, kelley2011implicit}, which assumes that noise can be diminished at will at any iteration.  Dennis and Walker \cite{DennWalker} and Ypma \cite{ypma} study bounded deterioration properties and local convergence of quasi-Newton methods with errors, when started near the solution with a Hessian approximation that is close to the exact Hessian.  
Barton \cite{barton1992computing} and Berahas, et al. \cite{berahas2019derivative} propose implementations of the BFGS method and L-BFGS method in which gradients are computed by an appropriate finite differencing technique, assuming that the noise level in the function evaluation is known. 
There has recently been some interest in designing quasi-Newton methods for machine learning applications using stochastic approximations to the gradient \cite{berahas2016multi,bollapragada2018progressive,byrd2016stochastic,gower2016stochastic,moritz2016linearly,schraudolph2007stochastic}. These papers avoid potential difficulties with BFGS or L-BFGS updating by assuming that the quality of gradient differences is sufficiently controlled, and as a result, the analysis follows similar lines as for their classical counterparts.
The work that is most relevant to this paper is by Xie et al. \cite{xie2020analysis}, who introduce the lengthening technique and establish conditions under which a steplength satisfying the  Armijo-Wolfe line search conditions exists.

The contributions of this work are as follows: i) we propose practical extensions of the BFGS and L-BFGS methods for nonlinear optimization that are capable of dealing with noise by employing a new line search/lengthening technique
that stabilizes the quasi-Newton update. This strategy relies on the noise control condition \eqref{eq:noise} introduced in this paper; ii) we provide a convergence analysis for the proposed method for strongly convex objective functions based on the properties the noise control condition instead of assuming knowledge of the strong convexity parameter, as is done in \cite{xie2020analysis}; iii) we describe implementations of the methods in full detail, and present extensive numerical results that suggest that our approach is robust for certain classes of noisy optimization problems.

The paper is organized into 6 sections. In section 2, we describe the proposed algorithms, and in section 3 we establish convergence for strongly convex objectives. In section 4, we describe practical implementations of the noise-tolerant BFGS and L-BFGS methods. In section 5, we present the results of experiments on noisy synthetic examples. Lastly, we give our final remarks in section 6. 

\section{The Algorithm}
\label{sec:algorithm}

The BFGS and L-BFGS methods for minimizing $\phi$, when only noisy observations \eqref{eq:noisy eval} of the function and gradient are available, have the form
\begin{equation}   \label{bfgs}
    x_{k + 1} = x_k - \alpha_k H_k g(x_k),
\end{equation}
where  $H_k \succ 0$ is an approximation to the inverse Hessian, $\nabla^2 \phi(x_k)^{-1},$ and the steplength $\alpha_k$ is computed by a line search.
Given  a curvature pair
\begin{align}
    (s_k, y_k)  &= (x_{k + 1} - x_k, g(x_{k + 1}) - g(x_k)) \\
     & = (\alpha_k p_k, g(x_k + \alpha_k p_k) - g(x_k)), \label{orange}
\end{align}
where $p_k = - H_k g(x_k)$, the BFGS formula updates $H_k$ as follows:
\begin{equation}
\label{eq:bfgs update}
    H_{k + 1} = (I - \rho_k s_k y_k^T) H_k (I - \rho_k y_k s_k^T) + \rho_k s_k s_k^T , \qquad\mbox{where} \  \rho_k = {1}/{y_k^Ts_k}.
\end{equation}
The L-BFGS method stores the past $t$ curvature pairs and computes the matrix-vector product $H_k g_k$ via a two-loop recursion, with memory and computational complexity that is linear with respect to the problem dimension $d$ \cite{LiuNocedal89}. For both methods, a line search ensures that $y_k^Ts_k >0$, guaranteeing that the update \eqref{eq:bfgs update} is well defined.

 As discussed in the previous section,  the difference in gradients $g(x_k + \alpha_k p_k) - g(x_k)$ may be dominated by noise, rendering the curvature information inaccurate and potentially malign.
 To safeguard against this, Xie et al. \cite{xie2020analysis} introduced a {lengthening operation}  that ensures that meaningful curvature information is being collected. Specifically, they redefine the curvature pair by \begin{equation}   \label{pairs}
    (s_k, y_k) = (\beta_k p_k, g(x_k + \beta_k p_k) - g(x_k)),
\end{equation}
where $\beta_k \geq \alpha_k$ is a sufficiently large lengthening parameter. 
The theoretical analysis in \cite{xie2020analysis} 
states that setting $\beta_k = O(\epsilon_g/m \|p_k\|)$ ensures linear convergence to a neighborhood of the solution for strongly convex problems, where $m$ is the strong convexity parameter and $\epsilon_g$ is an upper bound on the norm of the gradient noise, i.e.,
\begin{equation}  \label{upper}
    \|g(x) - \nabla \phi(x)\|_2 = \|e(x)\|_2 \leq \epsilon_g \quad\quad \forall x \in \mathbb{R}^d.
\end{equation} 
However, the analysis in \cite{xie2020analysis} does not
directly yield an implementable algorithm, as the parameter $m$ is generally not known in practice.  Furthermore, \cite{xie2020analysis}  does not propose a practical line search procedure for finding a steplength that satisfies the Armijo-Wolfe conditions---although it does establish the existence of such a steplength.

 We now propose a  rule for computing $\beta_k$ that does not require knowledge of $m$, as well as a practical line search procedure. In our approach, we enforce the following three conditions on the steplength $\alpha_k$ and the lengthening parameter $\beta_k$: 
\begin{align}
    f(x_k + \alpha_k p_k) & \leq f(x_k) + c_1 \alpha_k g(x_k)^T p_k & & \text{(Armijo condition)} \label{eq:armijo}\\
    g(x_k + \alpha_k p_k)^T p_k & \geq c_2 g(x_k)^T p_k & & \text{(Wolfe condition)}\label{eq:wolfe}\\
    (g(x_k + \beta_k p_k) - g(x_k))^T p_k & \geq 2 (1 + c_3) \epsilon_g \|p_k\| & & \text{(noise control)} \label{eq:noise}
\end{align}
where $0 < c_1 < c_2 < 1$ and $c_3 > 0$. Here and throughout the paper, $\| \cdot \|$ denotes the Euclidean norm.
The Armijo-Wolfe conditions \eqref{eq:armijo}--\eqref{eq:wolfe} ensure that the steplength $\alpha_k$ that is taken by the algorithm is not too short and yields sufficient decrease on the (noisy) objective function, while the noise control condition \eqref{eq:noise} on $\beta_k$ is designed so  that the difference in the observed directional derivatives is sufficiently large so as not to be dominated by noise. 
A motivation for \eqref{eq:noise} and a discussion of its salient properties are given below.

To satisfy the three conditions above one could find a steplength $\alpha_k$ that satisfies \eqref{eq:armijo}-\eqref{eq:wolfe}, and if \eqref{eq:noise} holds for $\beta_k = \alpha_k$, then set $\beta_k \leftarrow \alpha_k$. Otherwise, one can search for $\beta_k > \alpha_k$ to satisfy \eqref{eq:noise}. In practice, we employ a different strategy described in section~\ref{sec: two-phase} to achieve similar objectives.
 
The outline of the proposed method is given in Algorithm \ref{alg:noise-tolerant L-BFGS}.
 \newpage
\begin{algorithm}
\caption{Outline of Noise-Tolerant BFGS and L-BFGS Methods}
\label{alg:noise-tolerant L-BFGS}
\begin{algorithmic}[1]
\Input function $f(\cdot)$ and gradient $g(\cdot)$; noise level in gradient $\epsilon_g$; initial iterate $x_0$ and Hessian approximation $H_0 \succ 0$;
\For{$k = 0, 1, 2, ...$}
\State Compute $p_k = - H_k g(x_k)$ by matrix-vector multiplication (BFGS) or two-loop recursion  \cite{mybook} (L-BFGS);
\State Perform a line search to obtain $\alpha_k$ satisfying \eqref{eq:armijo} and \eqref{eq:wolfe}; if the line search fails, then compute $\alpha_k$ such that $f(x_k + \alpha_k p_k) \leq f(x_k)$;
\State Take the step $x_{k+1} = x_k + \alpha_k p_k$;
\State Perform a lengthening procedure to obtain $\beta_k$ satisfying \eqref{eq:noise};
\State Compute the curvature pair $(s_k, y_k)$ using $\beta_k$, as in \eqref{pairs};
\State Update the Hessian approximation $H_k$ by \eqref{eq:bfgs update} (BFGS) or update set $\{(s_i, y_i)\}$ of curvature pairs (L-BFGS);
\EndFor
\end{algorithmic}
\end{algorithm}

The Armijo-Wolfe line search is guaranteed to find a steplength $\alpha_k$ that satisfies conditions \eqref{eq:armijo}-\eqref{eq:wolfe} only when the gradient is sufficiently large relative to the noise level; otherwise $p_k$ is not guaranteed to be a descent direction. To handle this case, a line search failure occurs when a maximum number of trial points is computed without satisfying
\eqref{eq:armijo} and \eqref{eq:wolfe}. The algorithm requires an estimate of the noise level $\epsilon_g$, which can be obtained through sampling or through the Hamming procedure described in \cite{more2012estimating}.
The main remaining ingredient in this algorithm is a description of a procedure for computing $\alpha_k$ and $\beta_k$ in step~3 and 5. This will be discussed in \S\ref{sec:implement}.


\subsection{Motivation of the Noise Control Condition (2.9)} \label{motivation}
 We first note that the Wolfe condition \eqref{eq:wolfe} alone does not ensure that the BFGS update is productive in the noisy setting. Even though \eqref{eq:wolfe} guarantees that
\begin{equation*}
    y_k^T s_k \geq -(1 - c_2)  g(x_k)^Ts_k >0,
\end{equation*}
and this is sufficient for maintaining the positive definiteness of the BFGS matrix, this does not mean that $y_k$ properly reflects the curvature of the true function, namely $\nabla \phi(x_k + \alpha_k p_k) - \nabla \phi(x_k)$, because $y_k$ may be contaminated by noise, as discussed before.

Let us, in contrast, observe the effect of the noise control condition \eqref{eq:noise}. We have
\begin{align*}
    (g(x_k & + \beta_k p_k) - g(x_k))^T p_k \\ 
    & = \left[(\nabla \phi(x_k + \beta_k p_k) - \nabla \phi(x_k)) + (e(x_k + \beta_k p_k) - e(x_k))\right]^T p_k \\
    & \leq (\nabla \phi(x_k + \beta_k p_k) - \nabla \phi(x_k))^T p_k + (\|e(x_k + \beta_k p_k)\| + \|e(x_k)\|) \|p_k\| \\
    & \leq (\nabla \phi(x_k + \beta_k p_k) - \nabla \phi(x_k))^T p_k + 2 \epsilon_g \|p_k\|,
\end{align*}
by \eqref{upper}.
Combining this with \eqref{eq:noise} we have
\begin{equation}
\label{eq:noise bound}
     (\nabla \phi(x_k + \beta_k p_k) - \nabla \phi(x_k))^T p_k \geq 2 c_3 \epsilon_g \|p_k\|,
\end{equation}
and thus the true difference in the directional derivative is sufficiently large relative to the gradient noise $\epsilon_g$.  If we define 
\begin{equation} \label{eq: ytilde}
 \tilde{y}_k = \nabla \phi(x_k + \beta_k p_k) - \nabla \phi(x_k),
 \end{equation}
 and recall from \eqref{pairs} that $s_k= \beta_k p_k$,
 we can write \eqref{eq:noise bound} as
\begin{equation*}
     \tilde{y}_k^Ts_k \geq 2 \beta_k c_3 \epsilon_g \|p_k\|.
\end{equation*}

We can also establish a relationship between the observed and true curvature along the step $s_k$. In particular, if $\beta_k > 0$ satisfies the noise control condition \eqref{eq:noise} and \eqref{upper} holds, then
\begin{equation*}
    \left|\frac{s_k^T \tilde{y}_k}{s_k^T y_k} - 1 \right| \leq \frac{\|s_k\| \|  \tilde{y}_k - y_k\|}{s_k^T y_k} \leq \frac{2\epsilon_g\|s_k\|}{s_k^T y_k} \leq \frac{1}{1 + c_3}
\end{equation*}
which implies that 
\begin{equation}\label{eq:curv rel}
  \left(1 + \frac{1}{1 + c_3} \right)^{-1} s_k^T \tilde y_k \leq s_k^T y_k \leq \left(1 - \frac{1}{1 + c_3} \right)^{-1} s_k^T \tilde y_k.
\end{equation}

This result shows that when condition \eqref{eq:noise} is satisfied, the noisy curvature pair $(s_k, y_k)$ is a good approximation to the true curvature pair $(s_k, \tilde{y}_k)$, with the parameter $c_3$ trading off the quality of this approximation
with the locality of the curvature information being collected (in the sense that $\beta_k$ may be excessively large if $c_3$ is chosen to be large). 

Note that we are guaranteed finite termination of the lengthening procedure if there exists a $\bar{\beta} > 0$ such that for all $\beta > \bar{\beta}$,
\begin{equation*}
    \nabla \phi(x_k + \beta p_k)^T p_k \geq \nabla \phi(x_k)^T p_k + 2 c_3 \epsilon_g \|p_k\|.
\end{equation*}
This is guaranteed if $\lim_{\beta \rightarrow \infty} \nabla \phi(x_k + \beta p_k)^T p_k = \infty$, which holds for strongly convex functions, as well as for many other (but not all) functions. 

Let us now verify that the noise control condition is compatible with the choice 
\begin{equation}   \label{order}
\beta= O(\epsilon_g / m \|p_k\|)
\end{equation}
stipulated by Xie et al. \cite{xie2020analysis} in their convergence analysis of the BFGS method with errors for strongly convex functions.
If $\phi$ is $m$-strongly convex, then 
\begin{equation*}
    \tilde{y}_k^T p_k = (\nabla \phi(x_k + \beta_k p_k) - \nabla \phi(x_k))^T p_k \geq m \beta_k \| p_k \|^2.
\end{equation*}
Therefore, by our assumption, we have
\begin{equation*}
    y_k^T p_k \geq \tilde{y}_k^T p_k - 2\epsilon_g \|p_k\| \geq (m \beta_k \|p_k\| - 2 \epsilon_g)\|p_k\|.
\end{equation*}
Therefore it suffices to ensure that
\begin{equation} \label{initbeta}
    m \beta_k \|p_k\| - 2 \epsilon_g \geq 2(1+c_3) \epsilon_g,~\text{i.e.}~\beta_k   \geq \frac{2 (2 + c_3) \epsilon_g}{m \|p_k\|} ,
\end{equation}
to satisfy \eqref{eq:noise}. 



\medskip\textit{Remark 1.} 
It is natural to ask whether there are less expensive alternatives to the lengthening strategy mentioned above. The noise control condition \eqref{eq:noise} offers the possibility of skipping the BFGS update when it is not satisfied. 
We describe this approach and test it in \S\ref{numerical}. Another possibility is to use Powell damping \cite[chapter 18]{mybook}, but we consider this to be somewhat dangerous, as it would involve repeatedly introducing spurious information in the Hessian approximation without much safeguard. 

\section{Convergence Analysis}
Xie et al. \cite{xie2020analysis} established convergence results for the BFGS method using a lengthening strategy designed to cope with errors in the function and gradient.  They assume the lengthening parameter satisfies $\beta_k \|p_k\| \geq 2 \epsilon_g/m$. This leaves open the question of how to estimate the strong convexity parameter $m$ in practice so that the convergence results in \cite{xie2020analysis} still hold.

In this paper, we bypass this thorny issue and propose the lengthening strategy based on the noise control condition \eqref{eq:noise}, which employs an estimate of the noise level of the gradient $\epsilon_g$, but does not  require knowledge of $m$.
We now establish conditions under which Algorithm~\ref{alg:noise-tolerant L-BFGS} is linearly convergent to a neighborhood of the solution determined by the noise level. We make the following assumption about the underlying function $\phi$, which is standard in the analysis of quasi-Newton methods.

\begin{assumption}\label{as:mMsmooth}
The function $\phi$ is $m$-strongly convex and has $M$-Lipschitz continuous gradients, i.e., there exist constants $0 < m \leq M$ such that
\begin{equation*}
    m \| x - y \|^2 \leq  \left[\nabla \phi(x) - \nabla \phi(y)\right]^T (x - y) \leq M \|x - y\|^2,~~ \forall x, y \in \mathbb{R}^d .
\end{equation*}
\end{assumption}
In addition, we assume that the errors in the gradient and objective function approximation are bounded.
\begin{assumption}\label{as:bdd}
There are constants {$\epsilon_g \geq 0$} 
and $\epsilon_f \geq 0$ such that
\begin{equation}
    \|\nabla \phi(x) - g(x) \| \leq \epsilon_g, ~ \forall x \in \mathbb{R}^d, \quad\mbox{and}
\end{equation}
\begin{equation}
    | \phi(x) - f(x) | \leq \epsilon_f, ~ \forall x \in \mathbb{R}^d.
\end{equation}
\end{assumption}

Byrd and Nocedal \cite{ByNoTool} showed that if all curvature pairs $(s_k,y_k)$ satisfy 
\begin{equation}   \label{eq:ratios}
    \frac{s_k^T y_k}{s_k^T s_k} \geq \widehat m, \quad \frac{y_k^T y_k}{s_k^T y_k} \leq \widehat M , \quad \forall k \in \mathbb{N},
\end{equation} 
 for some constants $0 < \widehat m \leq \widehat M$, then most of the iterates generated by the (classical) BFGS method are ``good iterates'' in the sense that the angle between the search direction and the steepest direction is bounded away from orthogonality. This fact is used in \cite{ByNoTool} to establish convergence of the BFGS algorithm with various types of line searches for strongly convex functions. 

The first step in our analysis consists of showing that bounds of the form \eqref{eq:ratios} are satisfied for both the BFGS and L-BFGS versions of our noise tolerant Algorithm~\ref{alg:noise-tolerant L-BFGS}, due to the role of the noise control condition \eqref{eq:noise}. 
For convenience, we summarize the notation introduced in the previous section:
\[
    s_k = \beta_k p_k, \quad y_k = g(x_k + s_k) - g(x_k), \quad
 \tilde{y}_k = \nabla \phi(x_k + s_k) - \nabla \phi(x_k),
 \]
and therefore the noise control condition can be written as
 \[
    s_k^T \left[g(x_k + s_k) - g(x_k)\right] \geq c \, \epsilon_g \|s_k\|,
\]
with $c=2(1+c_3)$.

\medskip
{\it Notation.} So far we let $H_k$ denote the BFGS approximation of the inverse Hessian. The classical analysis of the BFGS method analyzes, however, the direct Hessian approximation $B_k$ defined as $B_k^{-1}= H_k$ \cite{mybook}. Therefore, some of the results quoted from \cite{xie2020analysis},  are stated in terms of $B_k$.
\smallskip

\begin{lemma}\label{lem:noise_to_tool}
Suppose that Assumptions \ref{as:mMsmooth} and~\ref{as:bdd} hold and that 
$s_k \neq 0$ is chosen such that
\begin{equation}\label{eq:noise_cond_dd}
    s_k^T \left[g(x_k + s_k) - g(x_k)\right] \geq c \, \epsilon_g \|s_k\|,
\end{equation}
with $c > 2$ and $\epsilon_g > 0$. 
Then we have that
\begin{equation} \label{eq:goodb}
    \frac{s_k^T y_k}{s_k^T s_k} \geq \frac{c}{c+2} m, \qquad \frac{y_k^T y_k}{s_k^T y_k} \leq \frac{c}{c-2} M .
\end{equation}
\end{lemma}

\begin{proof}
Since $\| s_k \| >0$ we have that
\begin{equation*}
    \frac{s_k^T y_k}{s_k^T s_k} \geq c \frac{\epsilon_g}{\|s_k\|}>0 .
\end{equation*}
In addition, since $\|\tilde{y}_k - y_k\| \leq 2 \epsilon_g$ and by Assumption~\ref{as:mMsmooth} we have
\begin{align*}
    \frac{s_k^T y_k}{s_k^T s_k} \geq \frac{s_k^T \tilde{y}_k}{s_k^T s_k} - \frac{2\epsilon_g}{\|s_k\|} \geq m - \frac{2\epsilon_g}{\|s_k\|}.
\end{align*}
Combining these two inequalities, we obtain
\begin{equation*}
    \frac{s_k^Ty_k}{s_k^Ts_k} \geq \frac{c}{c+2} m ,
\end{equation*}
which proves the first inequality in \eqref{eq:goodb}.

For the second bound in \eqref{eq:goodb},  first note that 
$ \|  y_k\| \leq M \|s_k\| + \| \tilde y_k - y_k \| \leq M \| s_k \| + 2 \epsilon_g.$
Therefore,
\begin{equation}  \label{pounders}
    \|s_k\|\left(M \|s_k\| + 2\epsilon_g\right) \geq \|s_k\|\|y_k\| \geq s_k^T y_k \geq c \epsilon_g \|s_k\| ,
\end{equation}
which yields the following lower bound on $\|s_k\|$:
\begin{equation} \label{pounders imp}
    \|s_k\| \geq (c-2) \frac{\epsilon_g}{M} .
\end{equation}
Since $\phi$ is $m$-strongly convex with $M$-Lipschitz continuous gradients, by \cite[Proposition 6.1.9 (b)]{bertsekas2015convex} we have
\begin{equation*}
    (x - z)^T\left[\nabla \phi(x) - \nabla \phi(z) \right] \geq \frac{mM}{m+M} \|x-z\|^2 + \frac{1}{m+M} \|\nabla \phi(x) - \nabla \phi(z)\|^2, ~\forall x, z \in \mathbb{R}^d .
\end{equation*}
Setting $x \gets x_k + s_k, z \gets x_k$, and noticing that $x - z = s_k$, $\nabla \phi(x) - \nabla \phi(z) = \tilde{y}_k$, we have
\begin{equation*}
    s_k^T \tilde{y}_k \geq \frac{mM}{m+M} \|s_k\|^2 + \frac{1}{m+M} \|\tilde{y}_k\|^2.
\end{equation*}
By re-arranging the terms, we get
\begin{equation*}
    \|\tilde{y}_k\|^2 - (M + m) s_k^T \tilde{y}_k + \left(\frac{M+m}{2}\right)^2 \|s_k\|^2 \leq \left(\frac{M-m}{2}\right)^2 \|s_k\|^2,
\end{equation*}
which is equivalent to
\begin{equation*}
    \left \|\tilde{y}_k - \frac{M+m}{2} s_k\right\| \leq \frac{M-m}{2}\left\|  s_k\right\|.
\end{equation*}

Consequently
\begin{equation*}
    \left \|y_k - \frac{M+m}{2} s_k\right\|^2 \leq \left(\frac{M-m}{2}\left\|  s_k\right\| + 2\epsilon_g\right)^2 ,
\end{equation*}
i.e.,
\begin{align*}
    & \|y_k\|^2 - (M+m) s_k^T y_k + \left(\frac{M+m}{2}\right)^2 \|s_k\|^2 \\ \leq & \left(\frac{M-m}{2}\right)^2 \|s_k\|^2 + 2(M-m)\|s_k\|\epsilon_g + 4\epsilon_g^2 .
\end{align*}
Note that we have shown $s_k^T y_k > 0$, therefore, we can simplify the equation above to
\begin{equation}  \label{eq:inter}
    \frac{y_k^T y_k}{s_k^T y_k} \leq (M+m) + \frac{(2\epsilon_g + M\|s_k\|)(2\epsilon_g - m\|s_k\|)}{s_k^T y_k}.
\end{equation}
\textbf{Case 1}: if $2\epsilon_g - m \|s_k\| < 0$, 
then we have
\begin{equation*}
    \begin{aligned}
        \frac{y_k^T y_k}{s_k^T y_k} & \leq (M+m) - \frac{\|s_k\|(2\epsilon_g + M\|s_k\|)}{s_k^T y_k} \left(m- 2\frac{\epsilon_g}{\|s_k\|}\right) \\
    \end{aligned}
\end{equation*}
From \eqref{pounders} we know that
\begin{equation*}
    \|s_k\|(M\|s_k\| + 2\epsilon_g) \geq s_k^T y_k
\end{equation*}
therefore,
\begin{equation*}
    \begin{aligned}
        \frac{y_k^T y_k}{s_k^T y_k} & \leq M + 2\frac{\epsilon_g}{\|s_k\|}
    \end{aligned}
\end{equation*}
Combining this with the lower bound $\|s_k\| \geq (c-2) {\epsilon_g}/{M}$ given in \eqref{pounders imp},
we have
\begin{equation*}
    \frac{y_k^T y_k}{s_k^T y_k} \leq M + \frac{2\epsilon_g}{\|s_k\|} \leq M + \frac{2}{c-2} M = \frac{c}{c-2} M .
\end{equation*}
\textbf{Case 2}: if $2\epsilon_g - m \|s_k\| \geq 0$, then we have from \eqref{eq:inter} and \eqref{eq:noise_cond_dd}
\begin{equation*}
\begin{aligned}
    \frac{y_k^T y_k}{s_k^T y_k} & \leq (M+m) + \frac{(2\epsilon_g + M\|s_k\|)(2\epsilon_g - m\|s_k\|)}{c \epsilon_g \|s_k\|} \\
    & = (M + m) + \frac{1}{c}\left(2 + M \frac{\|s_k\|}{\epsilon_g}\right) \left(2 \frac{\epsilon_g}{\|s_k\|} - m\right).
\end{aligned}
\end{equation*}
The right hand side increases as $\|s_k\|/\epsilon_g$ decreases, hence setting $\|s_k\|$ to the lower bound given in \eqref{pounders imp}, we have
\begin{equation*}
\begin{aligned}
    \frac{y_k^T y_k}{s_k^T y_k} 
    & \leq (M + m) + \frac{1}{c}\left(2 + M \frac{\|s_k\|}{\epsilon_g}\right) \left(2 \frac{\epsilon_g}{\|s_k\|} - m\right) \\
    & \leq (M + m) + \frac{1}{c}\left(2 + M \frac{c-2}{M}\right) \left(2 \frac{M}{c-2} - m\right) \\
    & = \frac{c}{c-2} M .
\end{aligned}
\end{equation*}
 This proves the second inequality.
\end{proof}

As mentioned above, if we set $c = 2(1 + c_3)$ in \eqref{eq:noise_cond_dd}, we obtain the  noise control condition \eqref{eq:noise}. Therefore, we have the following guarantee on the curvature pairs generated by Algorithm \ref{alg:noise-tolerant L-BFGS}:
\begin{equation}\label{eq:tool_cond}
    \frac{s_k^T y_k}{s_k^T s_k} \geq \widehat m = \frac{1+{c_3}}{2+{c_3}} m, \quad \frac{y_k^T y_k}{s_k^T y_k} \leq \widehat M = \left(1+\frac{1}{{c_3}}\right) M, ~~ k = 0, 1, 2, ...
\end{equation}

To continue using the results in  \cite{ByNoTool} we define, for any $\gamma > 0$, the index of ``good iterates'' $J(\gamma)$ as
\begin{equation}\label{eq:J_def}
    J(\gamma) = \{k \in \mathbb{N} | \cos \theta_k \geq \gamma\},
\end{equation}
where $\cos \theta_k$ is the angle between $p_k = -H_k g_k$ and $-g_k$. 
The following lemma uses the bounds (\ref{eq:tool_cond}) to show that that for some values $\gamma$, the set $J(\gamma)$ contains a fraction of the iterates.

\begin{lemma}\label{lem:lbfgs_bdd_cos} 
Let $\{x_k\}$, {$\{p_k\}$} be generated by Algorithm \ref{alg:noise-tolerant L-BFGS}, using either the full-BFGS or L-BFGS variant. 
Then for any $0 < q < 1$, there exists $\gamma > 0$ such that 
\begin{equation}\label{eq:J_qk}
    |{J(\gamma) \cap [0, k-1]}| \geq q k  ,
\end{equation}
where $J(\gamma)$ is defined by (\ref{eq:J_def}).

\end{lemma}
\begin{proof}
For the full-BFGS variant of Algorithm \ref{alg:noise-tolerant L-BFGS},
since we have shown that \eqref{eq:tool_cond} holds, Theorem 2.1 in \cite{ByNoTool} guarantees that for any $0 < q < 1$, there exists $\gamma_F > 0$ such that 
\begin{equation}
    |{J(\gamma_F) \cap [0, k-1]}| \geq q k  .
\end{equation}

For the L-BFGS method with memory length $t$, we have $B_k = H_k^{-1} = B_{k, t}$, where $B_{k, i+1}$ are computed by applying BFGS update to $B_{k, i}$ with the curvature pair $(s_{k + i - t}, y_{k + i - t})$, and $B_{k, 0}$ is defined by 
\begin{equation*}
    B_{k, 0} = \frac{1}{\gamma_k} I, ~\gamma_k = \frac{s_{k-1}^T y_{k-1}}{y_{k-1}^T y_{k-1}}.
\end{equation*}
Now we apply techniques developed in \cite{ByNoTool}. For any positive definite matrix $B$, let
\begin{equation*}
    \psi(B) = \tr B - \log \det B.
\end{equation*} 
Since all curvature pairs $\{(s_k, y_k)\}$ satisfy \eqref{eq:tool_cond}, by \cite[(2.9)]{ByNoTool} we have 
\begin{equation*}
    \psi(B_{k, i+1}) \leq \psi(B_{k, i}) + (\widehat M - \log \widehat m).
\end{equation*}
Therefore, we have 
\begin{equation*}
    \psi(B_{k}) = \psi(B_{k, t}) \leq \psi(B_{k, 0}) + t (\widehat M - \log \widehat m).
\end{equation*}
By \cite[(2.7)]{ByNoTool}, we have
\begin{align*}
    \kappa(B_{k}) & \leq \exp \left[\psi(B_{k})\right] \leq \exp \left[\psi(B_0) + t (\widehat M - \log \widehat m)  \right]  \\
    & = \left[\gamma_k e^{1/\gamma_k}\right]^d\exp \left[t (\widehat M - \log \widehat m)  \right] .
\end{align*}
By \eqref{eq:tool_cond} and the Cauchy-Schwarz inequality, 
\begin{equation*}
    \widehat{m} \leq \frac{s_{k-1}^T y_{k-1}}{s_{k-1}^T s_{k-1}} \leq \frac{y_{k-1}^T y_{k-1}}{s_{k-1}^T y_{k-1}} = \frac{1}{\gamma_k} \leq \widehat{M},
\end{equation*}
hence,
\begin{equation*}
\gamma_k e^{1/\gamma_k} = ~  e^{1/\gamma_k - \log (1/\gamma_k)} \leq ~  \exp [\widehat M - \log \widehat m],
\end{equation*}
which implies that
\begin{equation*}
    \kappa(B_{k}) \leq \exp \left[(d + t ) (\widehat M - \log \widehat m) \right].
\end{equation*}
Finally, note that since $s_k = - \beta_k H_k g_k$ and $H_k B_k = I$,
\begin{align*}
    \cos \theta_k & = \frac{g_k^T H_k g_k}{\|g_k\|\|H_k g_k\|} = \frac{s_k^T B_k s_k}{\lVert s_k \rVert \lVert B_k s_k \rVert} \geq \frac{\lambda_{\text{min}}(B_k)\lVert s_k \rVert^2}{\lambda_{\text{max}}(B_k)\lVert s_k \rVert^2} = \frac{1}{\kappa(B_k)} \\
    & \geq \exp \left[-(d + t ) (\widehat M - \log \widehat m) \right].
\end{align*}
Therefore, we have
\begin{equation*}
    \cos \theta_k \geq \gamma_L \equiv \exp \left[-(d + t ) (\widehat M - \log \widehat m) \right], ~ \forall k \in \mathbb{N} ,
\end{equation*}
i.e., 
\begin{equation*}
    |J(\gamma_L) \cap [0, k-1]| = k, ~\forall k \in \mathbb{N}
\end{equation*}
which finishes the proof.
\end{proof}

By the discussions above, for both full-BFGS and L-BFGS variants of Algorithm~\ref{alg:noise-tolerant L-BFGS}, we can choose a fixed $q^* \in (0, 1)$ and find $\gamma^* > 0$ such that 
\begin{equation}\label{eq:good_iters}
    |J(\gamma^*) \cap [0, k-1]| \geq q^* k, ~~\forall k \in \mathbb{N}; 
\end{equation}
i.e., such that a fraction of iterates are guaranteed to be good iterates. From now on, let us fix the choice $q^*$ and $\gamma^*$. 
Using the above results together with the analysis in \cite{xie2020analysis} we arrive at the following convergence result. 

\begin{theorem}\label{thm:bfgs_conv}  
Suppose that Assumptions \ref{as:mMsmooth} and \ref{as:bdd} hold. 
Let $\{x_k\}$ be generated by Algorithm \ref{alg:noise-tolerant L-BFGS}, using either L-BFGS or standard BFGS. Fix $q^* \in (0, 1)$ and choose $\gamma^* > 0$ such that \eqref{eq:good_iters} holds. 
Define
\begin{equation}\label{eq:n1_def}
   \mathcal{N}_1 = \left\{x ~\Bigg|~ \lVert{\nabla \phi(x)}\rVert \leq \max \Bigg\{ A \frac{\sqrt{{M}\epsilon_f}}{\gamma^*}, B  \frac{\epsilon_g}{\gamma^*}  \Bigg\}\right\},
\end{equation}
and 
\begin{equation}\label{eq:n2_def}
    \mathcal{N}_2 = \left\{x ~\Big|~ \phi(x) \leq 2\epsilon_f + \max_{y \in \mathcal{N}_1} ~ \phi(y)\right\} \supseteq  \mathcal{N}_1,
\end{equation}
where 
\begin{align*}
	A & = \max \left\{\frac{16 \sqrt{2}}{\sqrt{(c_2-c_1) (4-c_1-3 c_2)}},\frac{8}{\sqrt{c_1(1-c_2)}}\right\} \\
	B & = \max \left\{\frac{8}{1-c_2},\frac{8 (1+c_1)}{c_2-c_1}+6\right\}.
\end{align*}
Let
\begin{equation}\label{eq:K_def}
    K = \min_{k} ~ \{k \in \mathbb{N} ~|~ x_k \in \mathcal{N}_1\}
\end{equation}
be the index of the first iterate that enters $\mathcal{N}_1$. Assume that for all iterates $k \in J(\gamma^*)$ such that $x_k \notin \mathcal{N}_1$ the line search procedure finds $\alpha_k$ satisfying  \eqref{eq:armijo}--\eqref{eq:wolfe}. Then there exists $\rho \in (0, 1)$ such that
\begin{equation*}
    \phi(x_k) - \phi^* \leq \rho^k ~ (\phi(x_0) - \phi^*) + 2\epsilon_f, ~~ \forall k \leq K-1 .
\end{equation*}
Moreover, we have that $K < + \infty$ and
\begin{equation*}
    x_k \in \mathcal{N}_2, ~~ \forall k \geq K.
\end{equation*}
\end{theorem}
\begin{proof}
Note that Algorithm \ref{alg:noise-tolerant L-BFGS} differs from Algorithm 2.1 of \cite{xie2020analysis}, only in the quasi-Newton updating strategy and lengthening procedure.  This implies that the results through Theorem 3.5 of \cite{xie2020analysis} concerning the existence of an Armijo-Wolfe stepsize, also apply to Algorithm 2.1 of this paper, since the proofs of these these results do not depend on the update used. In Lemma~\ref{lem:noise_to_tool} of this paper we showed that the  lengthening procedure in step~5 of Algorithm~\ref{alg:noise-tolerant L-BFGS} guarantees bounds on $(s_k^T y_k/s_k^T s_k)$ and $(y_k^T y_k/s_k^T y_k) $
such as those of Lemma 3.8 of \cite{xie2020analysis}. Using these bounds we established Lemma~\ref{lem:lbfgs_bdd_cos} whose results are identical to those of Corollary 3.10 in \cite{xie2020analysis}, with $\gamma^*$ replacing $\beta_1$. With that change, the rest of the results of \cite{xie2020analysis}, including Theorems 3.16--3.18, hold for Algorithm~\ref{alg:noise-tolerant L-BFGS} of this paper,  proving the theorem.
\end{proof}

Theorem \ref{thm:bfgs_conv} states that the iterates generated by Algorithm \ref{alg:noise-tolerant L-BFGS} converge linearly to a neighborhood of the solution $\mathcal{N}_1$, whose size depends on the noise levels $\epsilon_f, \epsilon_g$; the iterates will enter $\mathcal{N}_1$ in finite number of iterations, and will remain in a larger neighborhood $\mathcal{N}_2$ thereafter.

\section{A Practical Algorithm}
In order to implement Algorithm~\ref{alg:noise-tolerant L-BFGS}, we need to design  a  practical procedure for computing 
the steplength  $\alpha_k$ and the lengthening parameter $\beta_k$. This can be done in various ways, and in this section we present a technique that has performed well in practice. After describing this algorithm in  detail, we present several heuristics designed to improve its practical performance.

\subsection{Two-Phase Line Search and Lengthening Procedure}\label{sec: two-phase}

Algorithm \ref{alg:noise-tolerant L-BFGS} and the convergence analysis of the previous section require that $\alpha_k$ and $\beta_k$ satisfy conditions \eqref{eq:armijo}, \eqref{eq:wolfe} and \eqref{eq:noise}. We now propose a procedure for computing these quantities.

The line search operates in two phases. The \textit{initial phase} attempts to satisfy three conditions with the same parameter $\alpha_k = \beta_k$:
\begin{align}
    f(x_k + \alpha_k p_k) & \leq f(x_k) + c_1 \alpha_k g(x_k)^T p_k & &  \label{eq:relax armijo} \\
    g(x_k + \alpha_k p_k)^T p_k & \geq c_2 g(x_k)^T p_k & &  \label{eq:relax wolfe}\\
    |(g(x_k + \alpha_k p_k) - g(x_k))^T p_k| & \geq 2 (1 + c_3) \epsilon_g \|p_k\|, & &  \label{eq:relax noise}
\end{align}
where $0 < c_1 < c_2 < 1$ and $c_3 > 0$. Observe that  \eqref{eq:relax noise} and the Wolfe condition \eqref{eq:relax wolfe} imply the  noise control condition \eqref{eq:noise} employed so far in the paper. We incorporate the absolute value in \eqref{eq:relax noise} in order to introduce a symmetric noise condition that can be used to determine when to adapt $\alpha_k$ and $\beta_k$ independently. If $\epsilon_f = \epsilon_g = 0$, then we can guarantee that the initial phase will reduce to the standard Armijo-Wolfe line search, as we describe below.


The initial phase is done using the logic of the standard bisection search: backtracking if the Armijo condition is not satisfied, and advancing if the Armijo condition is satisfied and the Wolfe condition is not, but with one important modification. If the Armijo condition \eqref{eq:relax armijo} is satisfied, we will check \eqref{eq:relax noise} prior to checking the Wolfe condition \eqref{eq:relax wolfe}.

If at \textit{any} iteration of the line search the noise control condition \eqref{eq:relax noise} is not satisfied or if the line search has performed more than the allowed number ($N_{\text{split}}$) of iterations, then the initial phase is terminated and the second phase, which we call the \textit{split phase}, is triggered. In this phase, $\alpha_k$ and $\beta_k$ are updated independently from each other. The steplength $\alpha_k$ is updated via the standard Armijo backtracking line search while the lengthening parameter $\beta_k$ is lengthened independently until the conditions
\begin{align}
    f(x_k + \alpha_k p_k) & \leq f(x_k) + c_1 \alpha_k g(x_k)^T p_k \\
    (g(x_k + \beta_k p_k) - g(x_k))^T p_k & \geq 2 (1 + c_3) \epsilon_g \|p_k\|
\end{align}
are satisfied. We backtrack more aggressively (by a factor of 10) in the split phase in order to mitigate the cost of additional function evaluations. The limit $N_{\text{split}}$ is imposed to prevent the line search from being fooled from noise indefinitely.

The  two-phase line search (without heuristics) is presented in Algorithms \ref{alg:initial phase} and \ref{alg:split phase}. For completeness, we also present the pseudocode for the complete practical algorithm in \ref{alg:complete}.

\begin{algorithm}
\caption{Two-Phase Armijo-Wolfe Line Search and Lengthening: Initial Phase}
\label{alg:initial phase}
\begin{algorithmic}[1]
\Input functions $f(\cdot)$ and $g(\cdot)$; noise level $\epsilon_g$; current iterate $x$; search direction $p$; initial steplength $\alpha = 1$; constants $0 < c_1 < c_2 < 1$, $c_3 > 0$; maximum number of line search iterations before split $N_{\text{split}}$
\State $l \leftarrow 0$, $u \leftarrow \infty$; \Comment{Initialize brackets for bisection}
\For{$i = 0, 1, 2, ..., N_{\text{split}} - 1$}
\If{$f(x + \alpha p) > f(x) + c_1 \alpha g(x)^T p$} \Comment{Armijo condition fails}
\State $u \leftarrow \alpha$;
\State $\alpha \leftarrow (u + l) / 2$; \Comment{Backtrack}
\ElsIf{$|(g(x + \alpha p) - g(x))^T p| < 2(1 + c_3) \epsilon_g \|p\|$} \Comment{Noise control condition fails}    
\State Break (for loop)
\ElsIf{$g(x + \alpha p)^T p < c_2 g(x)^T p$} \Comment{Wolfe condition fails}
\State $l \leftarrow \alpha$;
\If{$u = \infty$} \Comment{Advance}
\State $\alpha \leftarrow 2 \alpha$;
\Else
\State $\alpha \leftarrow (u + l) / 2$;
\EndIf
\Else \Comment{Satisfies all conditions}
\State $\beta \leftarrow \alpha$ ;
\State Return $\alpha, \beta$;
\EndIf
\EndFor
\State $\alpha, \beta \leftarrow {\rm SplitPhase}(f, g, \epsilon_g, x, p, \alpha, \beta)$; \Comment{Enter split phase}
\State Return $\alpha, \beta$;
\end{algorithmic}
\end{algorithm}

\begin{algorithm}
\caption{Split Phase}
\label{alg:split phase}
\begin{algorithmic}[1]
\Input functions $f(\cdot)$ and $g(\cdot)$; noise level $\epsilon_g$; current iterate $x$; search direction $p$; initial steplength $\alpha$; initial lengthening parameter $\beta$, constants $0 < c_1 < c_2 < 1$, $c_3 > 0$
\While{$f(x + \alpha p) > f(x) + c_1 \alpha g(x)^T p$} \Comment{Armijo condition}
\State $\alpha = \alpha / 10$; \Comment{Backtrack}
\EndWhile
\While{$(g(x + \beta p) - g(x))^T p < 2 (1 + c_3) \epsilon_g \|p\|$} \Comment{Noise control condition}
\State $\beta = 2 \beta$; \Comment{Lengthen}
\EndWhile
\State Return $\alpha, \beta$;
\end{algorithmic}
\end{algorithm}

\begin{algorithm}
\caption{Complete Practical Noise-Tolerant BFGS and L-BFGS Methods}
\label{alg:complete}
\begin{algorithmic}[1]
\Input function $f(\cdot)$ and gradient $g(\cdot)$; noise level in function $\epsilon_f$, noise level in gradient $\epsilon_g$; initial iterate $x_0$ and Hessian approximation $H_0 \succ 0$;
\For{$k = 0, 1, 2, ...$}
\State Compute $p_k = - H_k g(x_k)$ by matrix-vector multiplication (BFGS) or two-loop recursion \cite{mybook} (L-BFGS);
\State Perform two-phase Armijo-Wolfe line search (Algorithms \ref{alg:initial phase} and \ref{alg:split phase}) to find $\alpha_k$ and $\beta_k$;
\If{$\alpha_k$ satisfies \eqref{eq:relax armijo}}
\State Compute $x_{k+1} = x_k + \alpha_k p_k$;
\EndIf
\If{$\beta_k$ satisfies \eqref{eq:noise}}
\State Compute curvature pair $(s_k, y_k) = (\beta_k p_k, g(x_k + \beta_k p_k) - g(x_k))$;
\State Update $H_k$ by \eqref{eq:bfgs update} (BFGS) or update set $\{(s_i, y_i)\}$ of curvature pairs (L-BFGS);
\EndIf
\EndFor
\end{algorithmic}
\end{algorithm}

By the design of the two-phase line search, our algorithm behaves the same as the standard (L-)BFGS algorithm (without interpolation) for non-noisy problems as long as $N_{\text{split}}$ is sufficiently large because the split phase will never occur. In particular, if $\epsilon_g = 0$, then condition \ref{eq:relax noise} will always be satisfied by any $\alpha_k$ and therefore the initial phase reduces to the standard Armijo-Wolfe line search.
However, unlike the deterministic setting, the two-phase line search may not be guaranteed to find $\alpha_k$ and $\beta_k$ under certain scenarios. When the iteration has reached the region where errors are large relative to the gradient, the backtracking line search may fail to find $\alpha_k$; this is to be expected. A more subtle case is when the function is exceedingly flat along the search direction $p_k$ so that even for a large $\beta$ the function exhibits insufficient change in curvature; in this case the lengthening procedure may fail to find an appropriate $\beta_k$. To safeguard against both of these cases, the algorithm will terminate if it reaches a maximum number of line search iterations.

\medskip\textit{Remark 2.}  The two-phase algorithm just described may seem too complex. Let us  consider some simpler alternative strategies. 
One approach is to employ only the split phase: (1) Compute $\alpha_k$ solely through a backtracking line search until the Armijo condition is satisfied; and (2) Computing $\beta_k$ through a lengthening procedure that enforces both of the modified noise control and  Wolfe conditions. However, the Wolfe condition on the steplength $\alpha_k$ allows the algorithm to take longer steps that may yield larger reductions in the objective function. This is in agreement with our computational experience.

A second alternative, given in Algorithm~\ref{alg:noise-tolerant L-BFGS}, is the approach employed by Xie et al. \cite{xie2020analysis}, who first solve for a steplength $\alpha_k$ that satisfies the Armijo-Wolfe conditions \eqref{eq:relax armijo}-\eqref{eq:relax wolfe}, then lengthen $\beta_k \geq \alpha_k$ until $\beta_k$ satisfies the noise control condition \eqref{eq:relax noise}. However, we have found experimentally that performing an Armijo-Wolfe line search attempting to find a steplength that satisfies the Armijo-Wolfe conditions in the presence of noise can be costly in terms of function and gradient evaluations because the Armijo-Wolfe line search may be fooled for many iterations in the presence of moderate to large noise relative to the gradient. In particular, enforcing the Wolfe condition on the steplength when the gradient is dominated by noise may lead to ill-advised or unnecessary changes to the steplength.
Rather than doing this, we opt to split the computations of $\beta$ and $\alpha$ earlier, as done in Algorithm \ref{alg:initial phase} using \eqref{eq:relax noise} as a means to detect when to split and consider the Wolfe condition unreliable.

\subsection{Heuristics}  \label{sec:implement}
We now describe some heuristics that have improved the performance of the two-phase line search for the models of noise employed in our experiments.

\medskip{\it I. Relaxation of Armijo Condition.}
The last term in the Armijo condition \eqref{eq:relax armijo} ensures sufficient descent, but is useful only if the quantities involved are reliable; otherwise it is best to dispense with this term. To see this, consider the term  $g(x_k)^T p_k$. Although $g(x_k)^T p_k = - g(x_k)^T H_k g(x_k) < 0$ since $H_k$ is positive definite, this quantity could still be dominated by noise. If $g(x_k)^T p_k < - \epsilon_g \|p_k\|$, we can guarantee that $\nabla \phi(x_k)^T p_k < 0$, ensuring that $p_k$ is a descent direction with respect to the true objective function. If instead we have that $g(x_k)^T p_k \geq - \epsilon_g \|p_k\|$, it is not guaranteed that we can make progress on the true objective function along $p_k$. In this case, we will consider the gradient estimate unreliable and dispense the sufficient decrease term, instead relaxing the condition to only enforce simple decrease
$f(x_k + \alpha p_k) < f(x_k).$

Another feature that is useful when the algorithm reaches a region where the noise in the function is large relative to the  objective function is
%
%
to  relax the Armijo condition \eqref{eq:relax armijo} by adding  $2 \epsilon_f$ to the right hand side. This relaxation will be done only after the first attempt at satisfying the standard Armijo condition fails. If $p_k$ is a descent direction with respect to $\phi$, which is ensured when $g(x_k)^T p_k < - \epsilon_g \|p_k\|$, then this relaxation guarantees finite termination of the line search component in the split phase. Other related line searches employing this relaxation of the Armijo condition have been analyzed in \cite{berahas2019global}. 

Combining the two strategies described above, our relaxed Armijo condition can be summarized as follows:
\begin{equation} \label{eq:relaxed armijo}
    f(x_k + \alpha_k^i p_k) 
    \begin{cases}
    \leq f(x_k) + c_1 \alpha_k^i g(x_k)^T p_k & \mbox{if } i = 0, ~ g(x_k)^T p_k < -\epsilon_g \|p_k\| \\
    \leq f(x_k) + c_1 \alpha_k^i g(x_k)^T p_k + 2 \epsilon_f & \mbox{if } i \geq 1, ~ g(x_k)^T p_k < -\epsilon_g \|p_k\| \\
    < f(x_k) & \mbox{if } i = 0, ~ g(x_k)^T p_k \geq -\epsilon_g \|p_k\| \\
    < f(x_k) + 2 \epsilon_f & \mbox{if } i \geq 1, ~ g(x_k)^T p_k \geq -\epsilon_g \|p_k\|
    \end{cases}
\end{equation}
where $\alpha^i_k$ denotes the $i$-th trial steplength at iteration $k$.

\medskip{\it II. Reusing Previously Computed $\alpha$.} \label{sec: reinit}
Over the course of the initial phase, we will track the best steplength that we have seen that satisfies the Armijo condition
\begin{equation}
    \alpha_k^{\text{best}} \in \arg\min_{\alpha_k^i} \{f(x_k + \alpha_k^i p_k) : \text{\eqref{eq:relaxed armijo} is satisfied}\}
\end{equation}
as well as its corresponding function value. If the split phase is triggered, we will accept the previously computed value of $\alpha_k = \alpha_k^{\text{best}}$ that most decreased the objective function.


\medskip{\it III. Initial Value of $\beta$.}
It is important to employ a good initial estimate of $\beta$ when entering the split phase, in  order to mitigate the cost of the search procedure.   Recall from \eqref{initbeta} that an appropriate value of the lengthening parameter is, roughly,
\begin{equation}  \label{inbeta}
    {\beta}_k = \frac{2(1 + c_3) \epsilon_g}{m \|p_k\|_2}.
\end{equation}
This formula relies on the strong convexity parameter $m$, which is generally not known, but since we are only using it to compute an initial value for $\beta$, it is not critical to estimate $m$ accurately. In this vein, we compute a local estimate of $m$ using
the observed $(s, y)$ pairs from prior iterations. For $\beta_j$ with $j < k$ that satisfies both \eqref{eq:relax wolfe} and \eqref{eq:relax noise}, we first compute an estimate of the curvature along the search direction {$p_j$} corresponding to the interval length $\beta_j$:

\begin{equation}
    \bar{\mu}_j = \frac{(g(x_j + \beta_j p_j) - g(x_j))^T p_j}{\beta_j \|p_j\|_2^2}.
\end{equation}
To estimate the strong convexity parameter $m$ we track the last $h$ values of the $\bar{\mu}$'s, then use the smallest of these:
$$\mu_k = \min\{\bar{\mu}_{k - 1}, \bar{\mu}_{k - 2}, ..., \bar{\mu}_{k - h}\}.$$
This aims to be only a local strong convexity estimate, whereas taking the minimum over all previous $\bar{\mu}$'s may be overly pessimistic. Let us denote by $\bar{\beta}_k$ the value obtained by making the substitution $m \leftarrow \mu_k$ in \eqref{inbeta}, 
and let $\beta_k^i$ denote the $i$th trial lengthening parameter at iteration $k$, we define the initial value of the lengthening parameter for the split phase as 
\begin{equation}
    \beta_k^{i + 1} = \max \{2 \beta_k^i, \bar{\beta}_k\}.
\end{equation}
We have observed in our tests that this procedures allows us to significantly mitigate the cost of additional gradient evaluations that are incurred when lengthening $\beta_k$, only requiring an additional $1-3$ gradient evaluations for the lengthening procedure in our experiments.

\section{Numerical Experiments}     \label{numerical}

In this section, we present numerical results illustrating the behavior of the methods proposed in this paper on noisy optimization problems. We compare the classical methods, BFGS and L-BFGS, with their extensions, which we denote as BFGS-E and L-BFGS-E. 

In addition, we study another approach suggested by the noise control condition \eqref{eq:noise}, based on the well known strategy of skipping a quasi-Newton update when it may not be reliable. In the BFGS (Skips) and L-BFGS (Skips) methods, the quasi-Newton update is not performed if the noise control condition is not satisfied for $c_3 = 0$, that is,
\begin{equation}   \label{skipping}
    (g(x_k + \alpha_k p_k) - g(x_k))^T p_k < 2 \epsilon_g \|p_k\|.
\end{equation}
Specifically, these methods compute a steplength $\alpha_k$ satisfying the Armijo-Wolfe conditions \eqref{eq:armijo}-\eqref{eq:wolfe}, and if condition \eqref{skipping} holds, the BFGS update is not performed and the next step is computed using the Hessian approximation $B_k$ from the previous iteration; otherwise the iteration is identical to that of the BFGS and L-BFGS methods. (In the L-BFGS (Skips) method, the correction pair $(s_k, y_k)$ is not stored when \eqref{skipping} holds.)

In summary, the 6 methods tested are:
\begin{enumerate}
    \item  {\tt BFGS}: the standard BFGS method given by
    \eqref{bfgs}, \eqref{eq:bfgs update}; 
    \item {\tt L-BFGS}: the standard L-BFGS method with memory $t= 10$; \cite[chapter 7]{mybook};
    \item {\tt BFGS (Skips)}: the standard BFGS method given by \eqref{bfgs}, \eqref{eq:bfgs update}, but skipping the BFGS update when \eqref{eq:noise} is not satisfied for $\beta_k = \alpha_k$ with $c_3 = 0$;
    \item {\tt L-BFGS (Skips)}: the standard L-BFGS method with memory $t = 10$, but skipping the L-BFGS update when \eqref{eq:noise} is not satisfied for $\beta_k = \alpha_k$ with $c_3 = 0$;
    \item {\tt BFGS-E}: the noise tolerant BFGS method given by Algorithms~\ref{alg:noise-tolerant L-BFGS}, \ref{alg:initial phase} and \ref{alg:split phase};
    \item {\tt L-BFGS-E}: the noise tolerant L-BFGS method, which is identical to BFGS-E, except that the Hessian approximation is a limited memory matrix with memory $t=10$.
\end{enumerate}

The first four methods employ an Armijo-Wolfe line search that computes a steplength satisfying \eqref{eq:armijo}-\eqref{eq:wolfe}. The last two methods use the specialized line search described in Algorithms \ref{alg:initial phase} and \ref{alg:split phase}. In the deterministic case, it is common to employ cubic or quadratic interpolation to accelerate  the Armijo-Wolfe search. We did not do so in the  methods listed above, which use a simple bisection, because it is more robust in the presence of noise. The parameters for the line search and termination criteria are provided in Table \ref{tab:params}. 

\begin{table}[htp]
    \centering
    \caption{Parameter Settings for the Methods Tested}
    \label{tab:params}
    \begin{tabular}{c c c c c c}
        \toprule
        $c_1$ & $c_2$ & $c_3$ & $t$ & $N_{\text{split}}$ \\
        \midrule
        $10^{-4}$ & $0.9$ & $0.5$ & $10$ & $30$ \\
        \bottomrule
    \end{tabular}
\end{table}
We selected 41 unconstrained problems from the CUTEst collection \cite{cuter} (see Table \ref{tab:cutest}), and added stochastic uniform noise with different noise levels. The objective function and gradient have the form
\[
    f(x)  = \phi(x) + \epsilon(x), \qquad
    g(x)  = \nabla \phi(x) + e(x),
\]
where we sample $\epsilon(x)$ and $[e(x)]_i$ independently with distribution
\[
    \epsilon(x)  \sim U(-\xi_f, \xi_f), \qquad
    [e(x)]_i  \sim U\left(- \xi_g, \xi_g \right) \mbox{ for } i = 1, ..., d.
\]
This gives the noise bounds $|\epsilon(x)| \leq \epsilon_f = \xi_f$ and $\|e(x)\| \leq \epsilon_g = \sqrt{d} \xi_g$. Among methods for uncertainty quantification, ECNoise \cite{more2011estimating}, point-wise sampling, and domain knowledge could be applied to obtain these bounds in practice. The optimal value $\phi^*$ for each function was obtained by applying the BFGS method to the original deterministic problem until it could not make further progress.

The performance of the methods is best understood by studying the runs on each of the 41 test problems. Since this is impractical due to space limitations, for every experiment, we selected a problem that illustrates typical behavior over the whole test set.

\begin{table}[htp]
    \centering
    \caption{Unconstrained CUTEst Problems Tested. $d$ is the number of variables.}
    \label{tab:cutest}
    \small
    \begin{tabular}{c c c c c c c c}
        \toprule
        PROBLEM & $d$ & & PROBLEM & $d$ & & PROBLEM & $d$ \\
        \midrule
        ARWHEAD & $100$ & & DIXMAANL & $90$ & &  MOREBV & $100$ \\
        BDQRTIC & $100$ & & DIXMAANM & $90$ & &  NCB20B & $100$ \\
        CRAGGLVY & $100$ & & DIXMAANN & $90$ & & NONDIA & $100$ \\
        DIXMAANA & $90$ & & DIXMAANO & $90$ & & NONDQUAR & $100$ \\
        DIXMAANB & $90$ & & DIXMAANP & $90$ & & PENALTY1 & $100$ \\
        DIXMAANC & $90$ & & DQDRTIC & $100$ & & QUARTC & $100$ \\
        DIXMAAND & $90$ & & DQRTIC & $100$ & & SINQUAD & $100$ \\
        DIXMAANE & $90$ & & EIGENALS & $110$ & & SPARSQUR & $100$ \\
        DIXMAANF & $90$ & & EIGENBLS & $110$ & & TOINTGSS & $100$ \\
        DIXMAANG & $90$ & & EIGENCLS & $30$ & & TQUARTIC & $100$ \\
        DIXMAANH & $90$ & & ENGVAL1 & $100$ & & TRIDIA & $100$ \\
        DIXMAANI & $90$ & & FLETCBV3 & $100$ & & WATSON & $31$ \\
        DIXMAANJ & $90$ & & FREUROTH & $100$ & & WOODS & $100$ \\
        DIXMAANK & $90$ & & GENROSE & $100$ \\
        \bottomrule
    \end{tabular}
\end{table}

\subsection{Experiments with Uniform Noise in the Gradient}  \label{onlyg}

In the first set of experiments, the gradient contains uniform noise but the function does not, i.e., $\epsilon_g > 0$ and $\epsilon_f = 0$. This allows us to test the efficiency of the lengthening procedure in a benign setting that avoids the effects of the noisy line search. In these experiments, all algorithms were run for a fixed number of iterations.


We begin by revisiting the \texttt{ARWHEAD} problem from Figure \ref{fig:BFGS on noisy ARWHEAD}, where noise was inserted with $\xi_f = 0$ and $\xi_g = 10^{-3}$. The condition number of the BFGS and BFGS-E matrices is compared in Figure \ref{fig:BFGS-E on noisy ARWHEAD}, 
and shows that the noise control condition \eqref{eq:noise} stabilizes the quasi-Newton update. 
\begin{figure}[htp]
    \centering
    \includegraphics[width=0.45\linewidth]{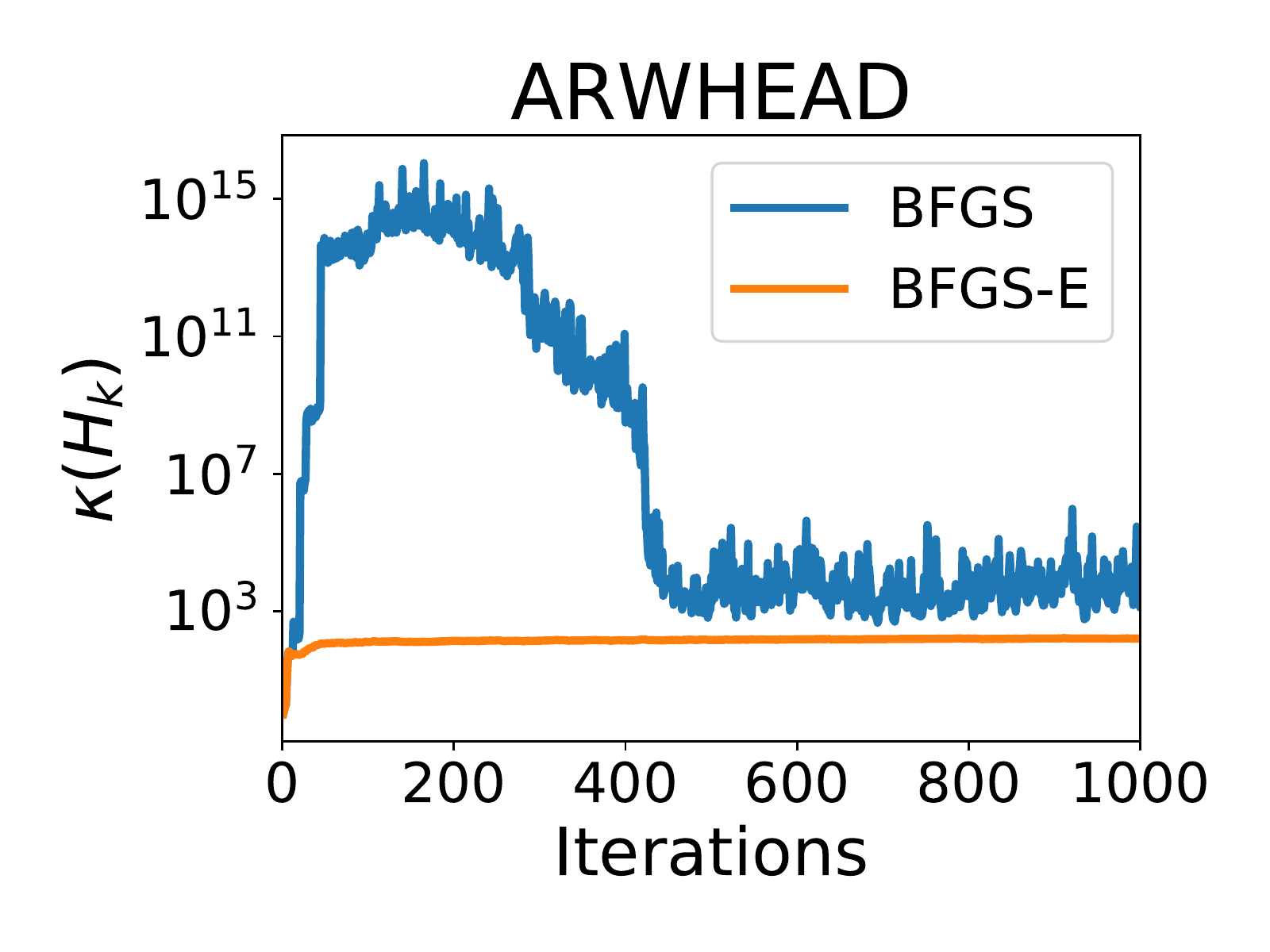}
    \includegraphics[width=0.45\linewidth]{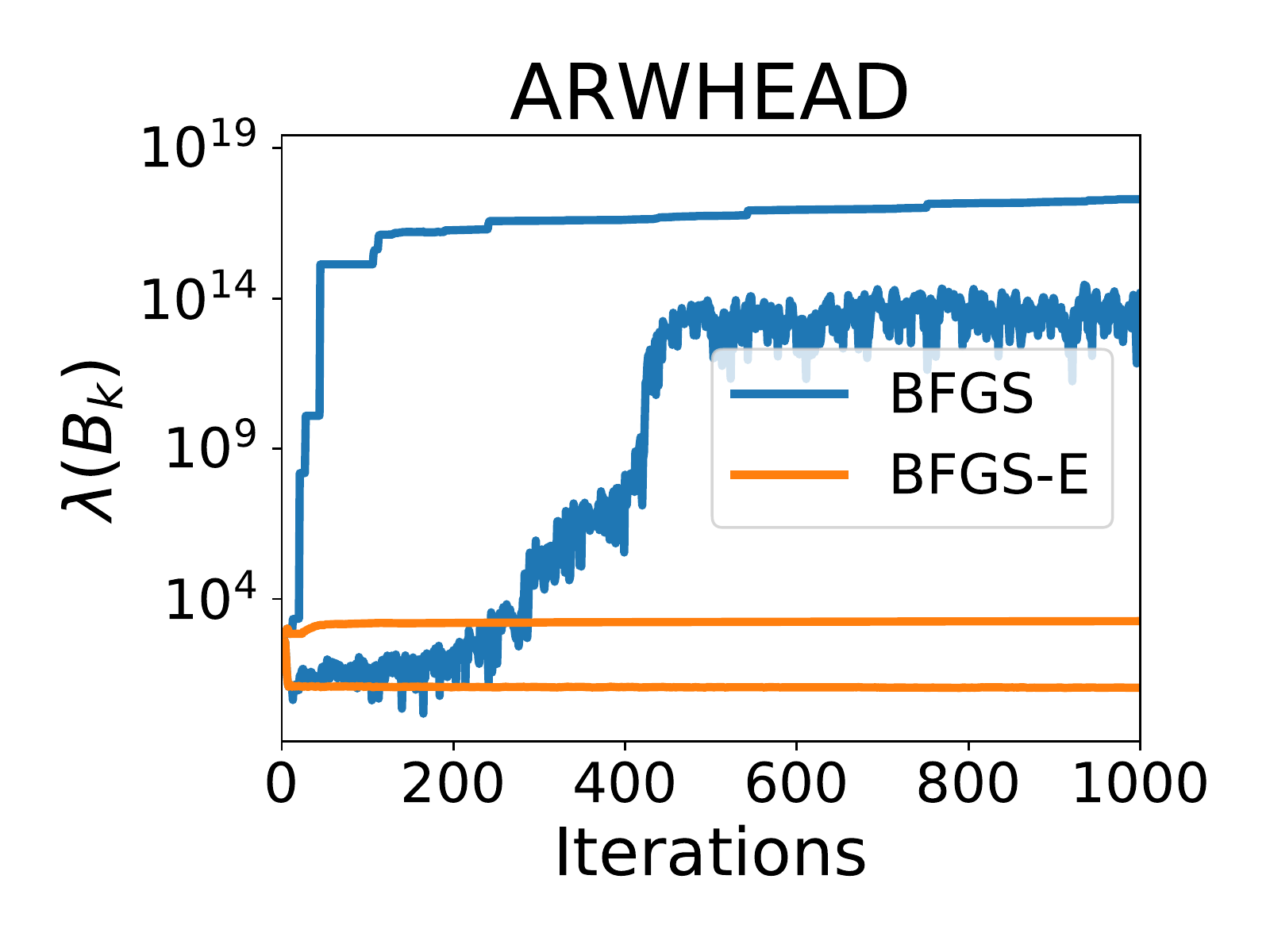}
    \caption{The condition number of the BFGS and BFGS-E matrices $\kappa(H_k)$ against the number of iterations (left) and the smallest and largest eigenvalues of $B_k$ against the number of iterations (right) on the \texttt{ARWHEAD} problem. The final norm of the true gradient achieved by BFGS is approximately $1.97\mathrm{e}{-04}$.}
    \label{fig:BFGS-E on noisy ARWHEAD}
\end{figure}
It may seem surprising that in  Figure \ref{fig:BFGS-E on noisy ARWHEAD} the condition number of the BFGS matrix, $\kappa(H_k)$, decreases after having increased sharply. This can be explained by noting that as the iterates enter into the noisy regime, the difference in the gradient $y_k$ can be corrupted by noise, and we may have $s_k^T y_k \gg s_k^T \tilde{y}_k$. Thus, some of the eigenvalues of the BFGS matrix $B_k = H_k^{-1}$ will increase.  As the iteration proceeds, the rest of the eigenvalues become large too, hence decreasing the condition number.

Figure \ref{fig:obj for ARWHEAD} plots the optimality gap $\phi(x) - \phi^*$ vs the number of gradient evaluations performed for the four methods on the \texttt{ARWHEAD} problem. BFGS and L-BFGS do not achieve as high accuracy in the solution as their noise-tolerant counterparts because the deterioration in the Hessian approximation leads, at some point, to the generation of very small steps that severely limit the decrease in the objective function. The behavior of the methods on this problem is typical of what we have observed. In particular BFGS-E and L-BFGS-E  trigger lengthening of the curvature pairs prior to the point where BFGS and L-BFGS stagnate due to noise. This indicates that the lengthening procedure stabilizes the Hessian approximation prior to reaching this neighborhood.

\begin{figure}[htp]
    \centering
    \includegraphics[width=0.32\linewidth]{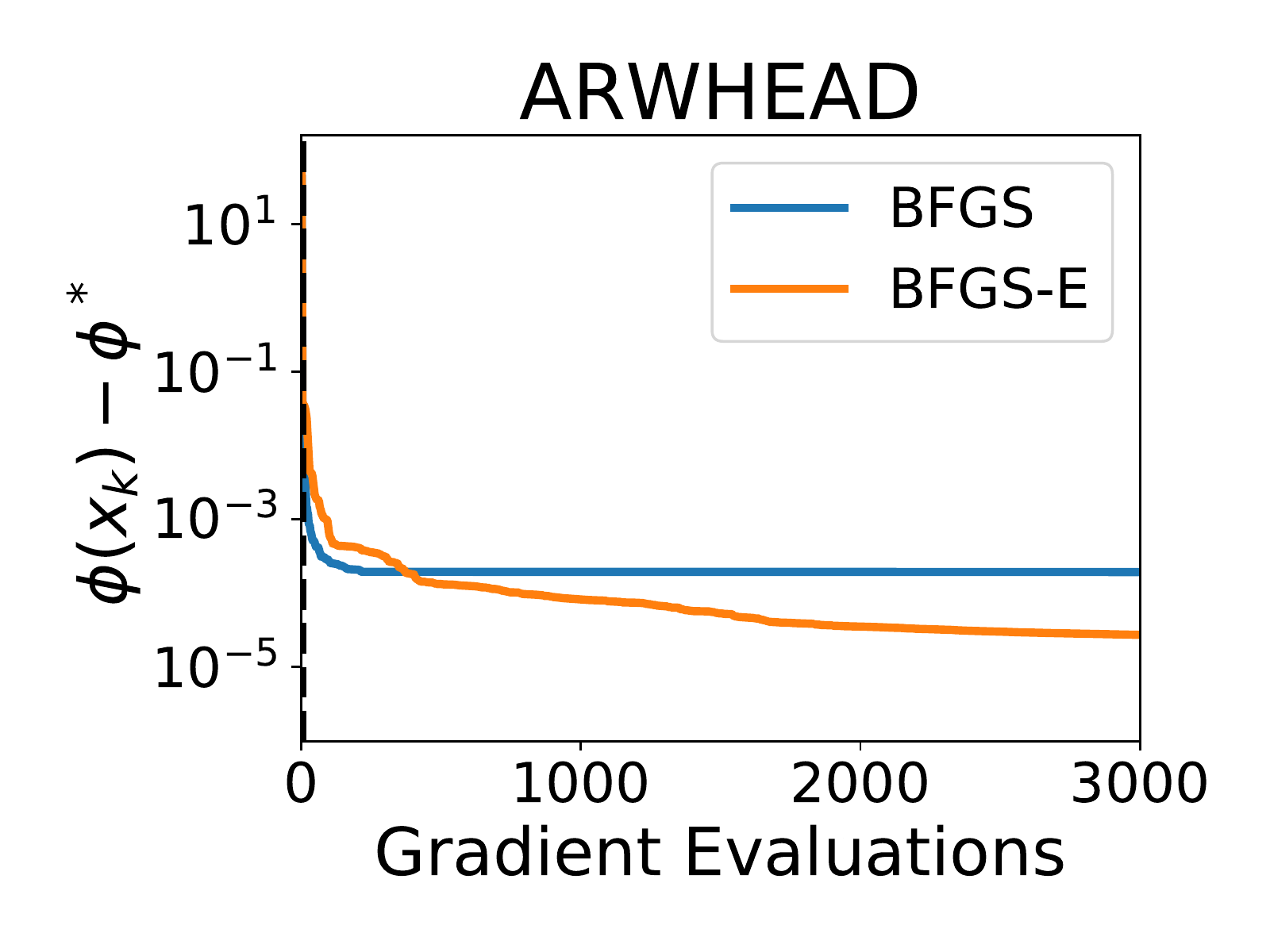}
    \includegraphics[width=0.32\linewidth]{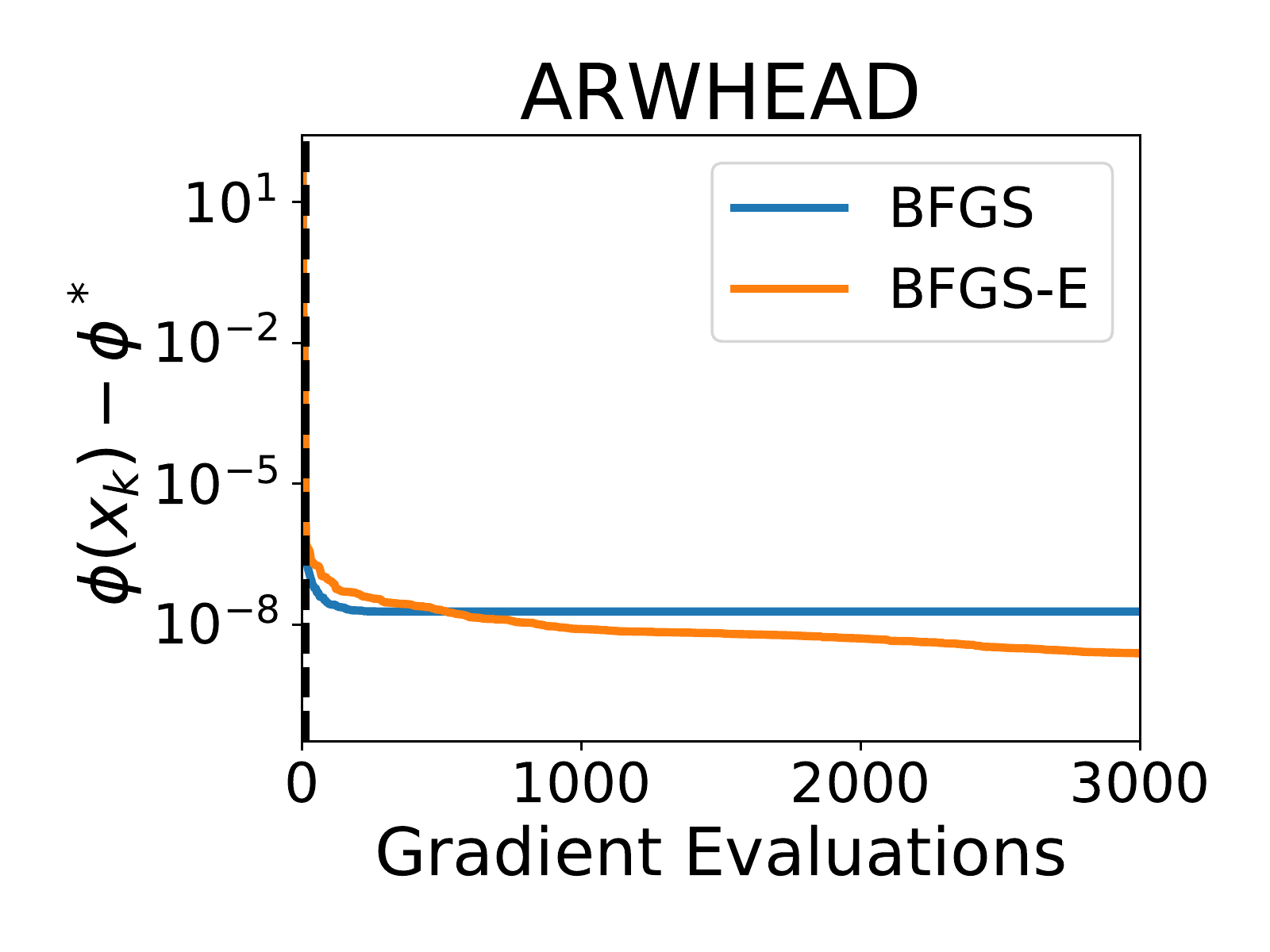}
    \includegraphics[width=0.32\linewidth]{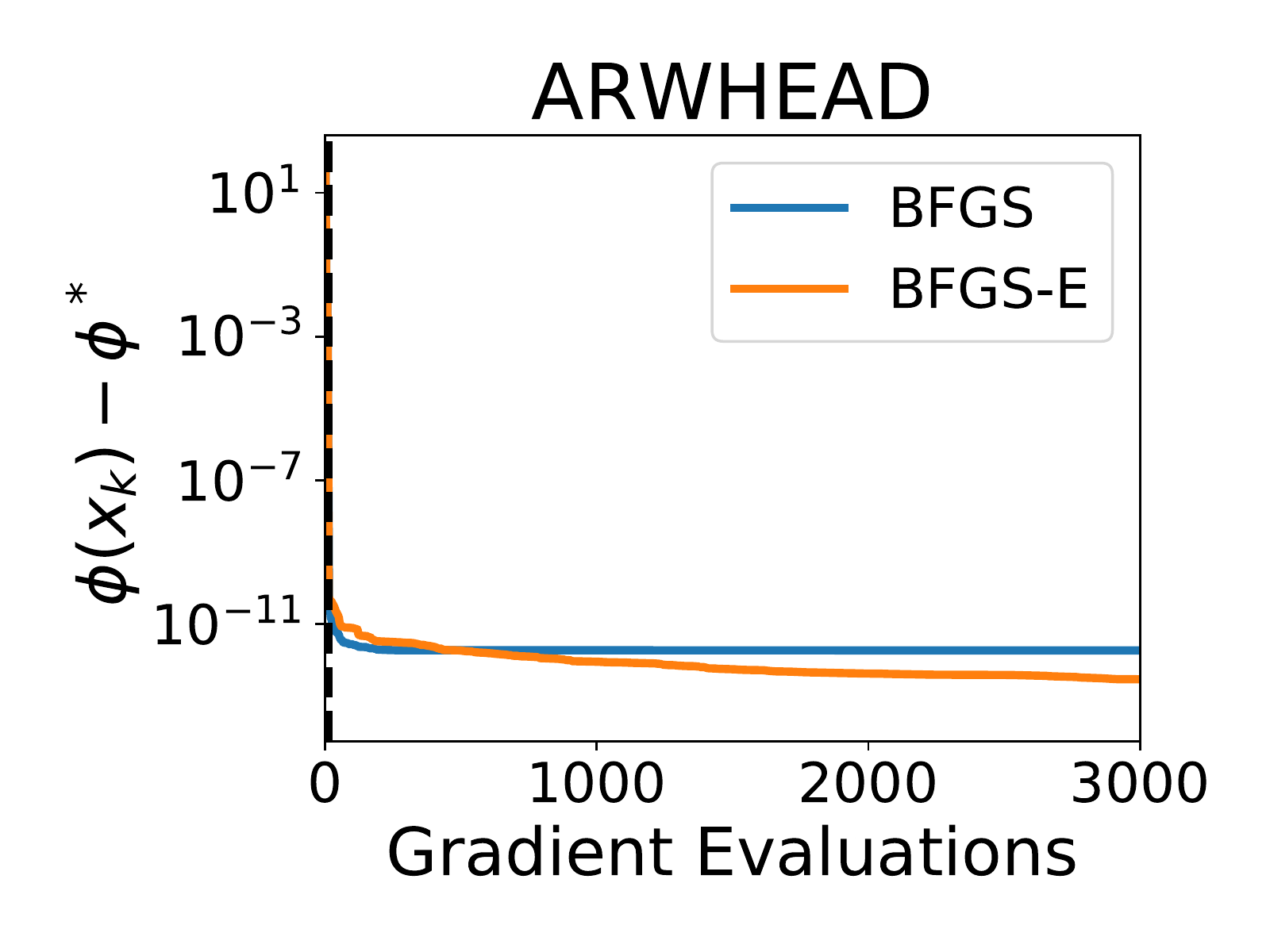}
    \includegraphics[width=0.32\linewidth]{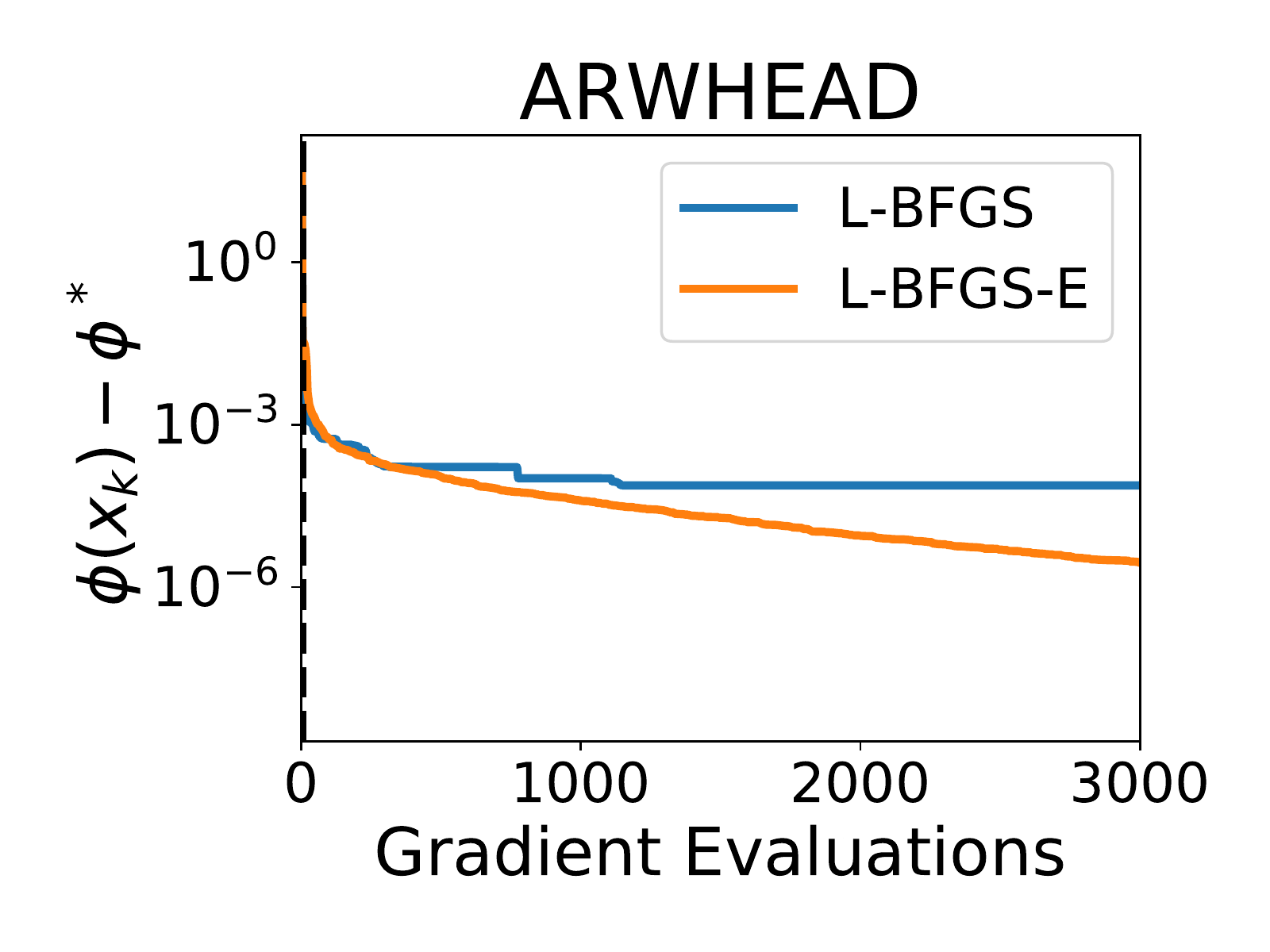}
    \includegraphics[width=0.32\linewidth]{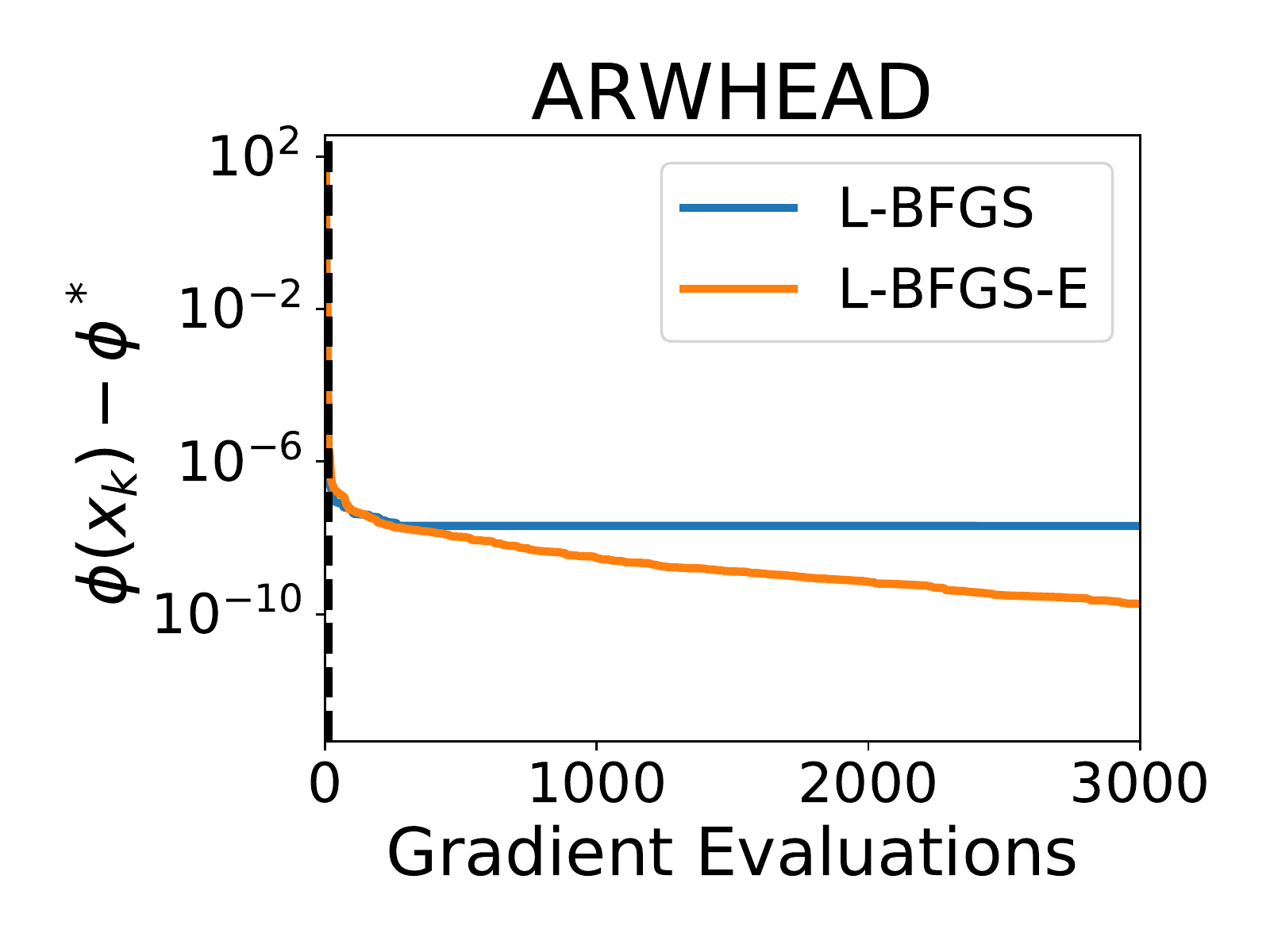}
    \includegraphics[width=0.32\linewidth]{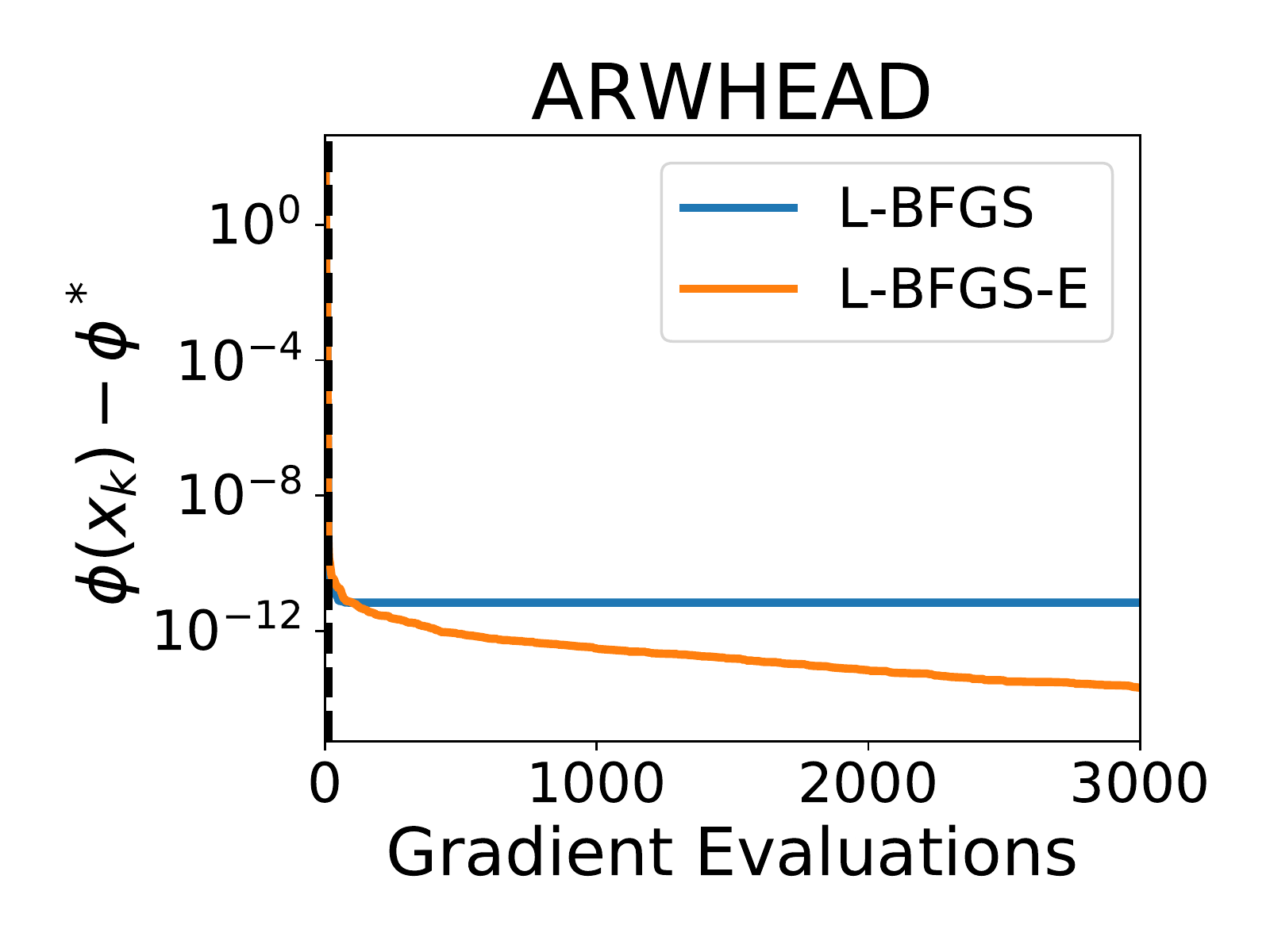}
    \caption{The true optimality gap $\phi(x_k) - \phi^*$ against the number of gradient evaluations on the \texttt{ARWHEAD} problem for $\epsilon_f=0$, and for the following gradient noise levels:  $\xi_g = 10^{-1}$ (left), $10^{-3}$ (middle), and $10^{-5}$ (right). The black dashed line denotes the iteration before the split phase becomes active.}
    \label{fig:obj for ARWHEAD}
\end{figure}

The lengthening procedure in our noise-tolerant algorithms comes at an additional computational cost. Figure~\ref{fig:grad evals for ARWHEAD} plots the cumulative number of gradient evaluations against the iteration count for the \texttt{ARWHEAD} problem. We observe that for BFGS or L-BFGS, the cumulative number of gradient evaluations is approximately equal to the number of iterations.
For the noise-tolerant methods, the number of gradient evaluations match the standard BFGS and L-BFGS methods until the split phase activates. Upon entering the split phase, we notice that the cost of each iteration is approximately $2-4$ gradient evaluations. This may be explained by the additional $1-3$ gradient evaluations necessary to find the appropriate $\beta_k$ that satisfies both the noise and Wolfe conditions, plus one gradient evaluation for triggering the split phase. This cost is  worthwhile in that it allows the algorithm to make progress in the noisy regime.

\begin{figure}
    \centering
    \includegraphics[width=0.45\linewidth]{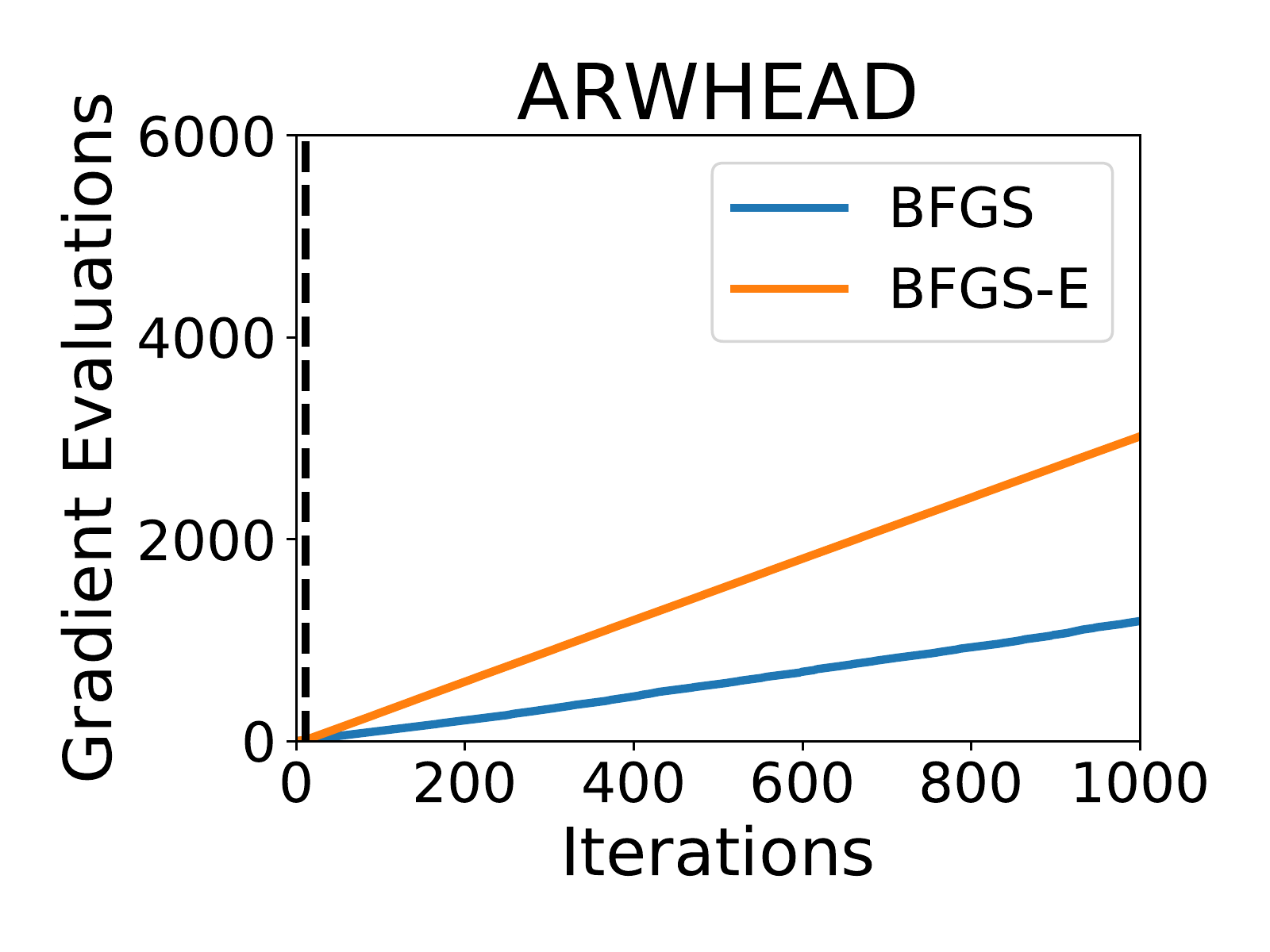}
    \includegraphics[width=0.45\linewidth]{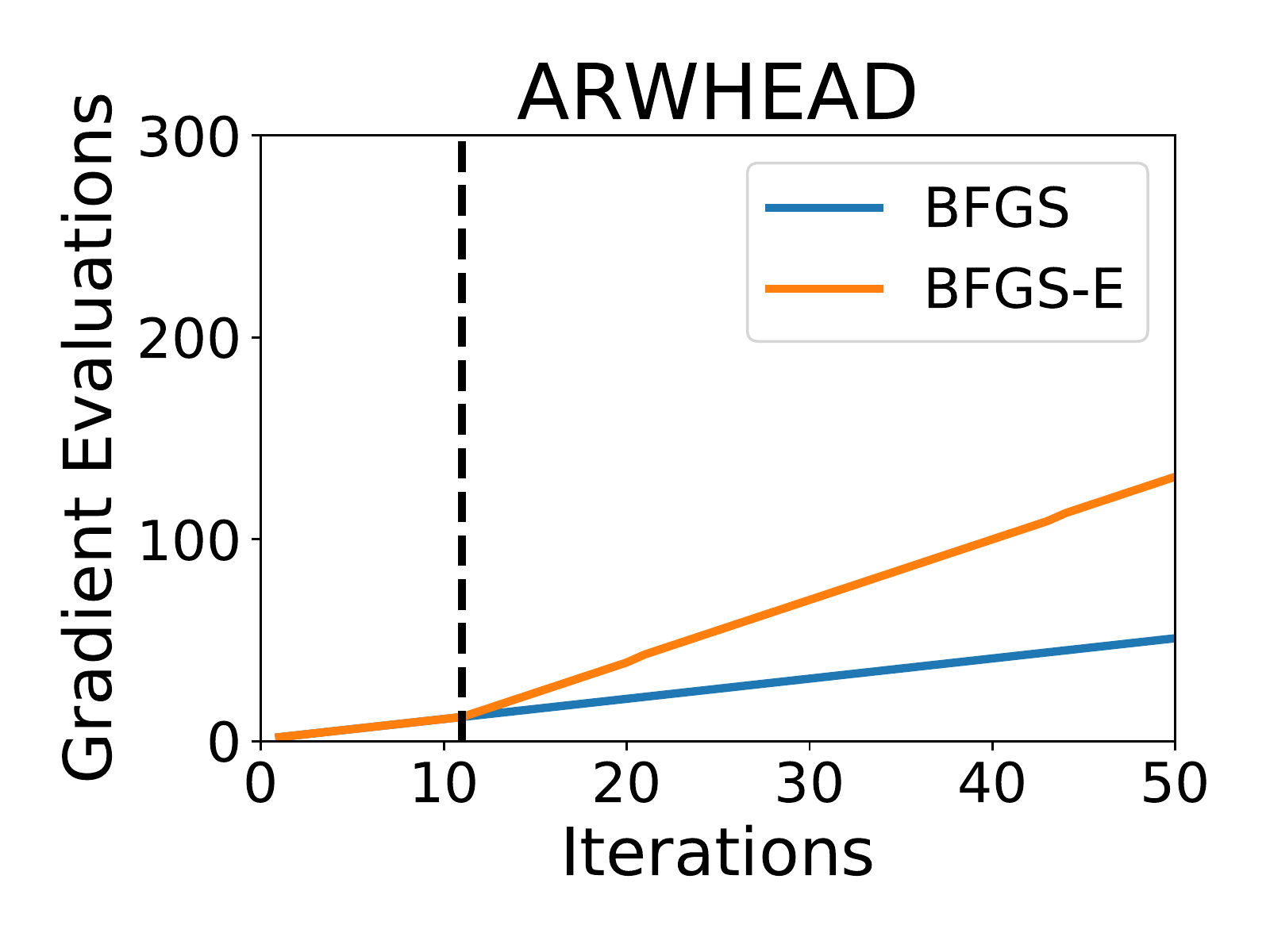}
    \caption{Cumulative number of gradient evaluations against the iteration count on the \texttt{ARWHEAD} problem for $\epsilon_f=0$ and $\xi_g = 10^{-3}$ for BFGS and BFGS-E. The left figure plots the long-term behavior and the right figure plots the short-term behavior. The results for L-BFGS and L-BFGS-E as well as different noise levels are similar. The black dashed line denotes the iteration before the split phase becomes active.}
    \label{fig:grad evals for ARWHEAD}
\end{figure}

\subsubsection{Sensitivity  with respect to $\epsilon_g$}

Since the bound on the gradient error $\epsilon_g$ may be estimated by an external procedure, it is possible for $\epsilon_g$ to be input incorrectly. In order to investigate the sensitivity of the choice of $\epsilon_g$, we consider both under- and overestimation of it. We perform the same experiment with a fixed $\xi_g = 10^{-3}$ and $\epsilon_f = 0$ but provide the algorithm an incorrect $\bar{\epsilon}_g = \omega \epsilon_g$ where $\omega \in \{\frac{1}{10}, \frac{1}{5}, \frac{1}{2}, 1, 2, 5, 10\}$. This is shown in Figure \ref{fig:sensitivity}. We plot only BFGS-E since L-BFGS-E performs similarly.

\begin{figure}
    \centering
    \includegraphics[width=0.32\textwidth]{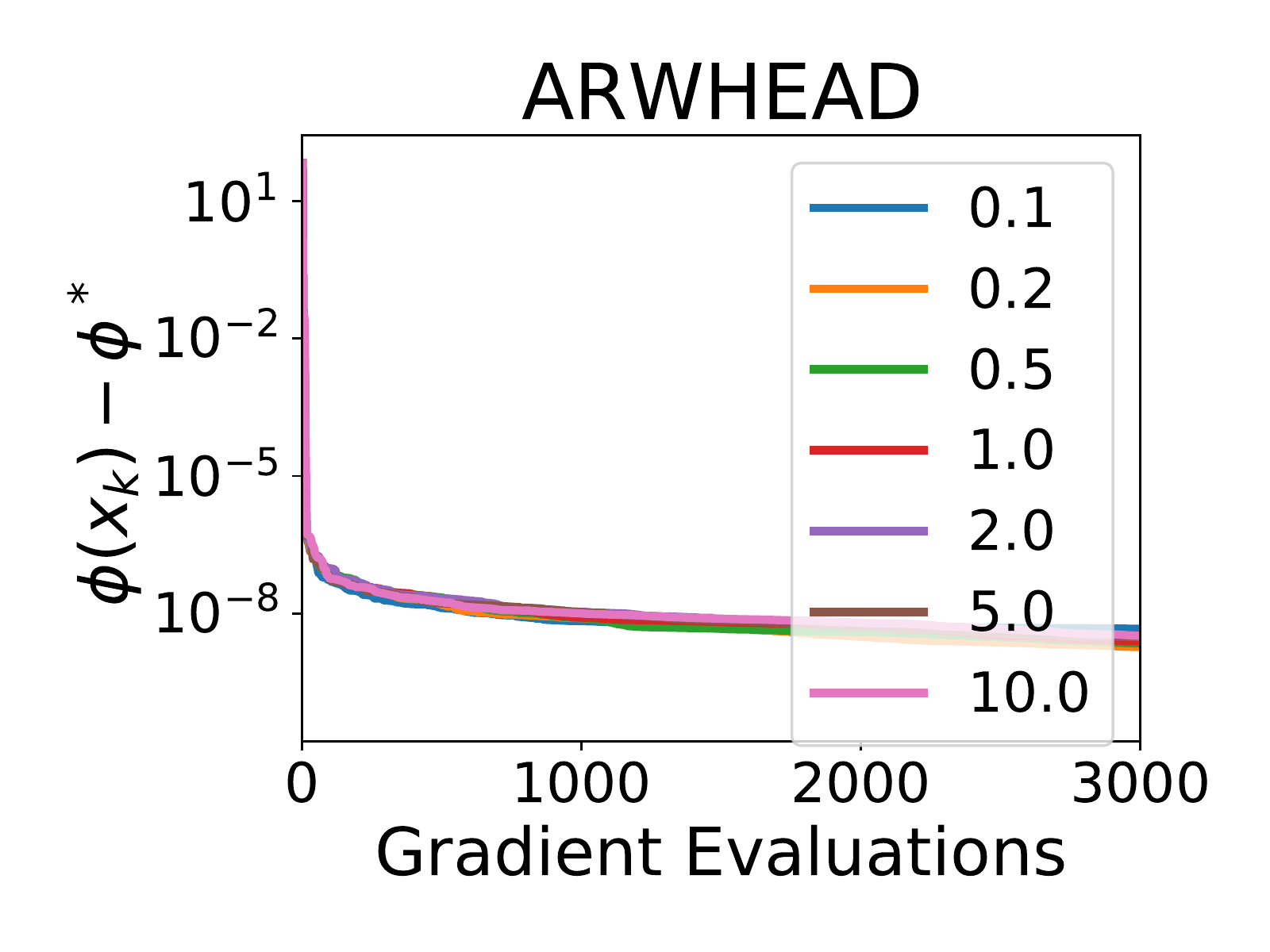}
    \includegraphics[width=0.32\textwidth]{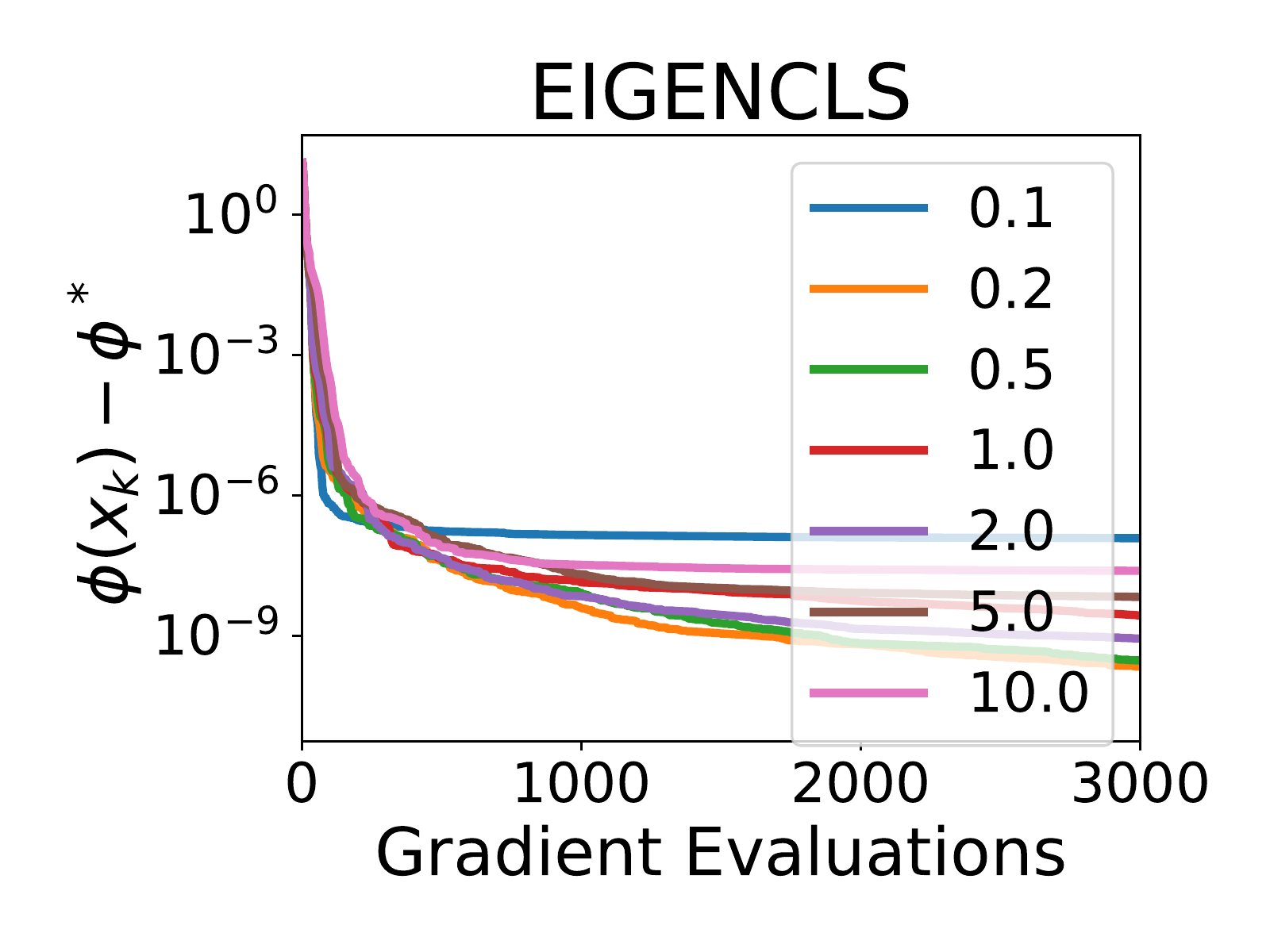}
    \includegraphics[width=0.32\textwidth]{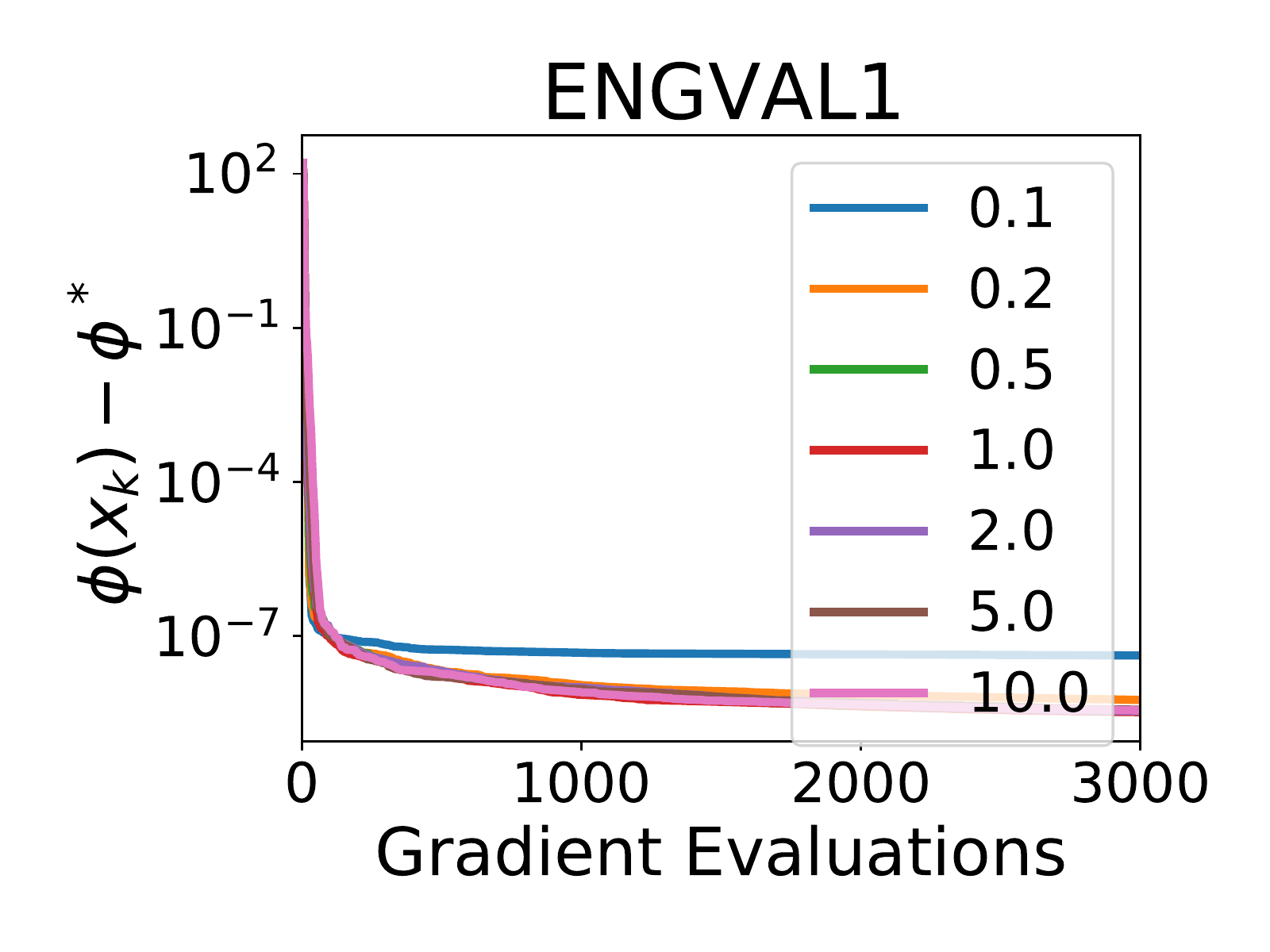}
    \caption{The true optimality gap $\phi(x_k) - \phi^*$ against the number of gradient evaluations applying BFGS-E on the \texttt{ARWHEAD}, \texttt{EIGENCLS}, and \texttt{ENGVAL1} problems for $\epsilon_f = 0$ and $\xi_g = 10^{-3}$ with incorrectly input $\bar{\epsilon}_g = \omega \epsilon_g$ for $\omega \in \{\frac{1}{10}, \frac{1}{5}, \frac{1}{2}, 1, 2, 5, 10\}$.}
    \label{fig:sensitivity}
\end{figure}

If the noise is severely underestimated, it can lead to early stagnation of the algorithm due to corruption of the BFGS matrix. If the noise is severely overestimated, then the collection of non-local curvature information can result in slower progress towards the solution. Overall, the method tolerates overestimation better than underestimation of the noise level, as one would expect.


\subsection{Experiments with Intermittent Noise in the Gradient}  \label{onlyi}

In some applications, the noise level in the gradient evaluation may fluctuate rather than remain constant. One special case is that of intermittent noise. 
%
To simulate it, we will set $\xi_f=0$ and let the noise level in the gradient $\xi_g$ alternate between 0 and a fixed nonzero value every $N_{\text{noise}}$ iterations, where $N_{\text{noise}} \in \{10, 25, 50\}$. We show representative results using the \texttt{CRAGGLVY} problem in Figure \ref{fig:intermit CRAGGLVY (no skip)}. The \texttt{CRAGGLVY} problem is chosen because it requires more than $N_{\text{noise}}$ iterations to solve, whereas the \texttt{ARWHEAD} problem can be solved in under 25 iterations. The noise-tolerant  methods are provided the value of $\epsilon_g$ but not $N_{\text{noise}}$.

BFGS suffers the most from the inclusion of intermittent noise, and is unable to recover quickly enough to make progress even when there is no noise in the gradient. In contrast, BFGS-E is able to continue to make progress immediately once noise is diminished since the BFGS matrix $H_k$ is less corrupted by noise and therefore able to take advantage of the non-noisy gradient; see in particular the stepwise behavior on the top right plot in Figure \ref{fig:intermit CRAGGLVY (no skip)}. L-BFGS-E performs even better than BFGS-E, but note that standard L-BFGS is quite effective when the noise toggles every $N_{\text{noise}}=25$ or 50 iterations. This is because, if the number $N_{\text{noise}}$ of non-noisy iterations is larger than the memory $t = 10$, L-BFGS is able to forget all noise-contaminated curvature pairs, then it is able to recover and make progress. 

\begin{figure}[htp]
    \centering
    \includegraphics[width=0.32\linewidth]{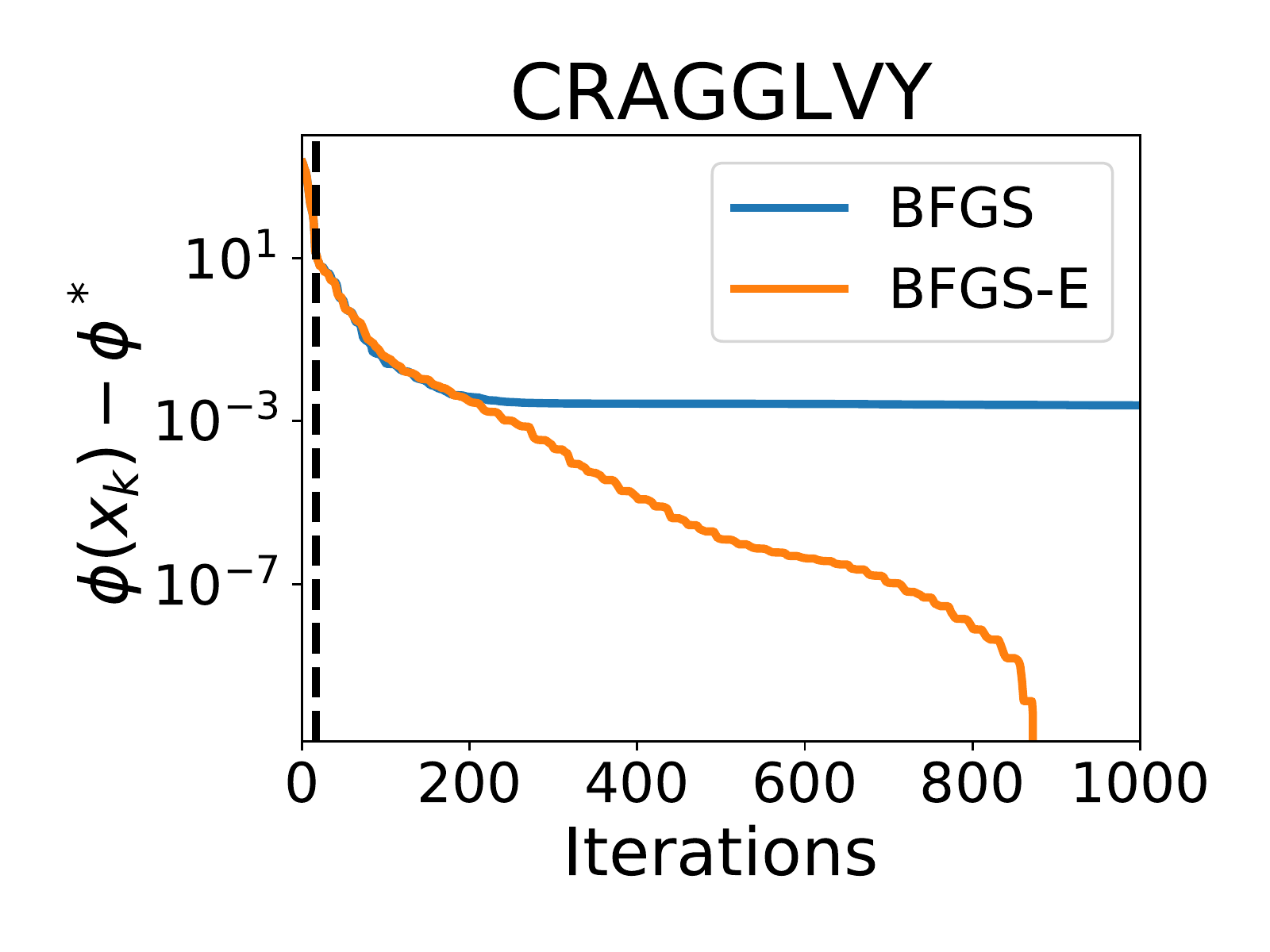}
    \includegraphics[width=0.32\linewidth]{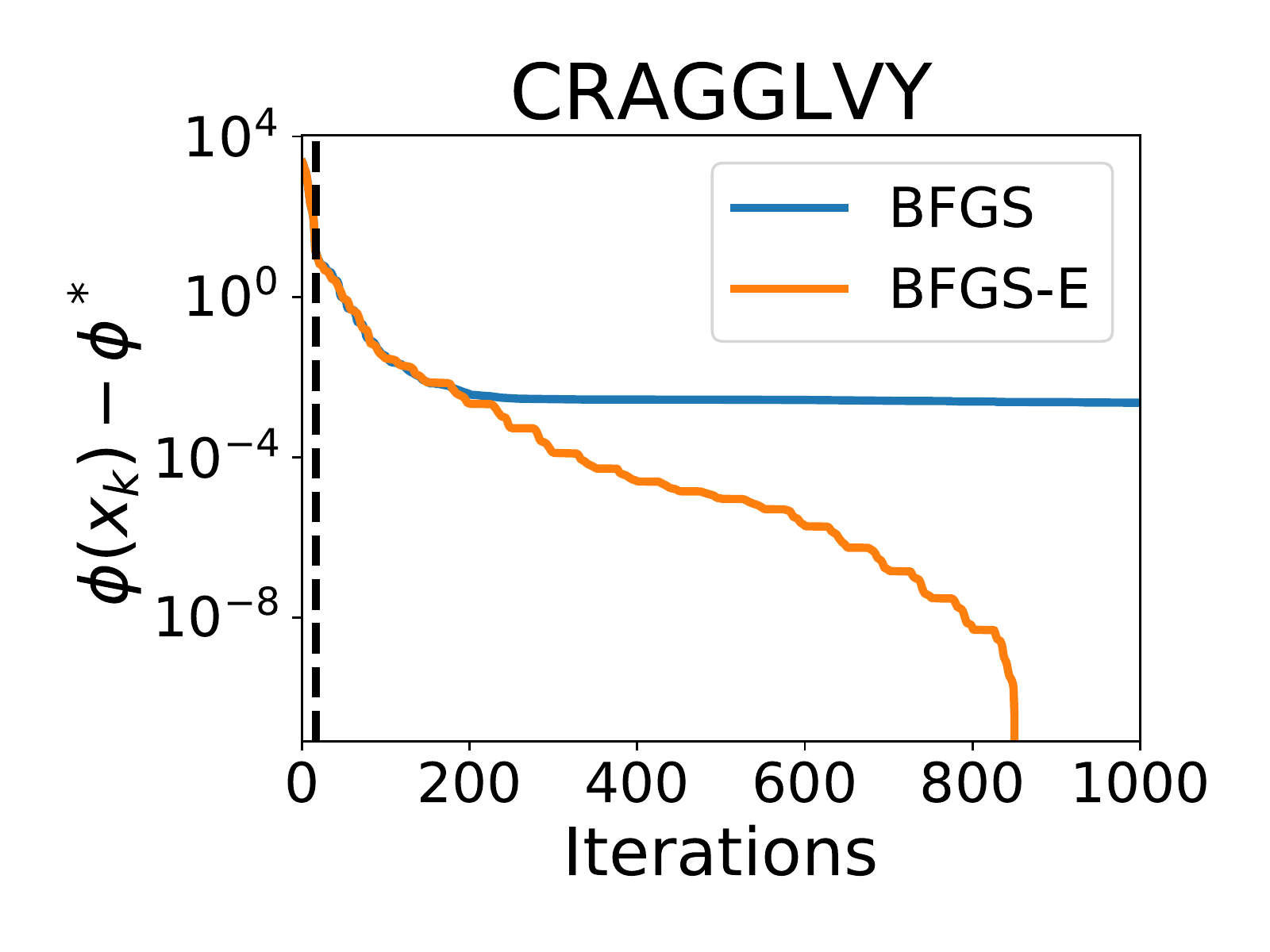}
    \includegraphics[width=0.32\linewidth]{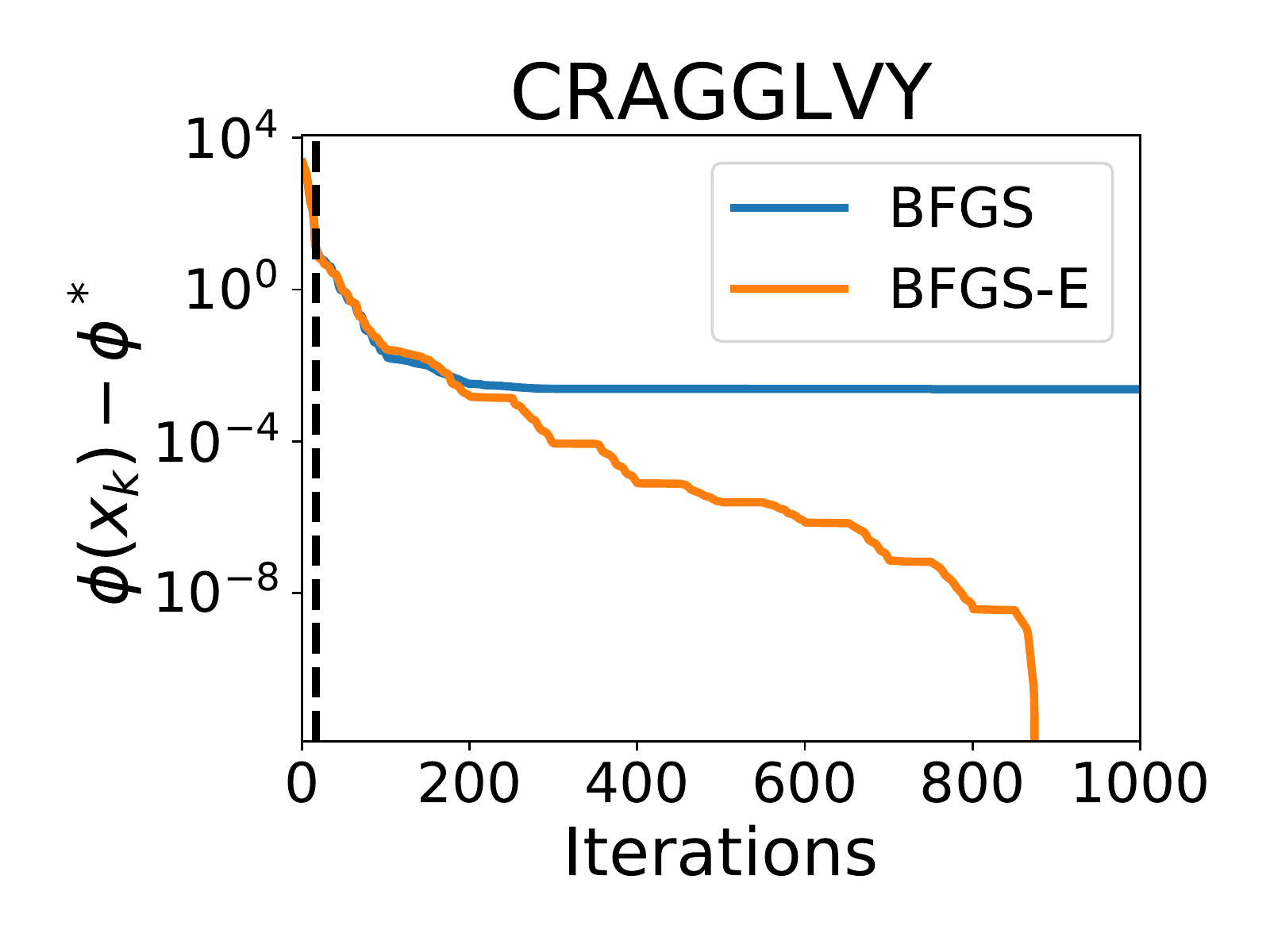}
    \includegraphics[width=0.32\linewidth]{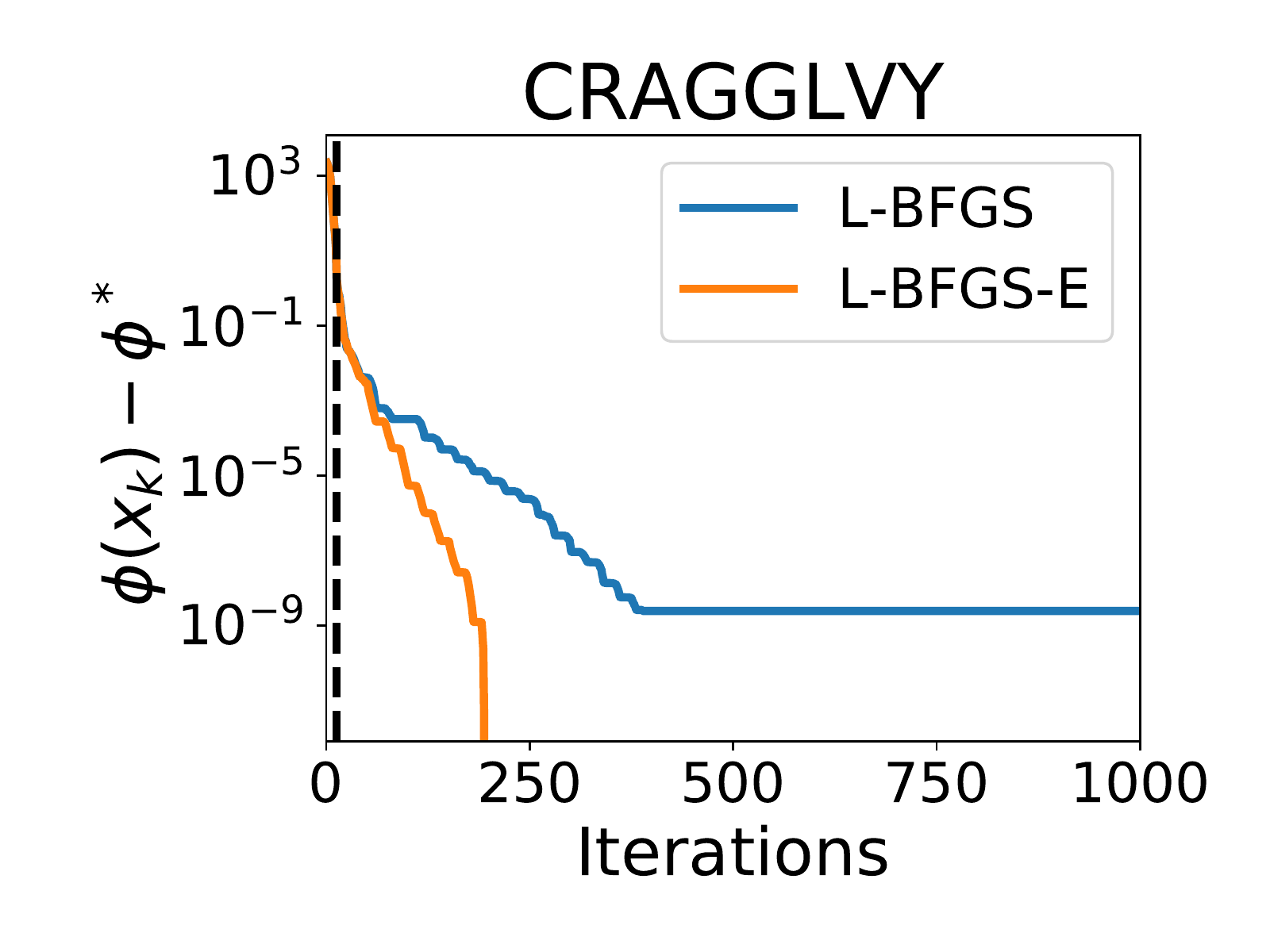}
    \includegraphics[width=0.32\linewidth]{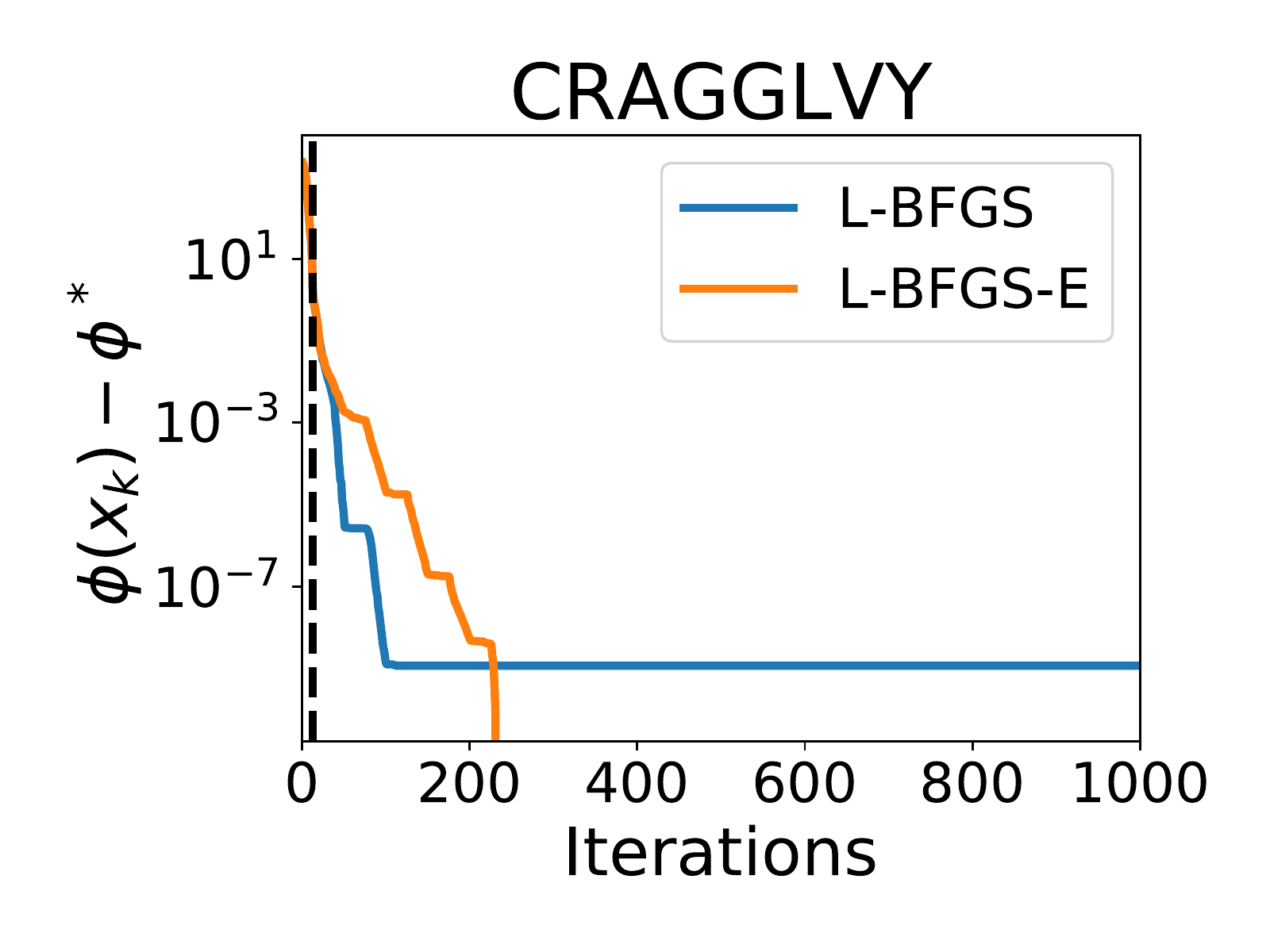}
    \includegraphics[width=0.32\linewidth]{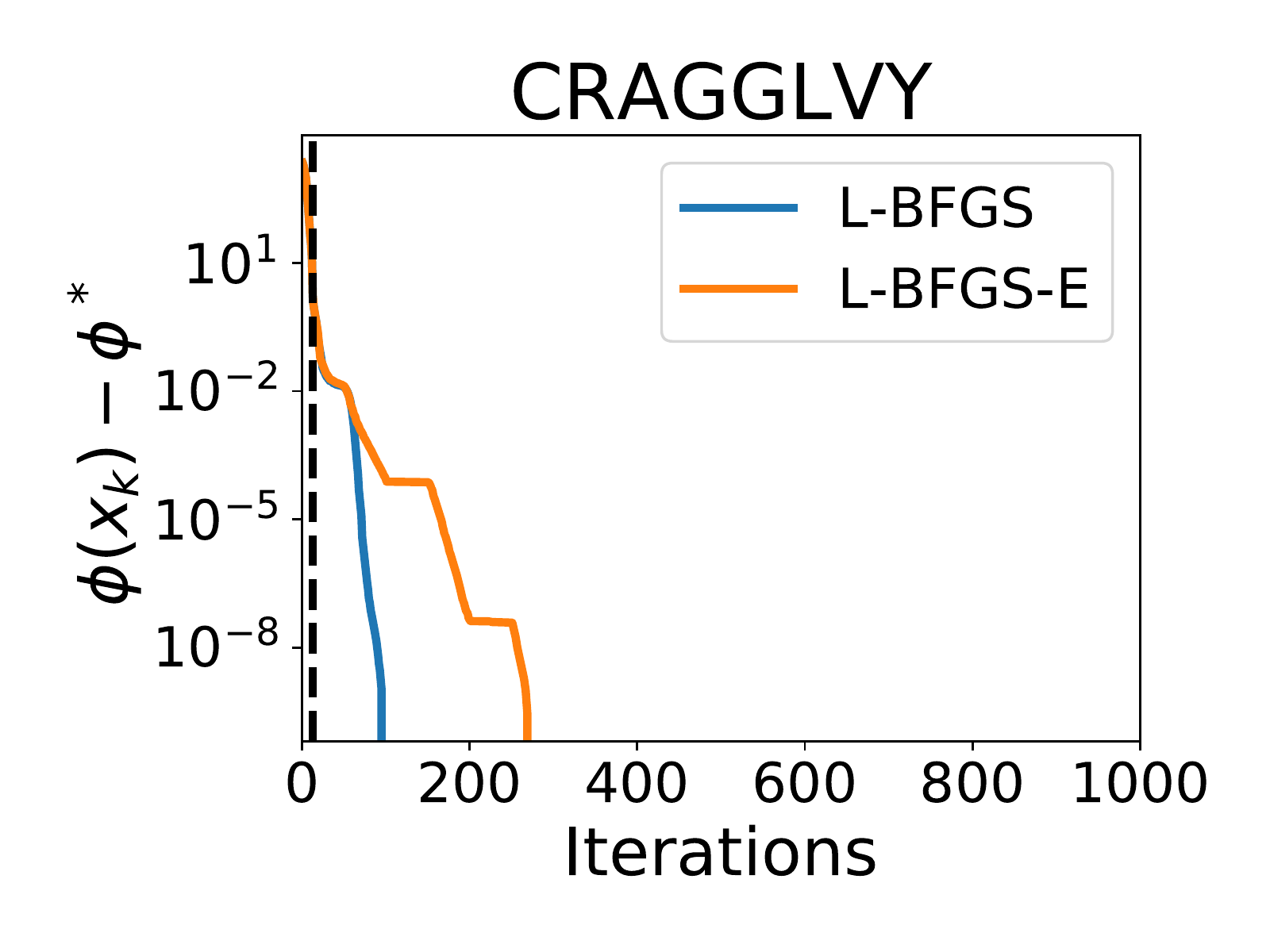}
    \caption{Intermittent Noise. Optimality gap $\phi(x_k) - \phi^*$ against the number of iterations on the \texttt{CRAGGLVY} problem. $\xi_f=0$ and $\xi_g$ alternates between 0 and with $\xi_g = 10^{-1}$ every $N_{\text{noise}}$ iterations. Results for $N_{\text{noise}} = 10$ (left), $25$ (middle), and $50$ (right). The black dashed line denotes the iteration before the split phase becomes active.}
    \label{fig:intermit CRAGGLVY (no skip)}
\end{figure}

\subsection{Comparison Against Methods that Employ Update Skipping}\label{skip}
We now consider the performance of the BFGS (Skips) and L-BFGS (Skips) methods for constant and intermittent noise. The appeal of skipping the update when the quality of the correction pair is not assured  is its economy, since the lengthening procedure involves additional gradient evaluations. 

In Figures \ref{fig:obj for ENGVAL1} and \ref{fig:obj for EIGENCLS}, we compare the performance of BFGS (Skips) and L-BFGS (Skips) to both the standard and extended methods when there is uniform constant noise in the gradient. We report the results for the \texttt{ENGVAL1} and \texttt{EIGENCLS} problems, which are of easy and medium difficulty, respectively. We chose these problems to demonstrate nuanced cases where fixing the BFGS matrix is not sufficient for making fast progress to the solution.

\begin{figure}
    \centering
    \includegraphics[width=0.32\linewidth]{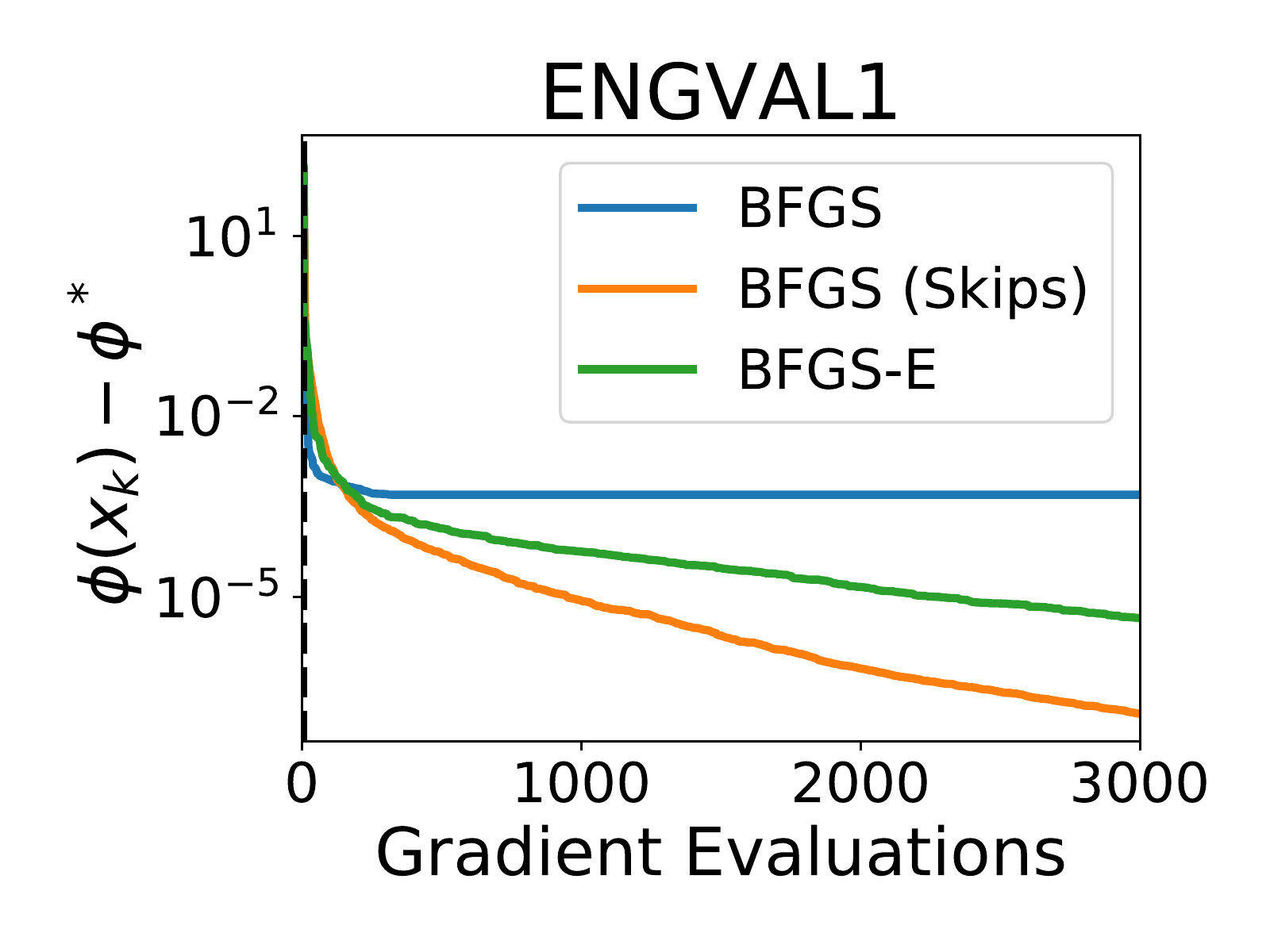}
    \includegraphics[width=0.32\linewidth]{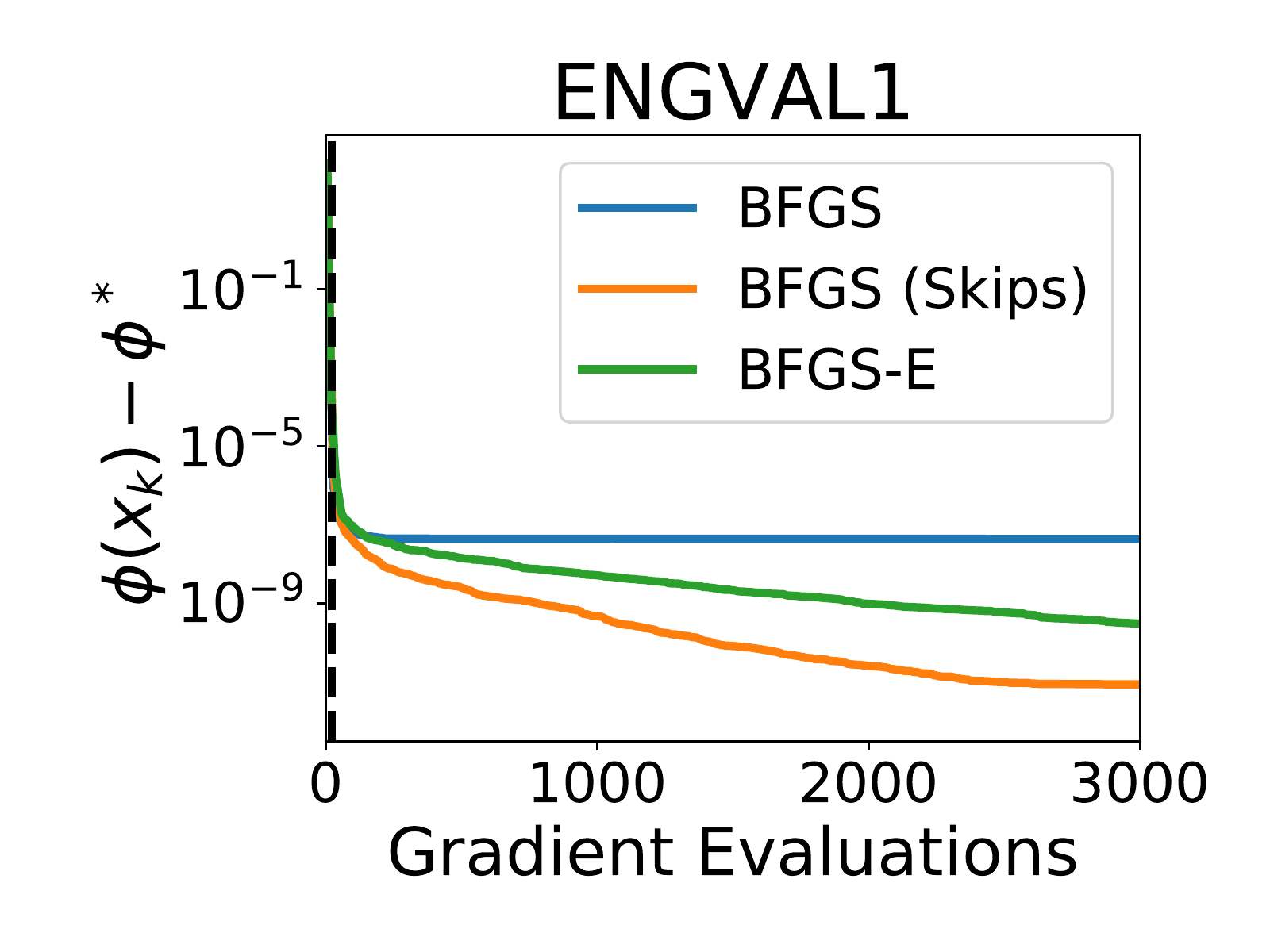}
    \includegraphics[width=0.32\linewidth]{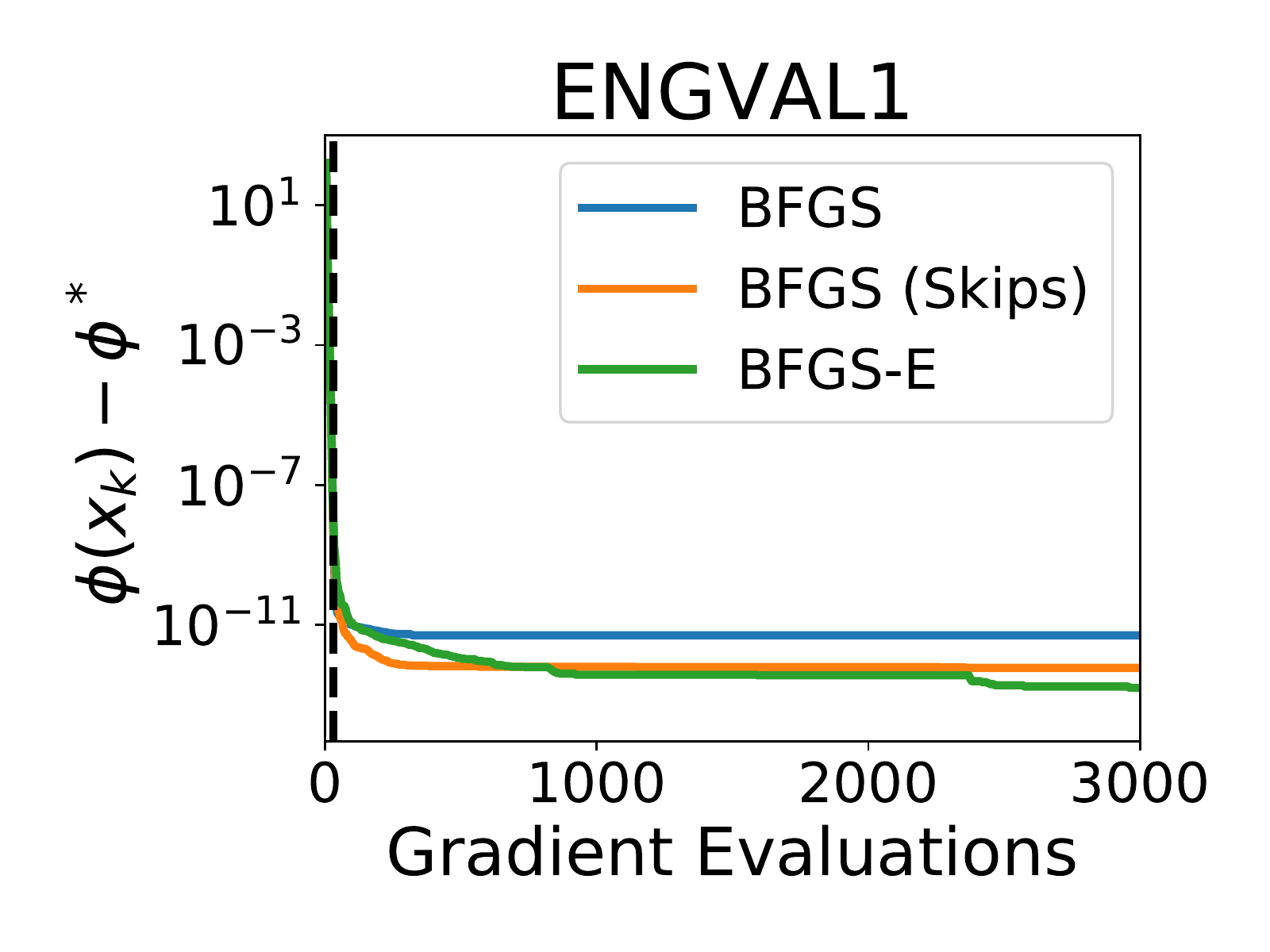}
    \includegraphics[width=0.32\linewidth]{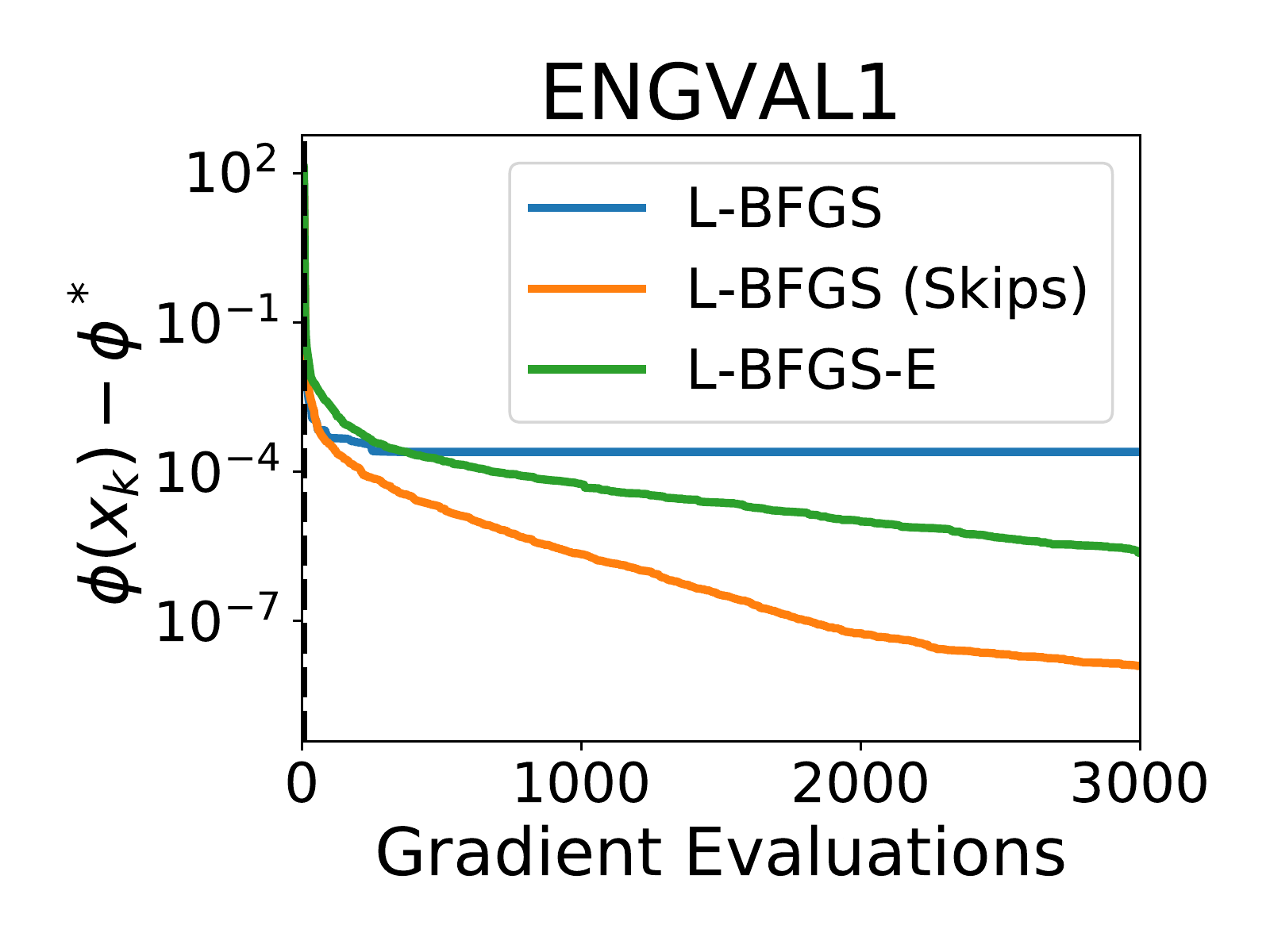}
    \includegraphics[width=0.32\linewidth]{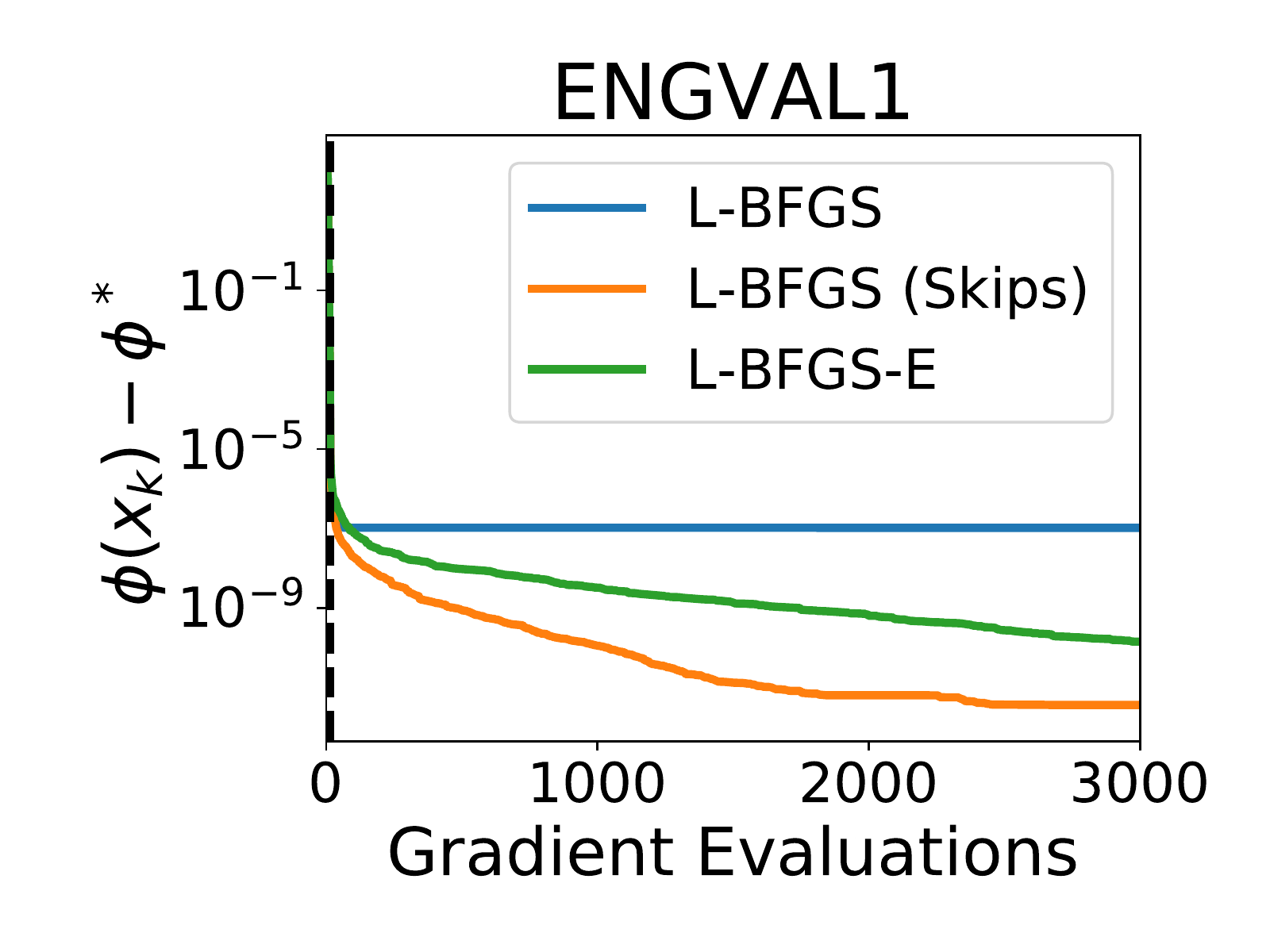}
    \includegraphics[width=0.32\linewidth]{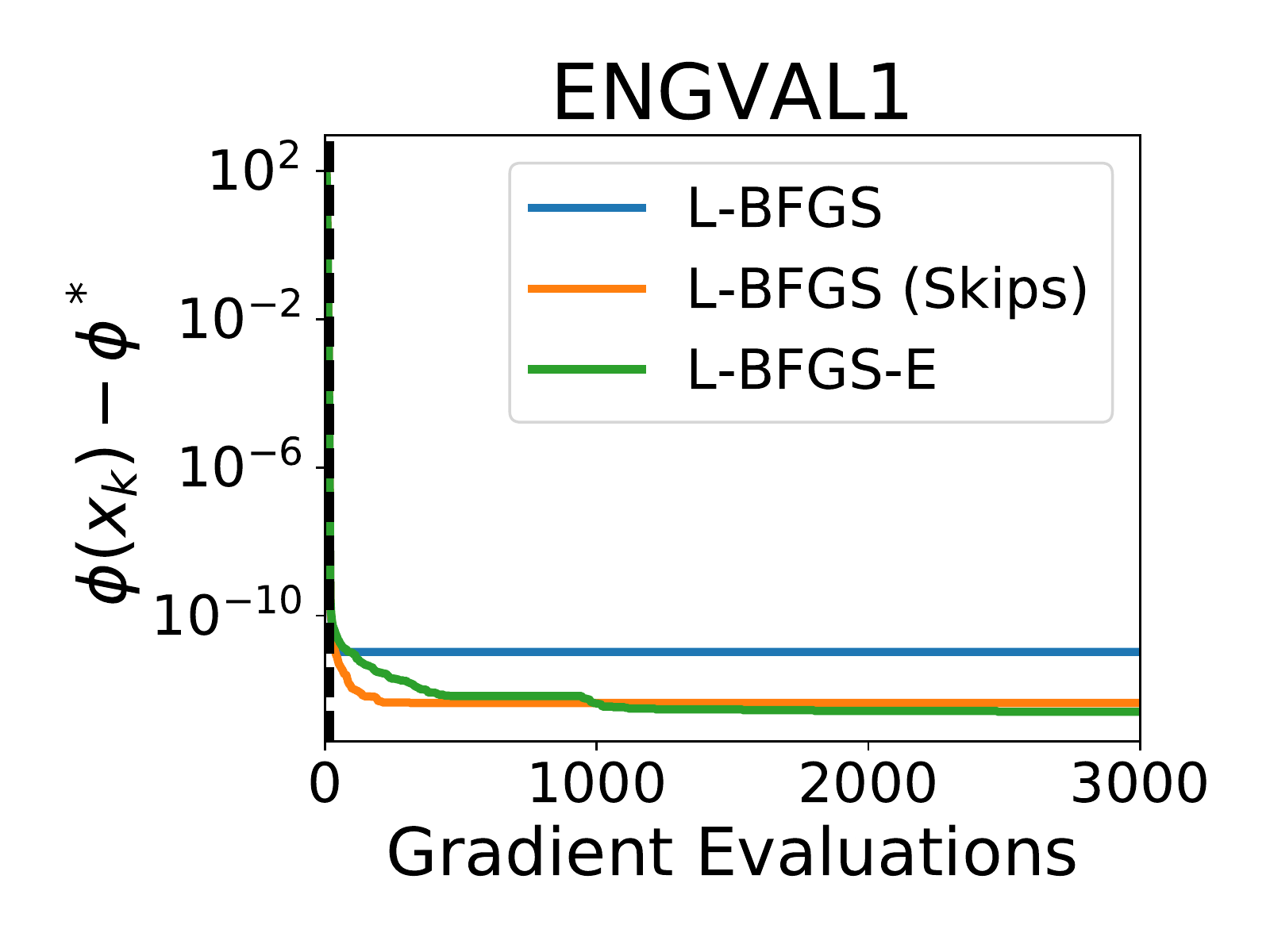}
    \caption{The true optimality gap $\phi(x_k) - \phi^*$ against the number of gradient evaluations on the \texttt{ENGVAL1} problem for $\epsilon_f=0$, and for the following gradient noise levels:  $\xi_g = 10^{-1}$ (left), $10^{-3}$ (middle), and $10^{-5}$ (right). The black dashed line denotes the iteration before the split phase becomes active.}
    \label{fig:obj for ENGVAL1}
\end{figure}

\begin{figure}[htp]
    \centering
    \includegraphics[width=0.32\linewidth]{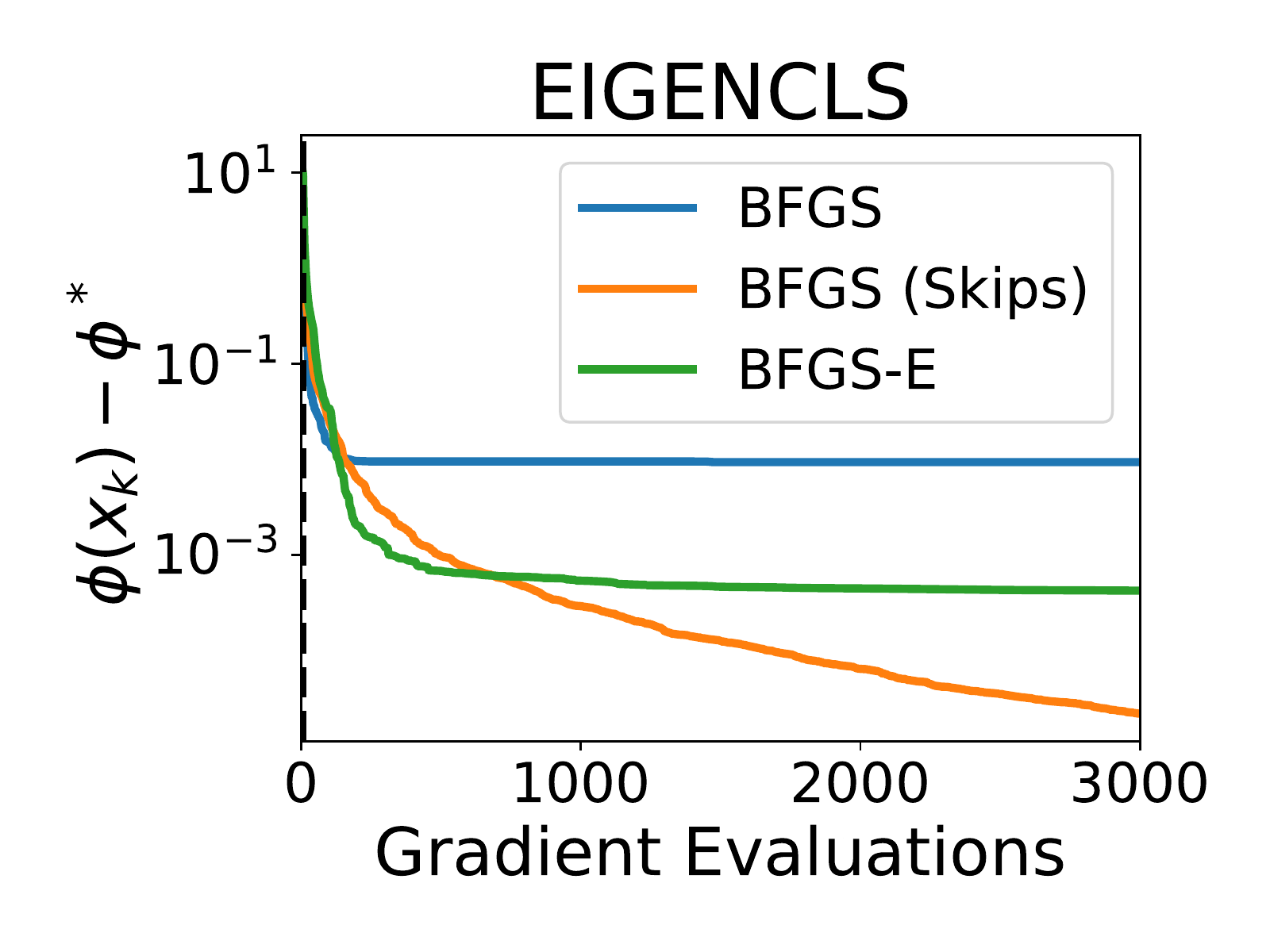}
    \includegraphics[width=0.32\linewidth]{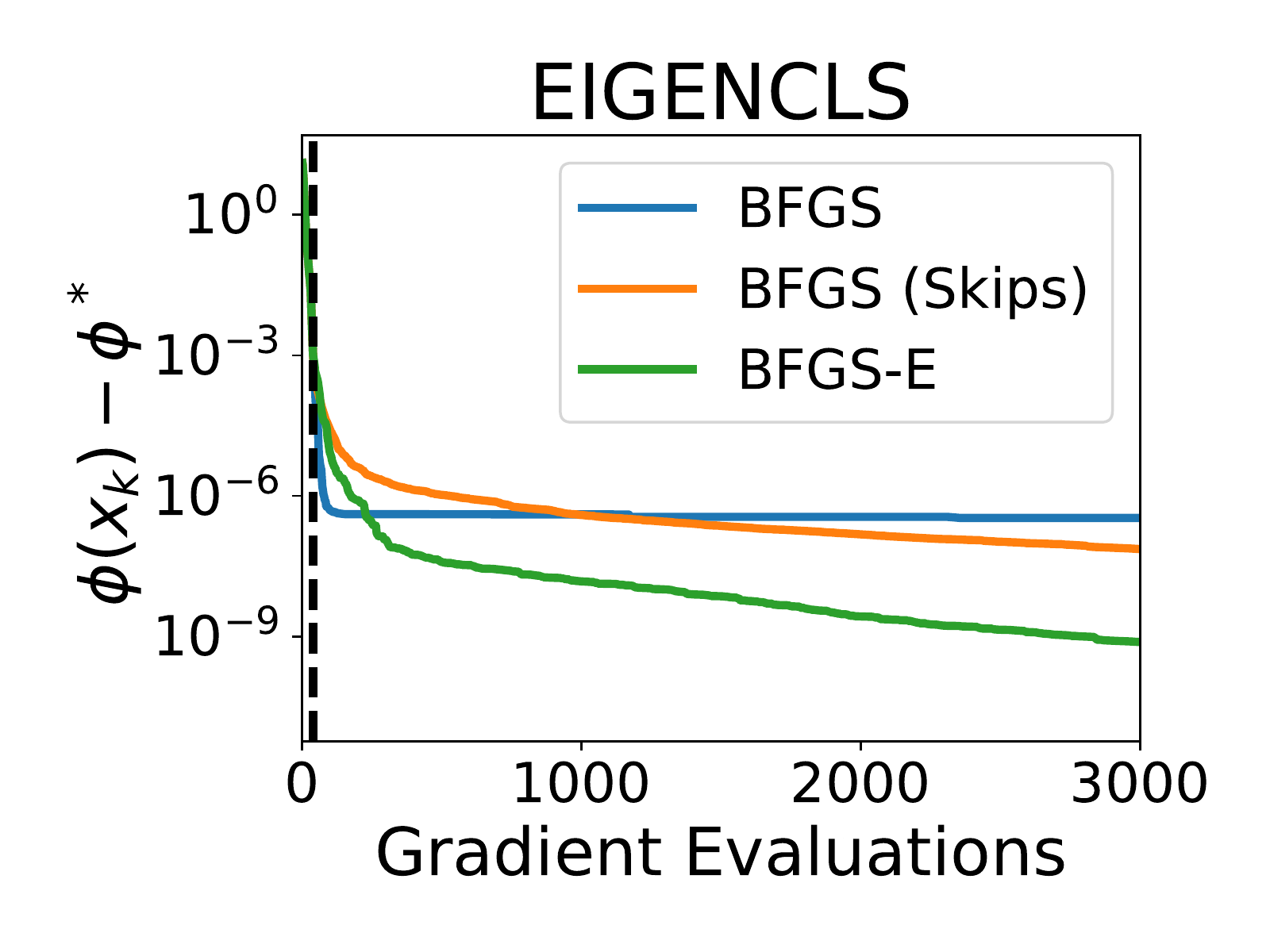}
    \includegraphics[width=0.32\linewidth]{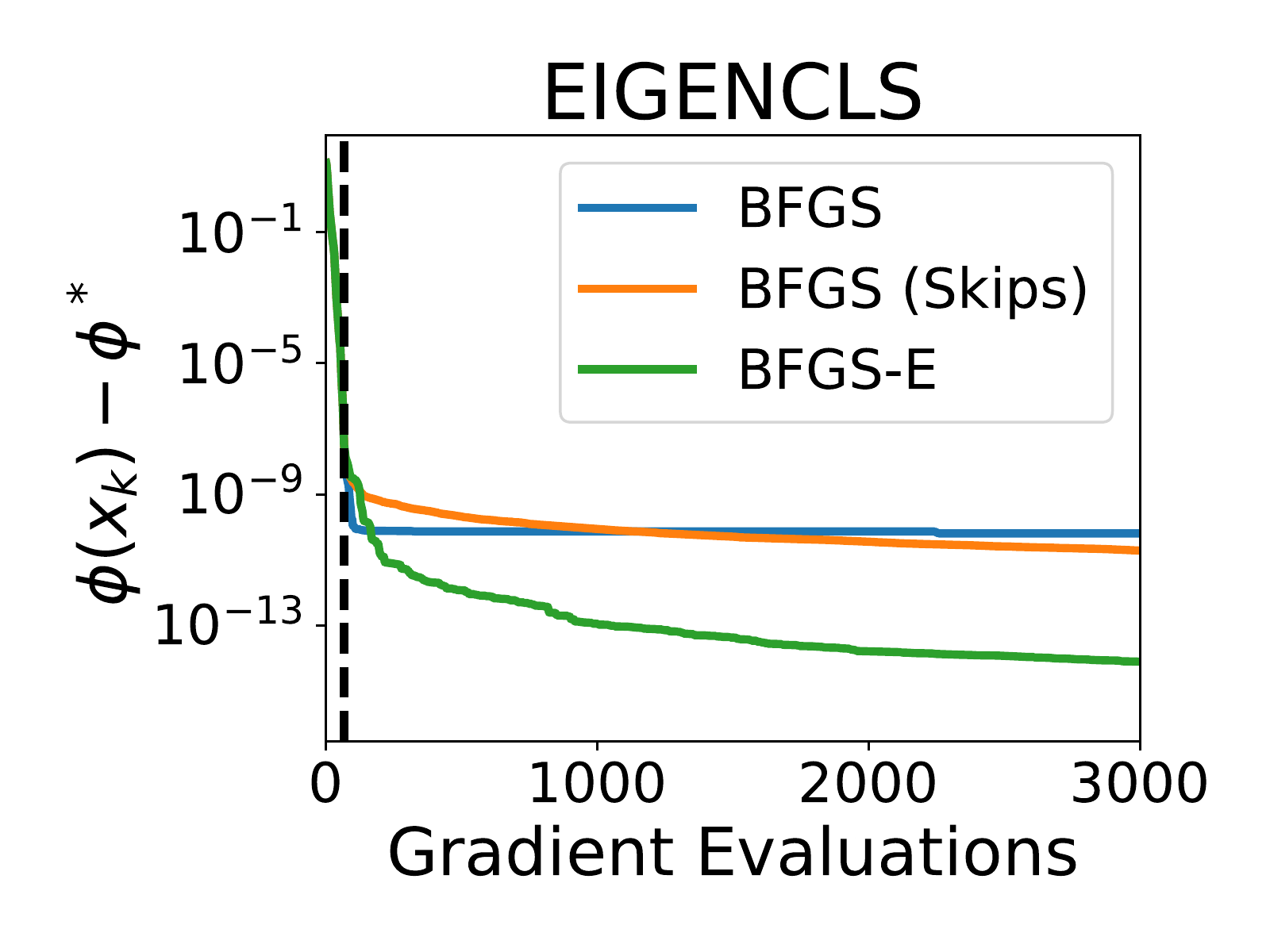}
    \includegraphics[width=0.32\linewidth]{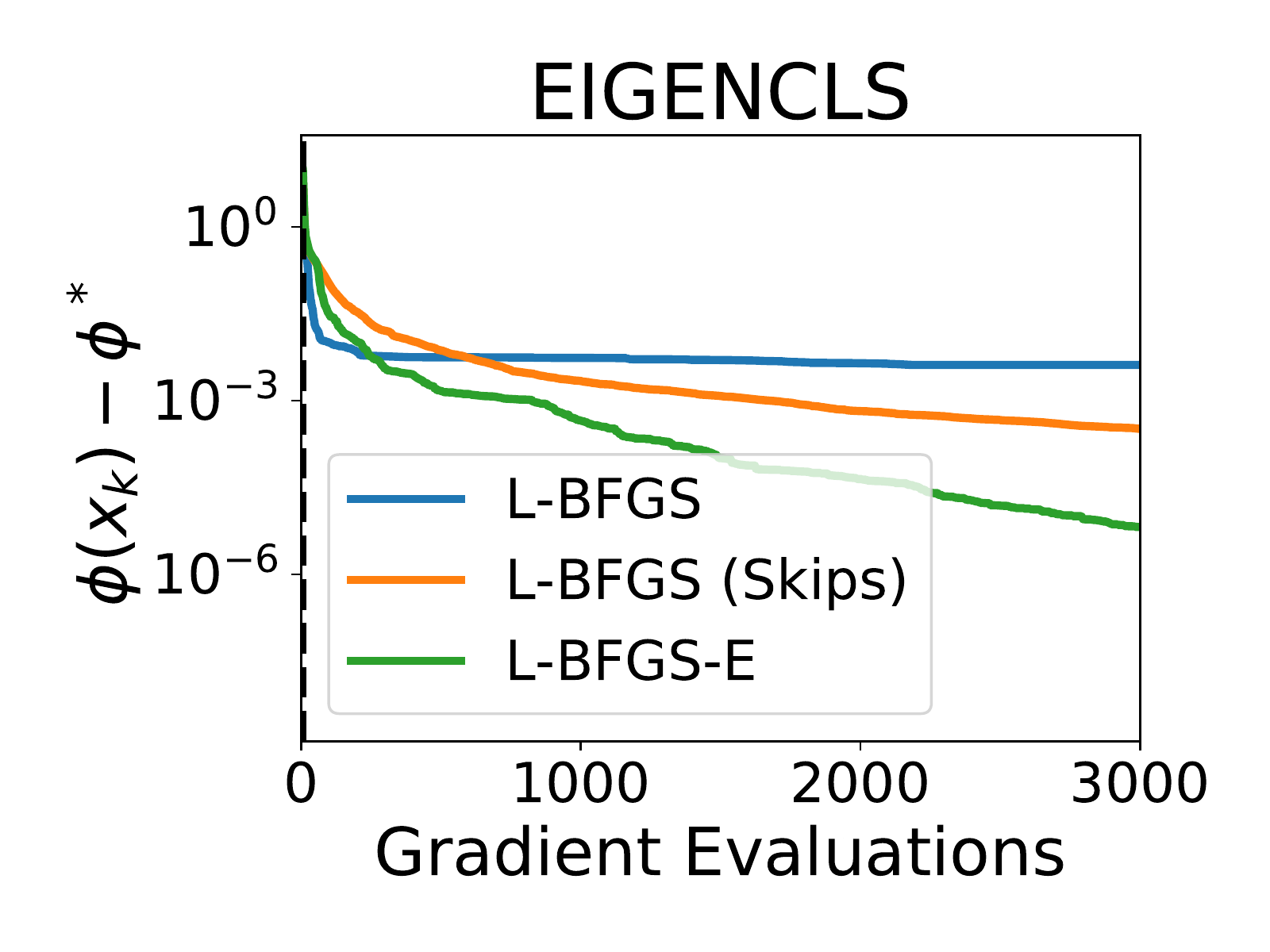}
    \includegraphics[width=0.32\linewidth]{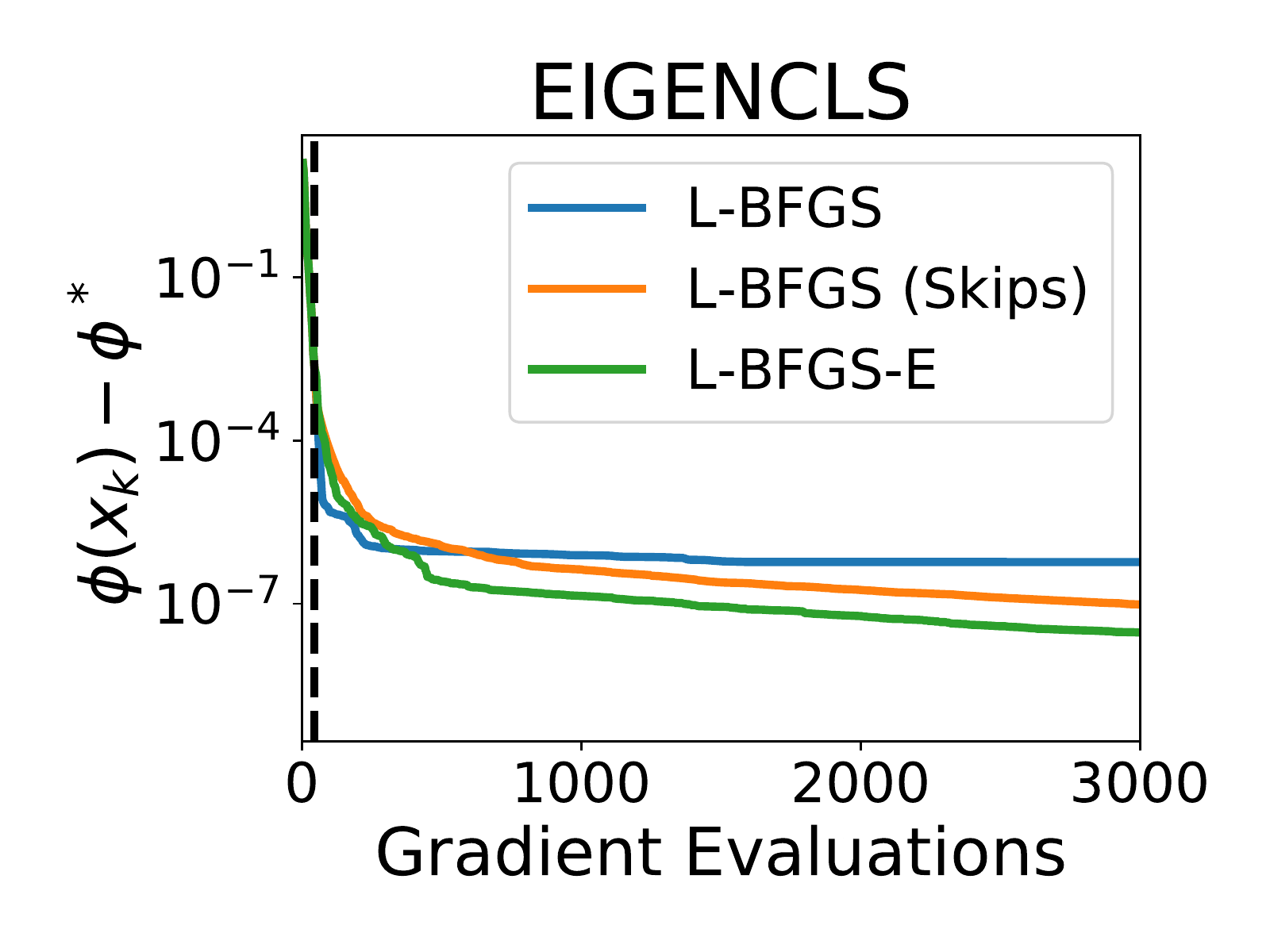}
    \includegraphics[width=0.32\linewidth]{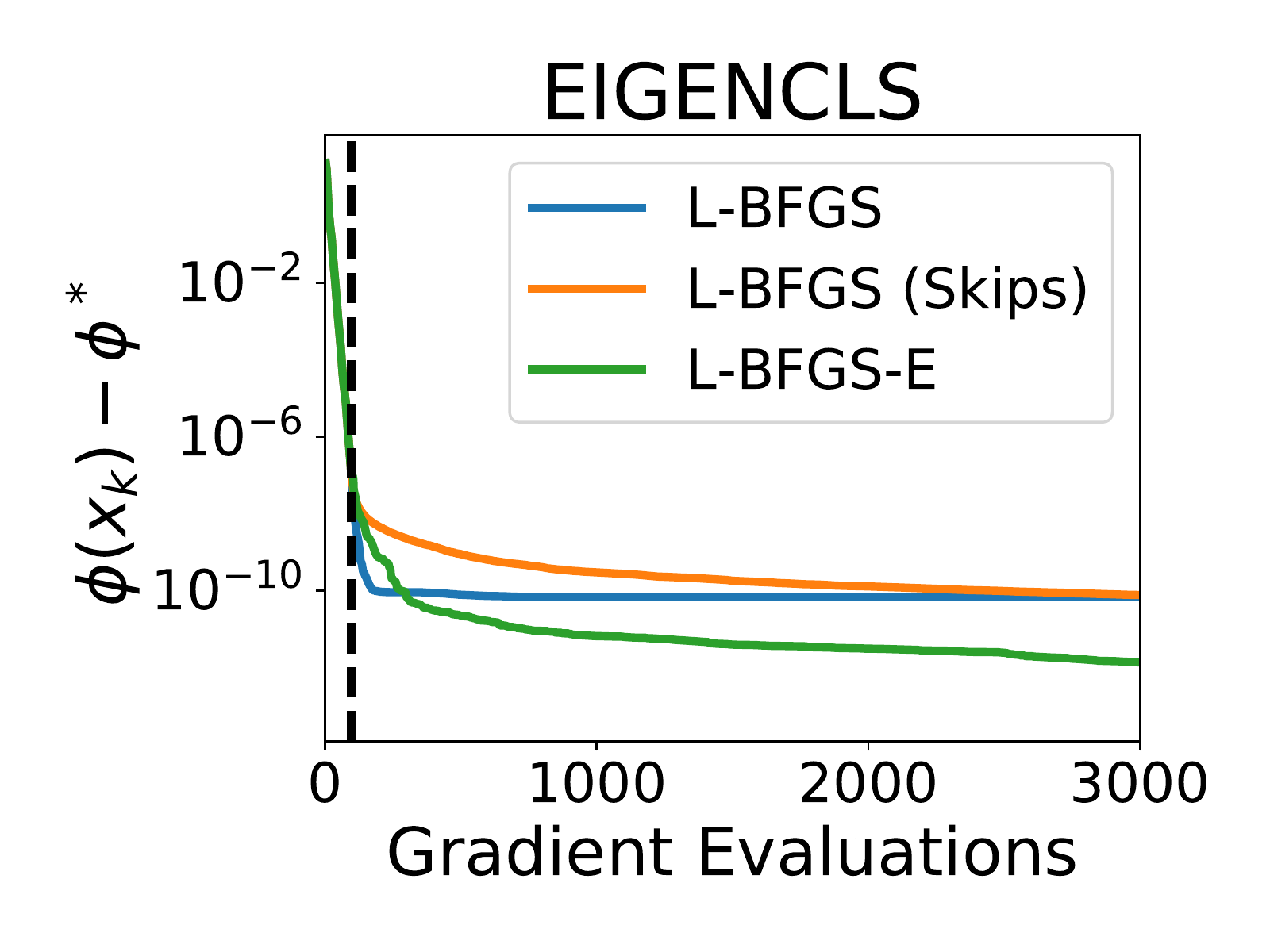}
    \caption{The true optimality gap $\phi(x_k) - \phi^*$ against the number of gradient evaluations on the \texttt{EIGENCLS} problem for $\epsilon_f=0$, and for the following gradient noise levels:  $\xi_g = 10^{-1}$ (left), $10^{-3}$ (middle), and $10^{-5}$ (right). The black dashed line denotes the iteration before the split phase becomes active.}
    \label{fig:obj for EIGENCLS}
\end{figure}

\begin{figure}[htp]
    \centering
    \includegraphics[width=0.32\linewidth]{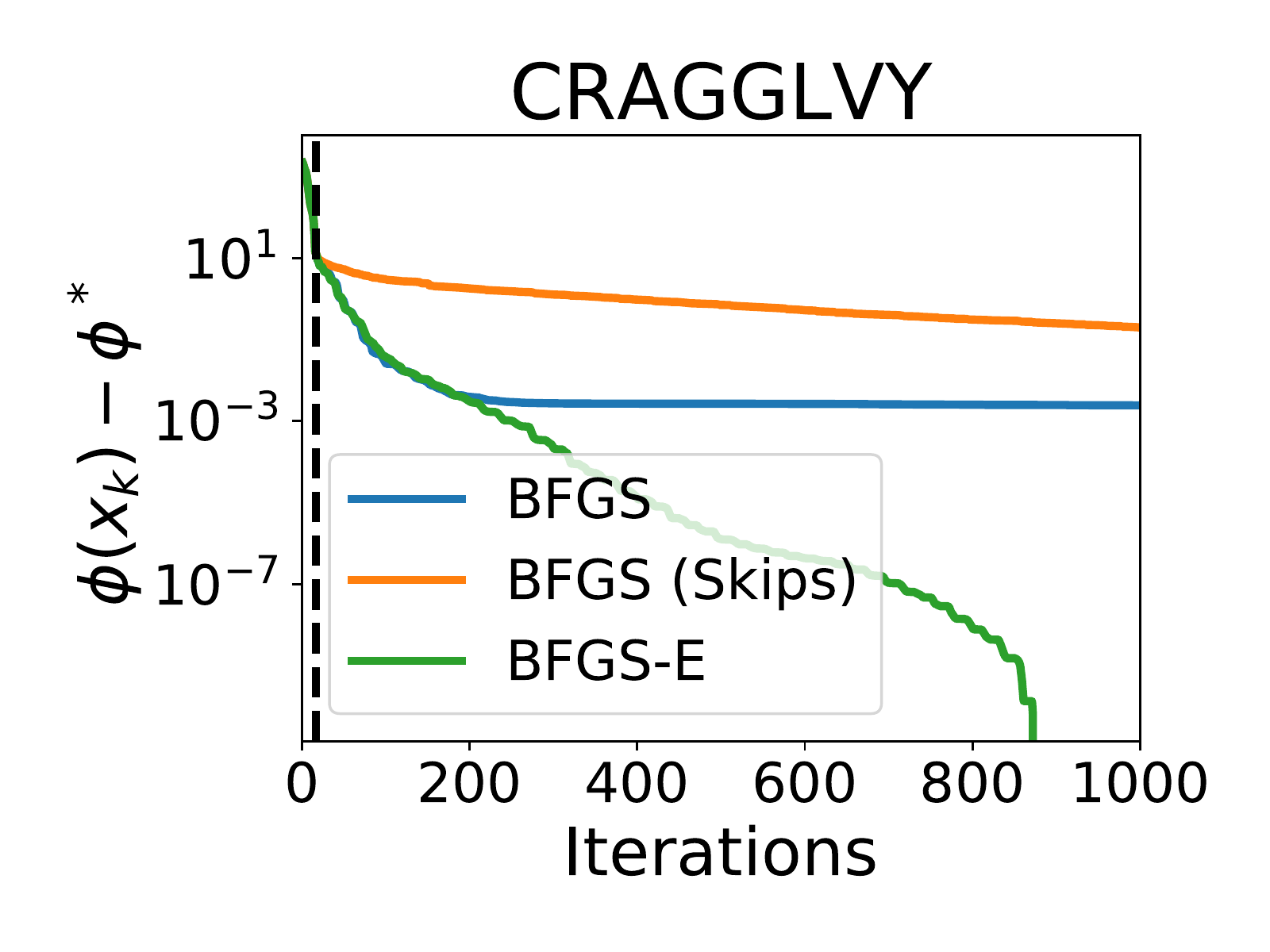}
    \includegraphics[width=0.32\linewidth]{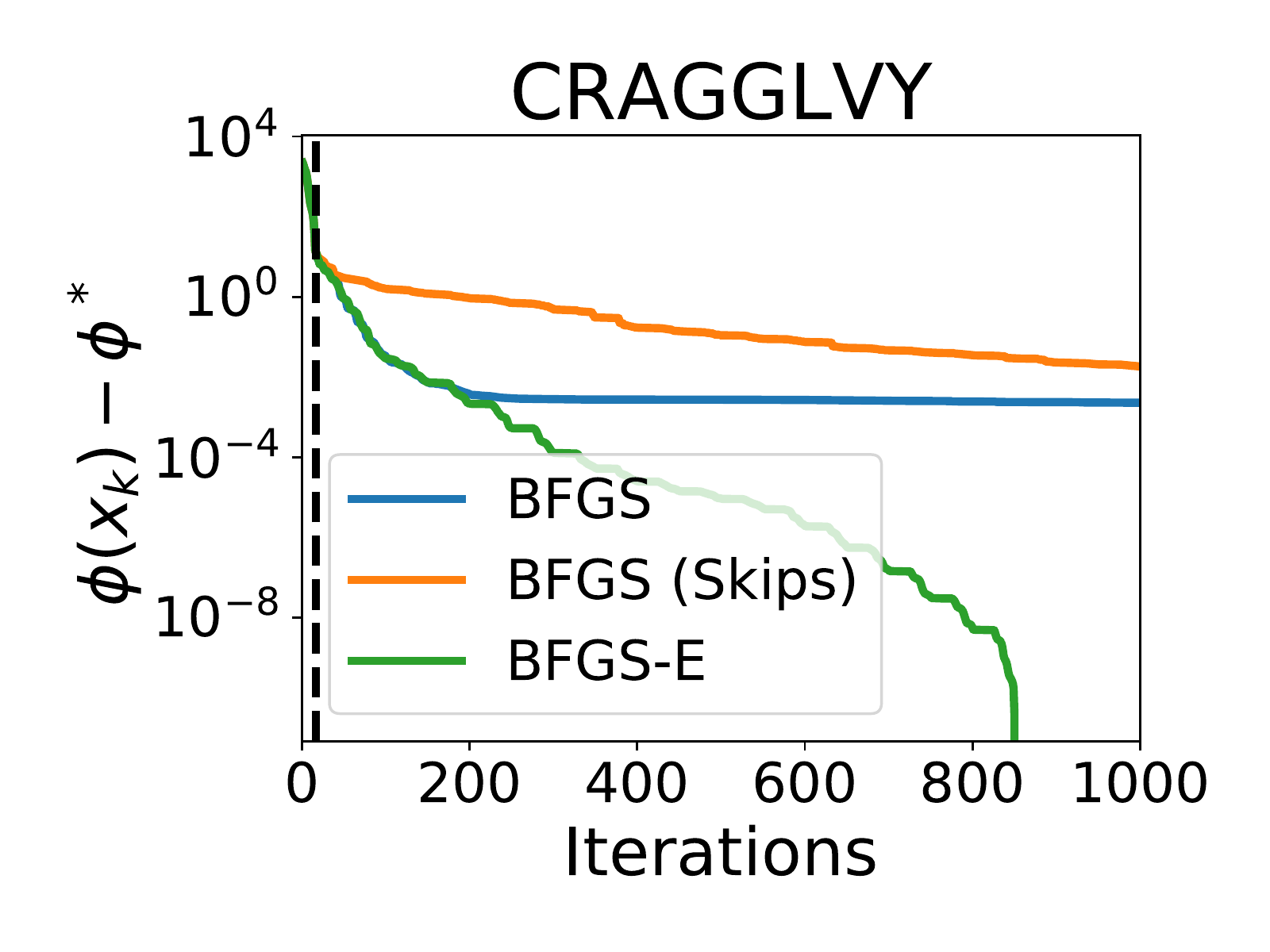}
    \includegraphics[width=0.32\linewidth]{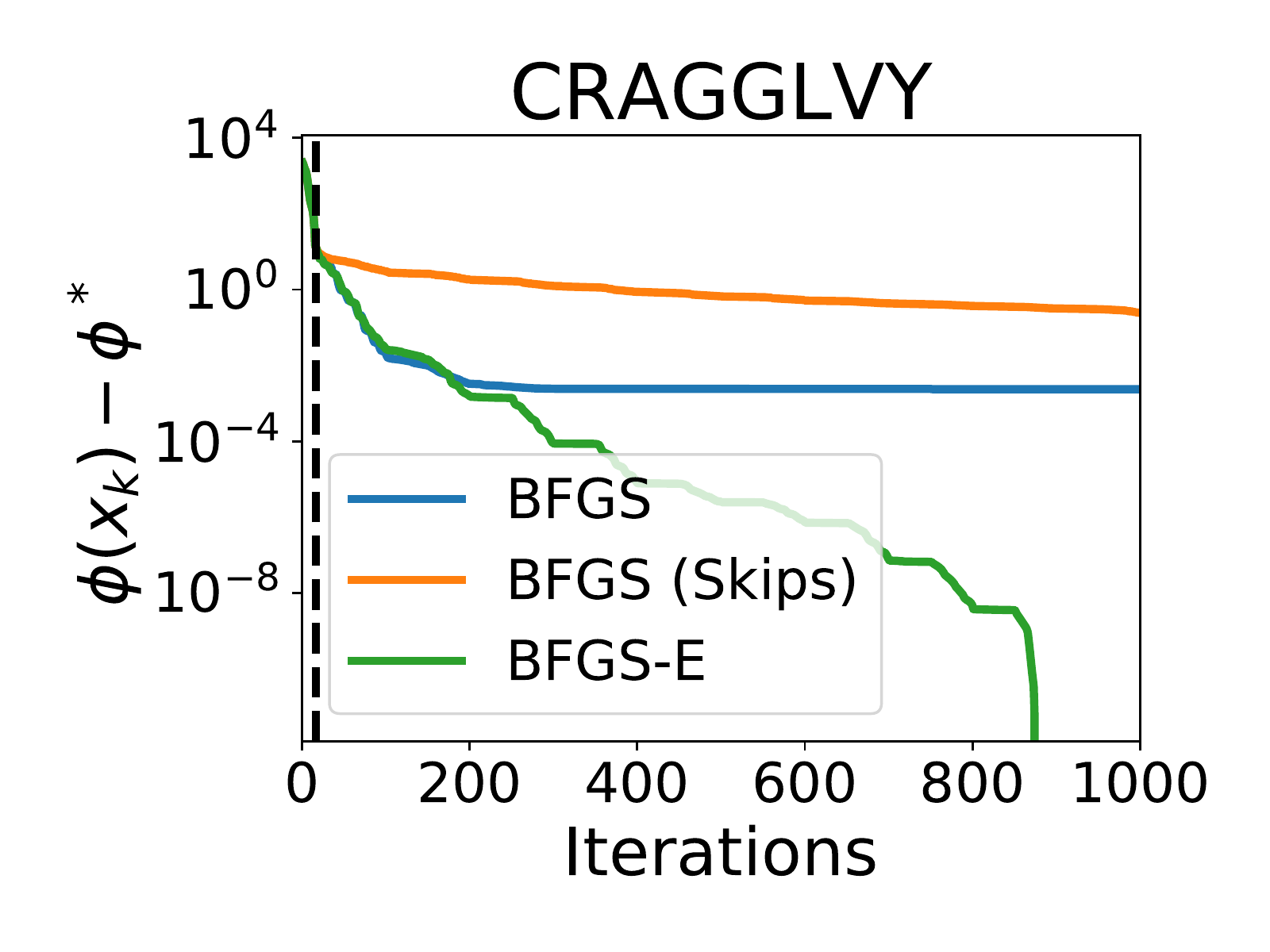}
    \includegraphics[width=0.32\linewidth]{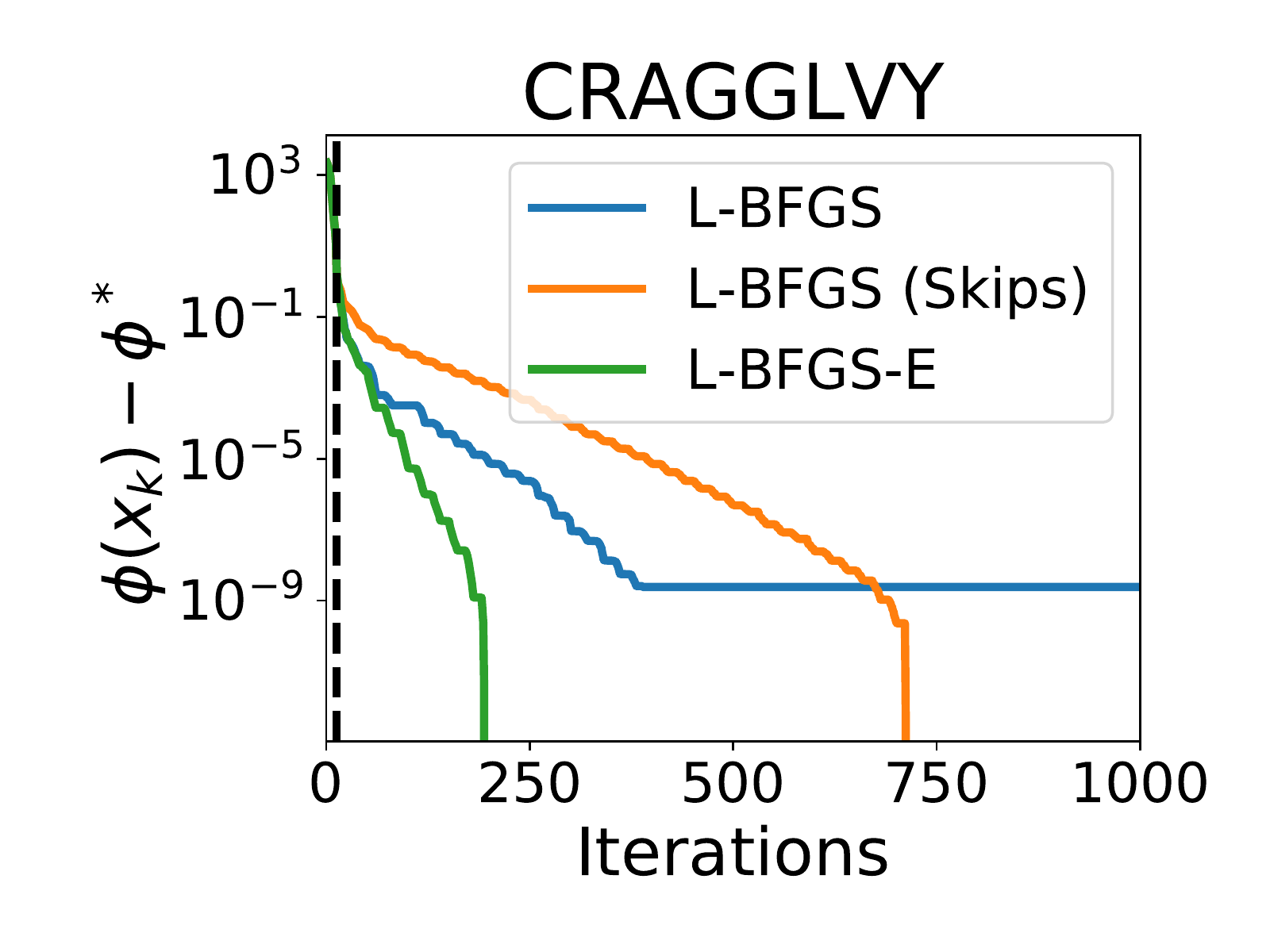}
    \includegraphics[width=0.32\linewidth]{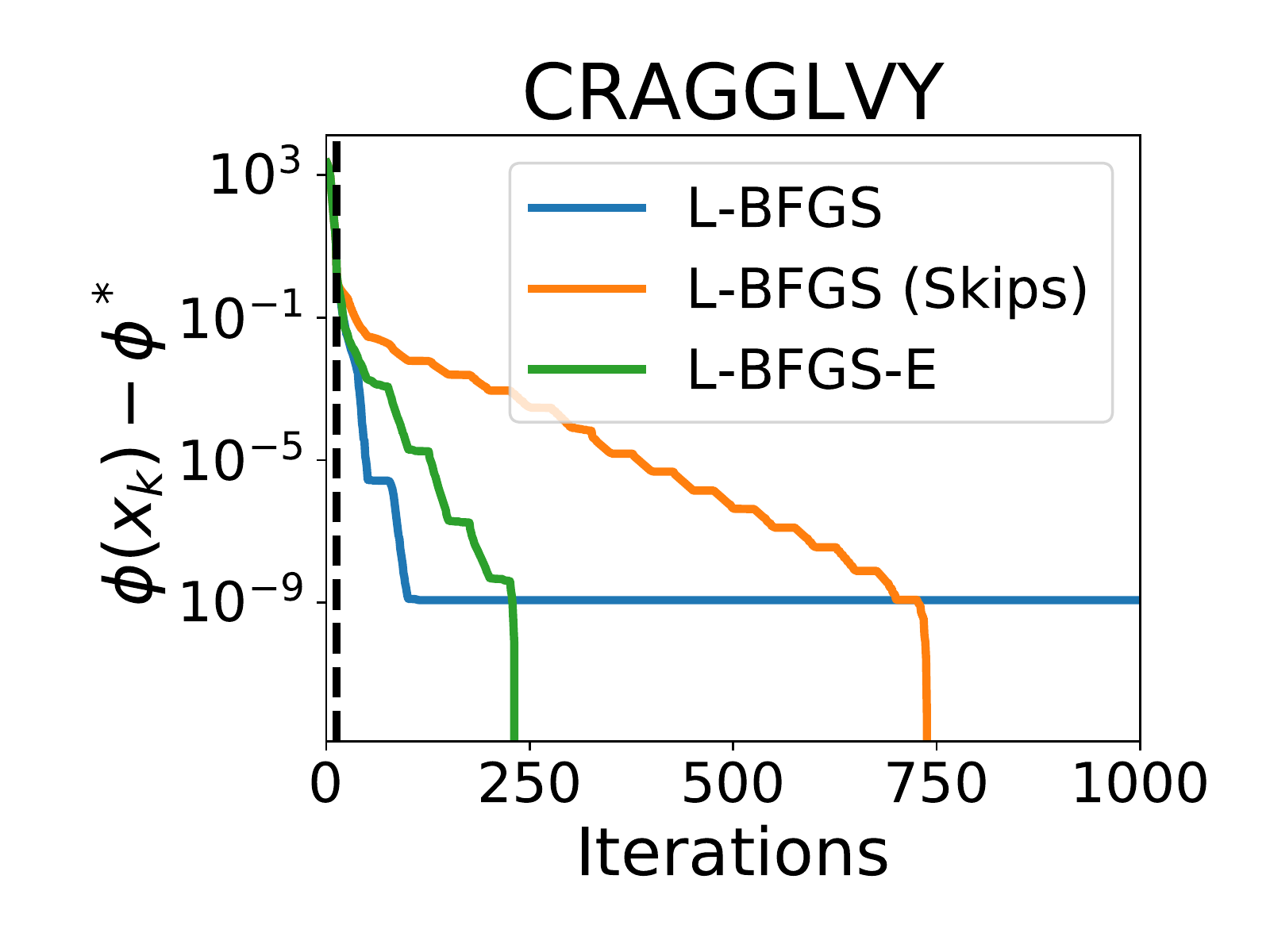}
    \includegraphics[width=0.32\linewidth]{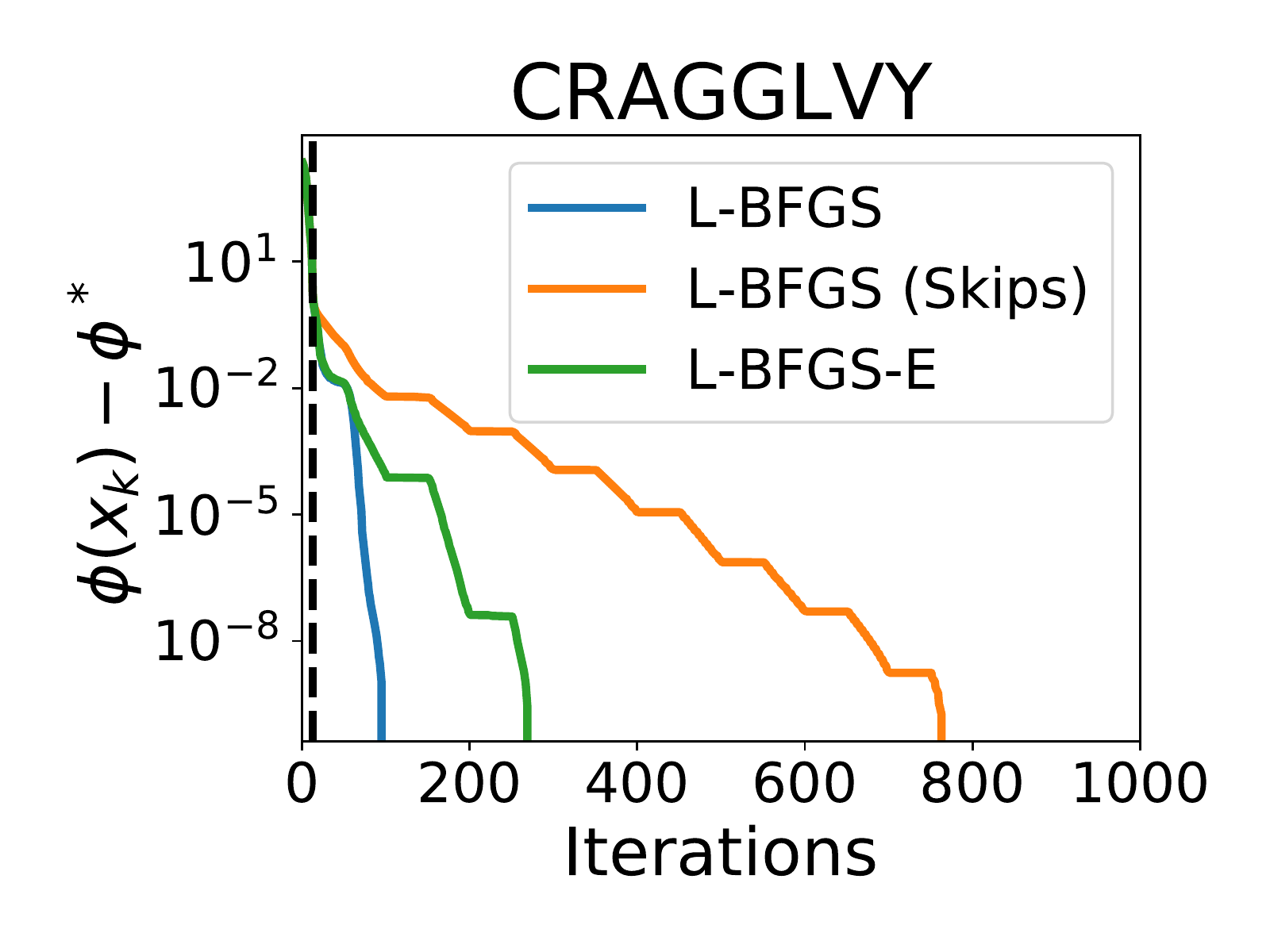}
    \caption{Intermittent Noise. Optimality gap $\phi(x_k) - \phi^*$ against the number of iterations on the \texttt{CRAGGLVY} problem. $\xi_f=0$ and $\xi_g$ alternates between 0 and with $\xi_g = 10^{-1}$ every $N_{\text{noise}}$ iterations. Results for $N_{\text{noise}} = 10$ (left), $25$ (middle), and $50$ (right). The black dashed line denotes the iteration before the split phase becomes active.}
    \label{fig:intermit CRAGGLVY}
\end{figure}

In Figure \ref{fig:obj for ENGVAL1}, we see that BFGS (Skips) and L-BFGS (Skips) can be much more efficient than BFGS-E and L-BFGS-E. In general, we found that methods that employ update skipping can be a strong alternative to lengthening if the problem is fairly well-conditioned and the Hessian does not change much, using much fewer gradient evaluations than the two-phase line search. However, it can fail to capture the change in curvature that is necessary for more difficult problems, such as \texttt{EIGENCLS} in Figure~\ref{fig:obj for EIGENCLS}. In such cases, continuing to update the BFGS matrix using lengthening is able to continue to improve the quality of the Hessian approximation for more difficult problems, leading to faster decrease in the objective value compared to skipping.

To see how update skipping compares to lengthening in the intermittent setting, we report in Figure \ref{fig:intermit CRAGGLVY}  results on \texttt{CRAGGLVY},  a problem of high difficulty.
The skipping methods are able to make faster progress when noise is diminished but not as quickly as the noise-tolerant methods since they do not benefit from good updates to the BFGS matrix. 

Since skipping is not as robust as lengthening for handling more difficult problems and in taking advantage of fluctuating noise, we do not report its numerical results for the experiments in the following section.

\subsection{Experiments with Function and Gradient Noise}

In this set of experiments, we inject noise in both the function and gradient, i.e., $\epsilon_f, \epsilon_g > 0$. First, we report in Figures \ref{fig:DIXMAANH fg noise 1} and \ref{fig:DIXMAANH fg noise 2} results for a representative example: problem \texttt{DIXMAANH}. We ran all methods for 3000 gradient evaluations to illustrate their long term behavior,  for different values of $\epsilon_g$ and $\epsilon_f$. We note that the lengthening procedure safeguards the Hessian updating in the presence of function noise, and the relaxation in the Armijo condition \eqref{eq:relaxed armijo} allows the methods to continue making progress far below the noise level of the function if the gradient noise is sufficiently small to provide good search directions.


\begin{figure}[htp]
    \centering
    \includegraphics[width=0.32\linewidth]{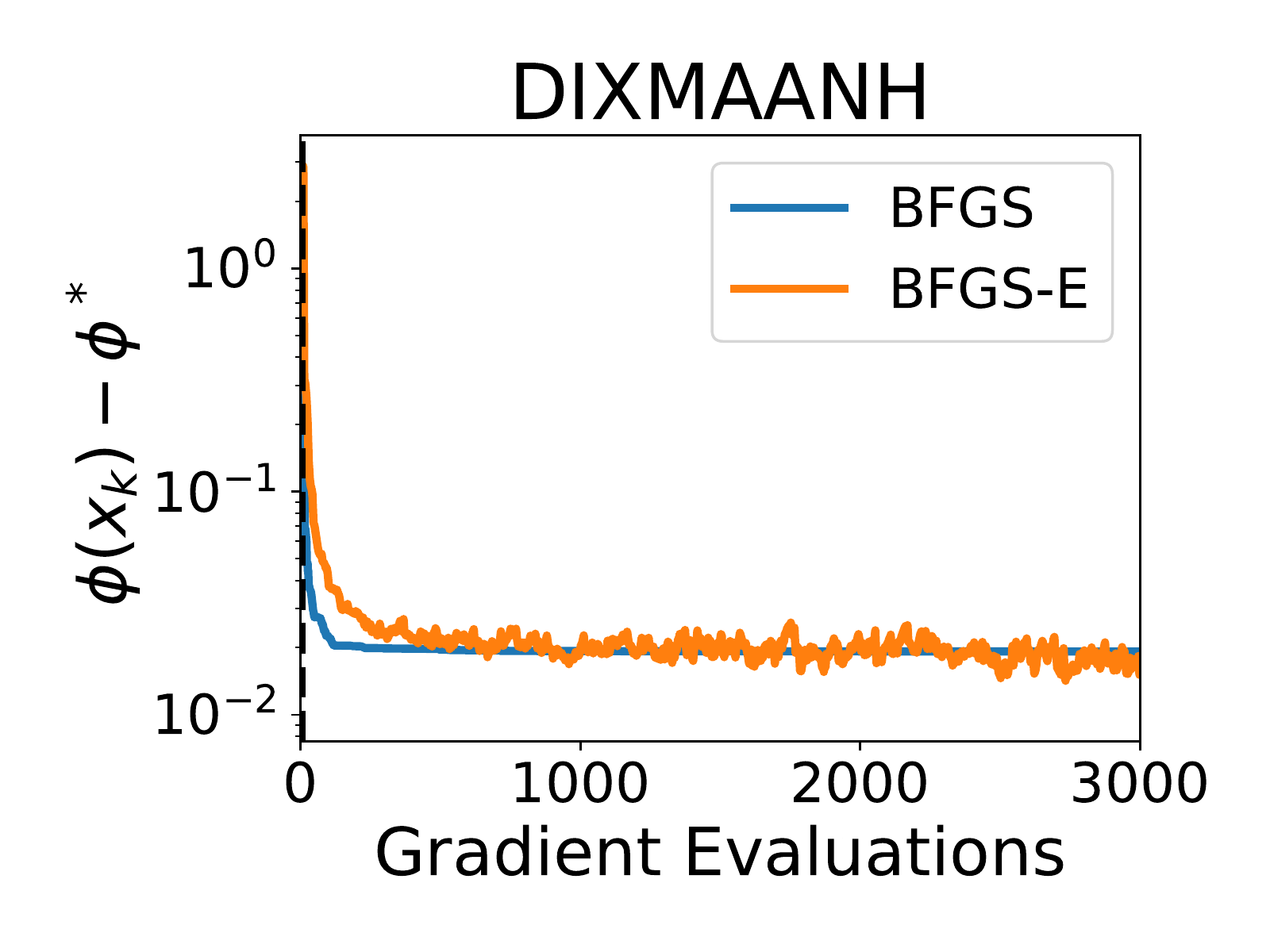}
    \includegraphics[width=0.32\linewidth]{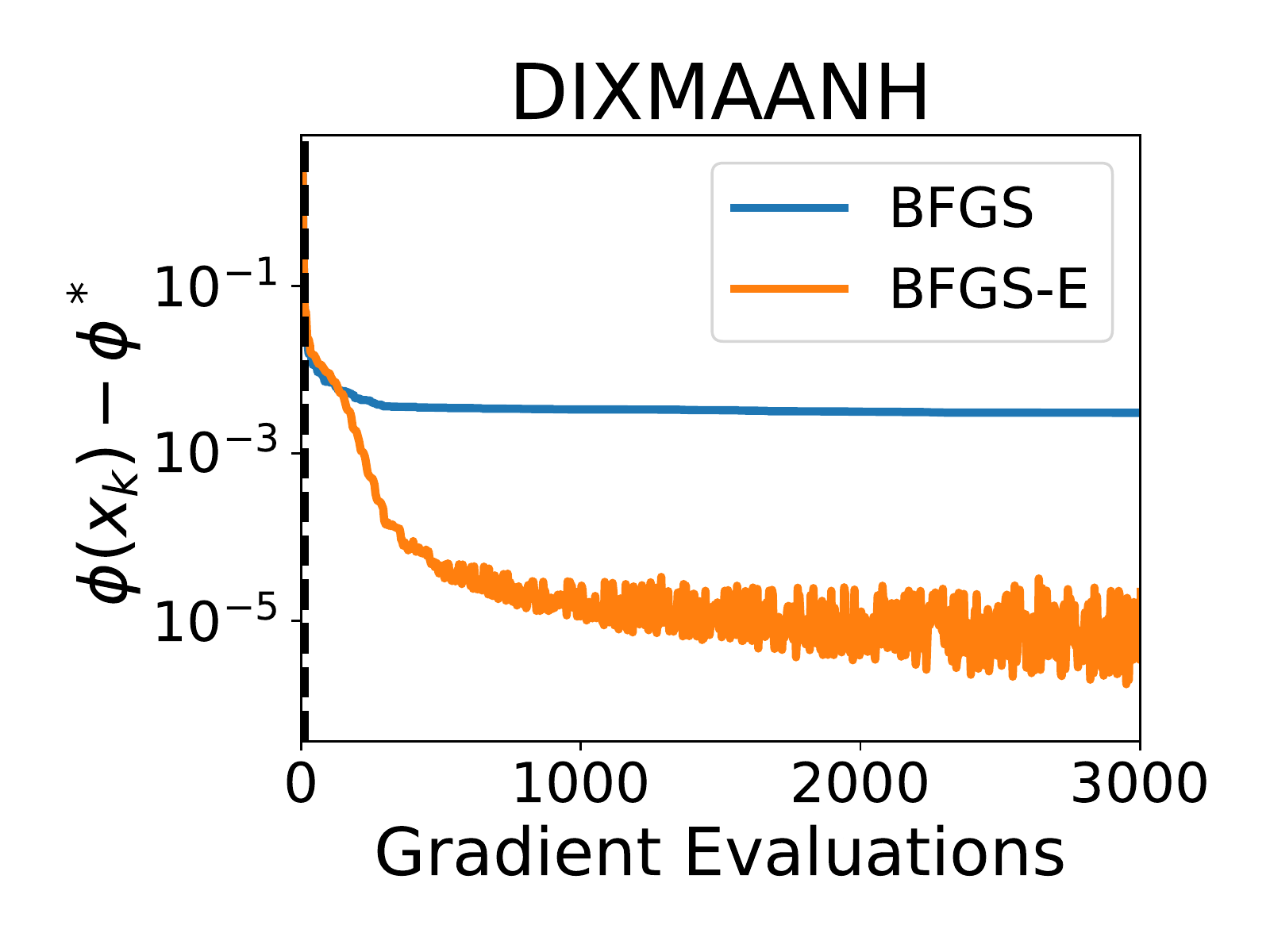}
    \includegraphics[width=0.32\linewidth]{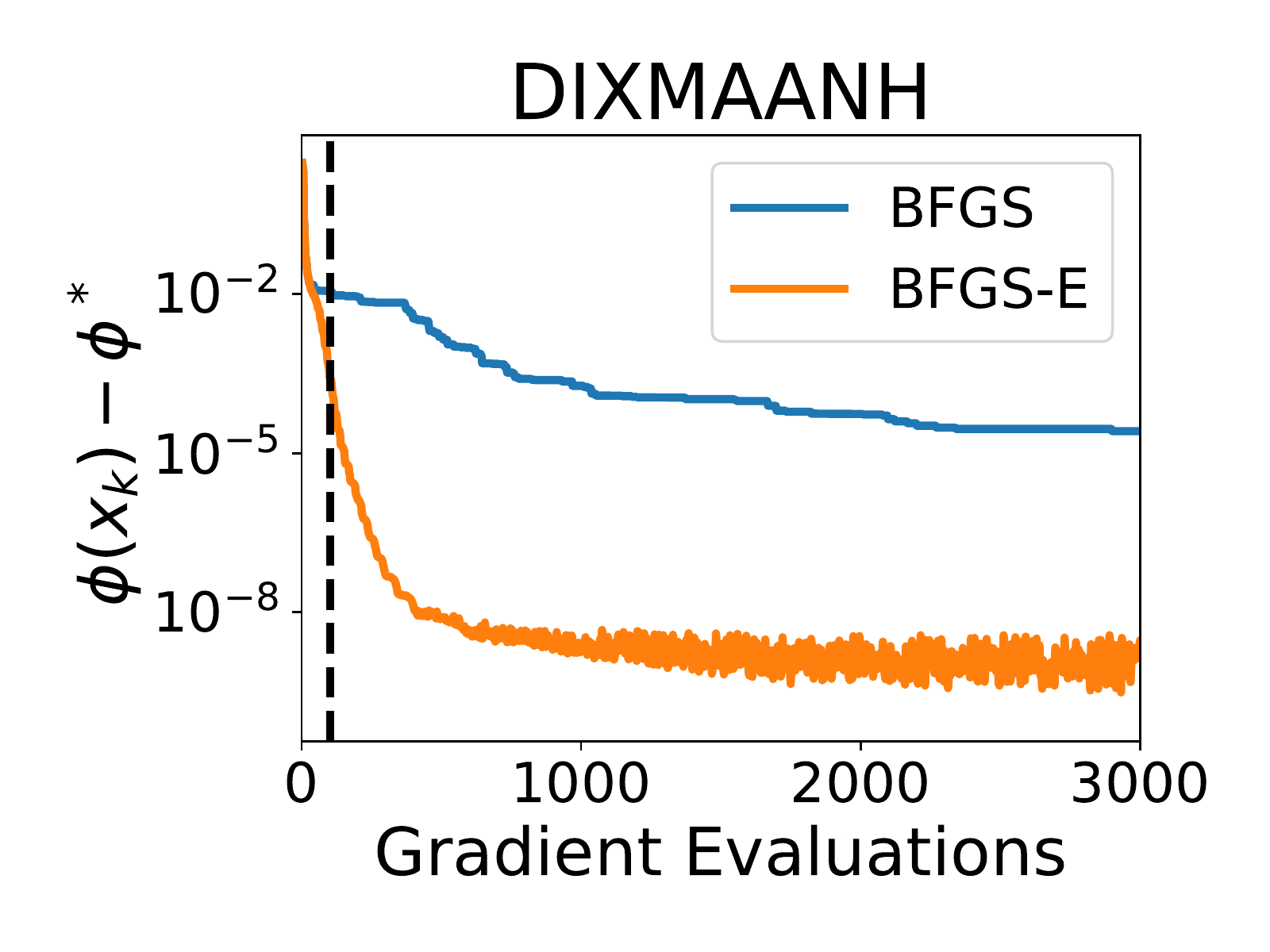} \\
    \includegraphics[width=0.32\linewidth]{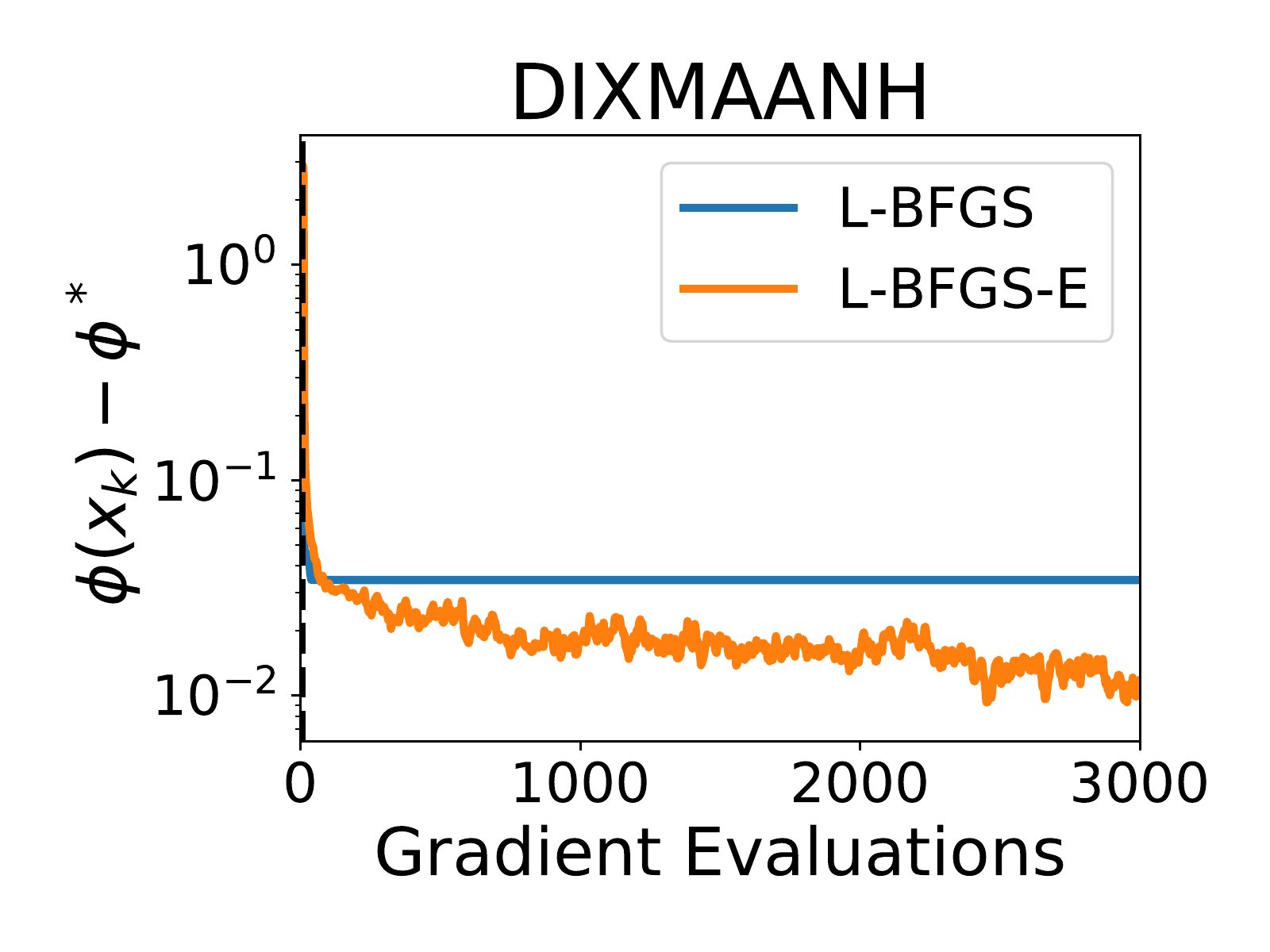}
    \includegraphics[width=0.32\linewidth]{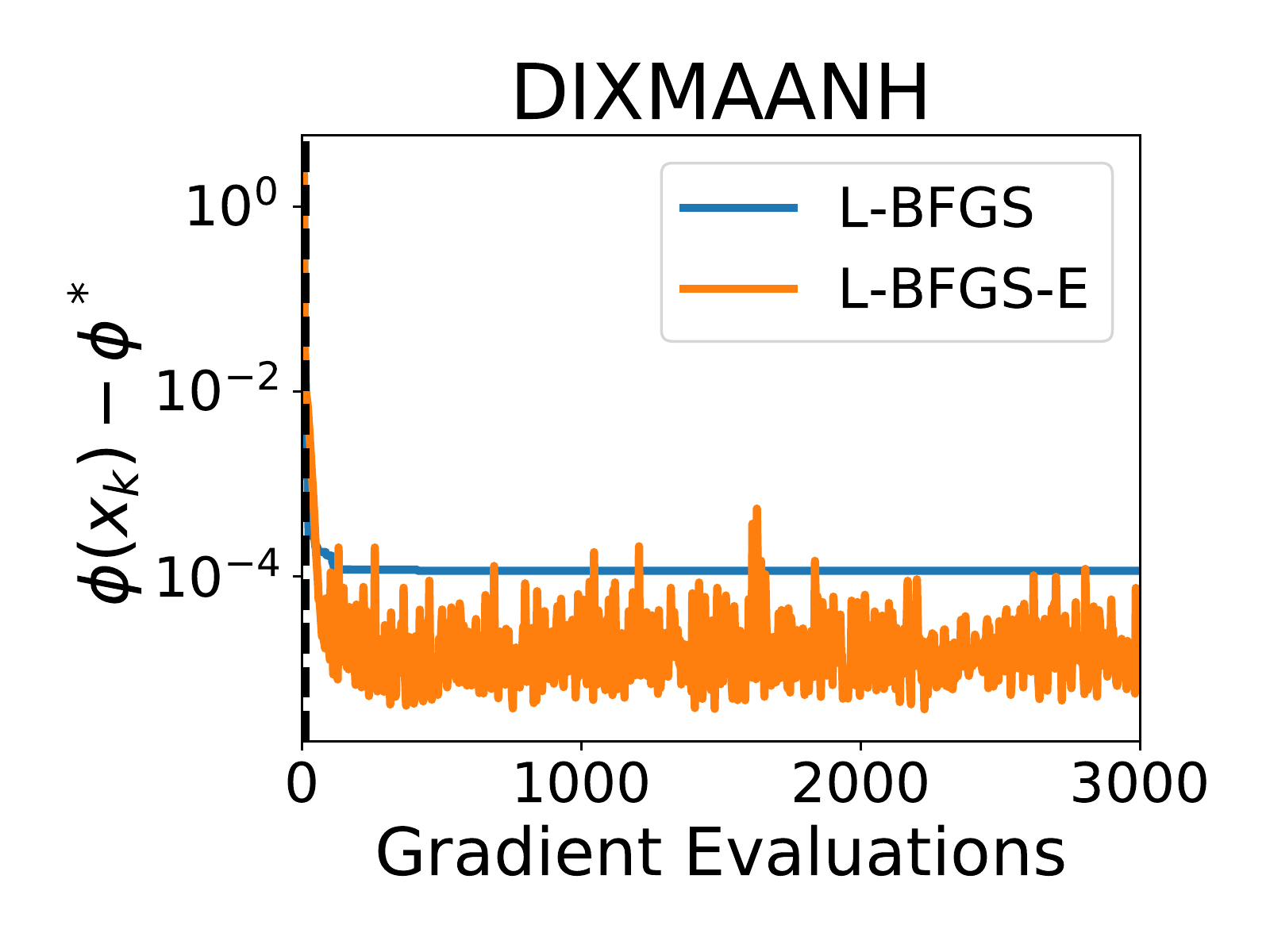}
    \includegraphics[width=0.32\linewidth]{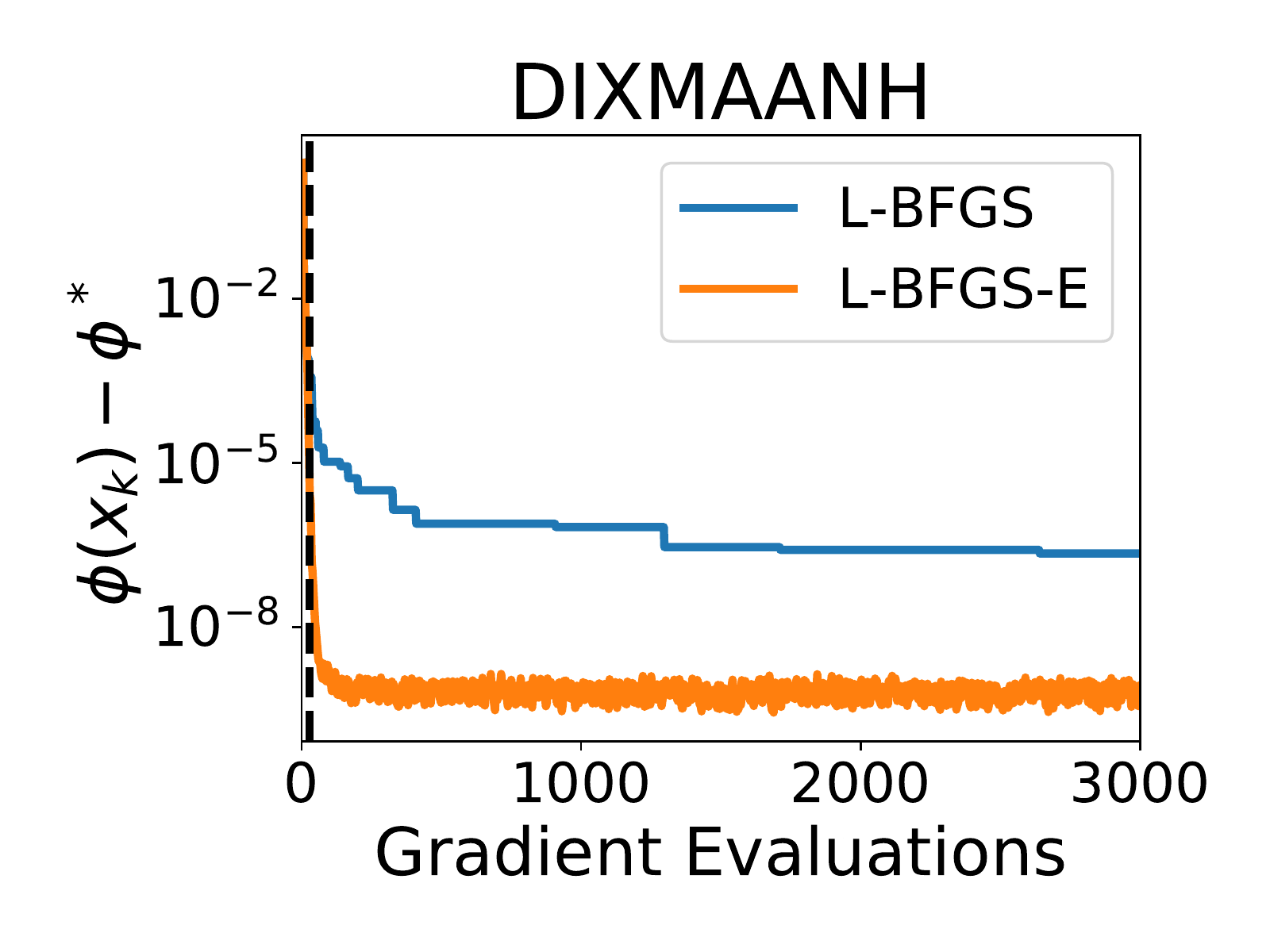}
    \caption{Optimality gap $\phi(x_k) - \phi^*$ against the number of gradient evaluations on problem \texttt{DIXMAANH}, with $\xi_f = 10^{-3}$ on all six plots, and with $\xi_g = 10^{-1}$ (left), $\xi_g=10^{-3}$ (middle), and $\xi_g= 10^{-5}$ (right). The black dashed line denotes the iteration before the split phase becomes active.
    }
    \label{fig:DIXMAANH fg noise 1}
\end{figure}

\begin{figure}
    \centering
    \includegraphics[width=0.32\linewidth]{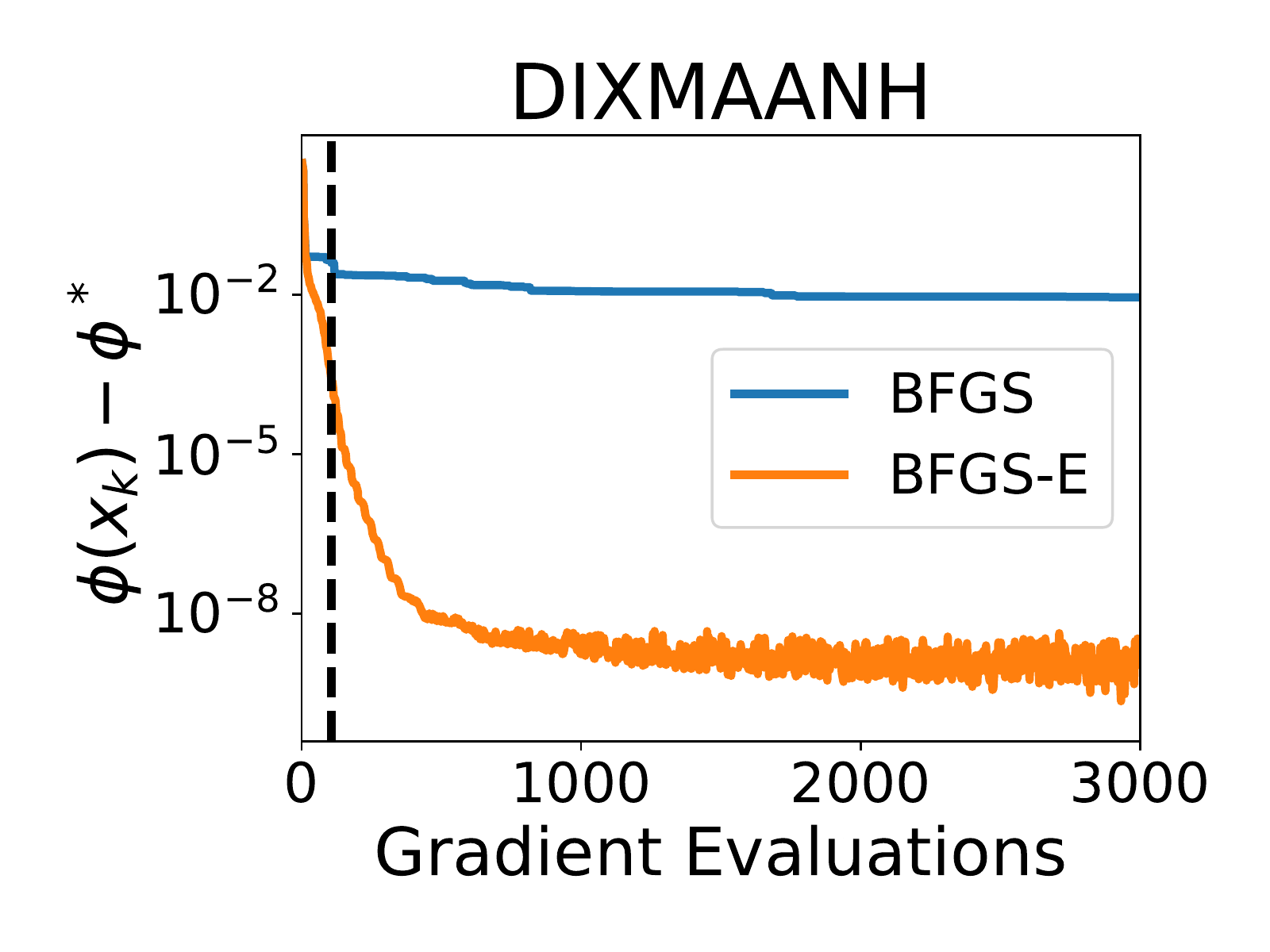}
    \includegraphics[width=0.32\linewidth]{final_plots/fg_noise/BFGS/obj_ts_grad_evals_DIXMAANH_SUA_1e-3_1e-5_1226.pdf}
    \includegraphics[width=0.32\linewidth]{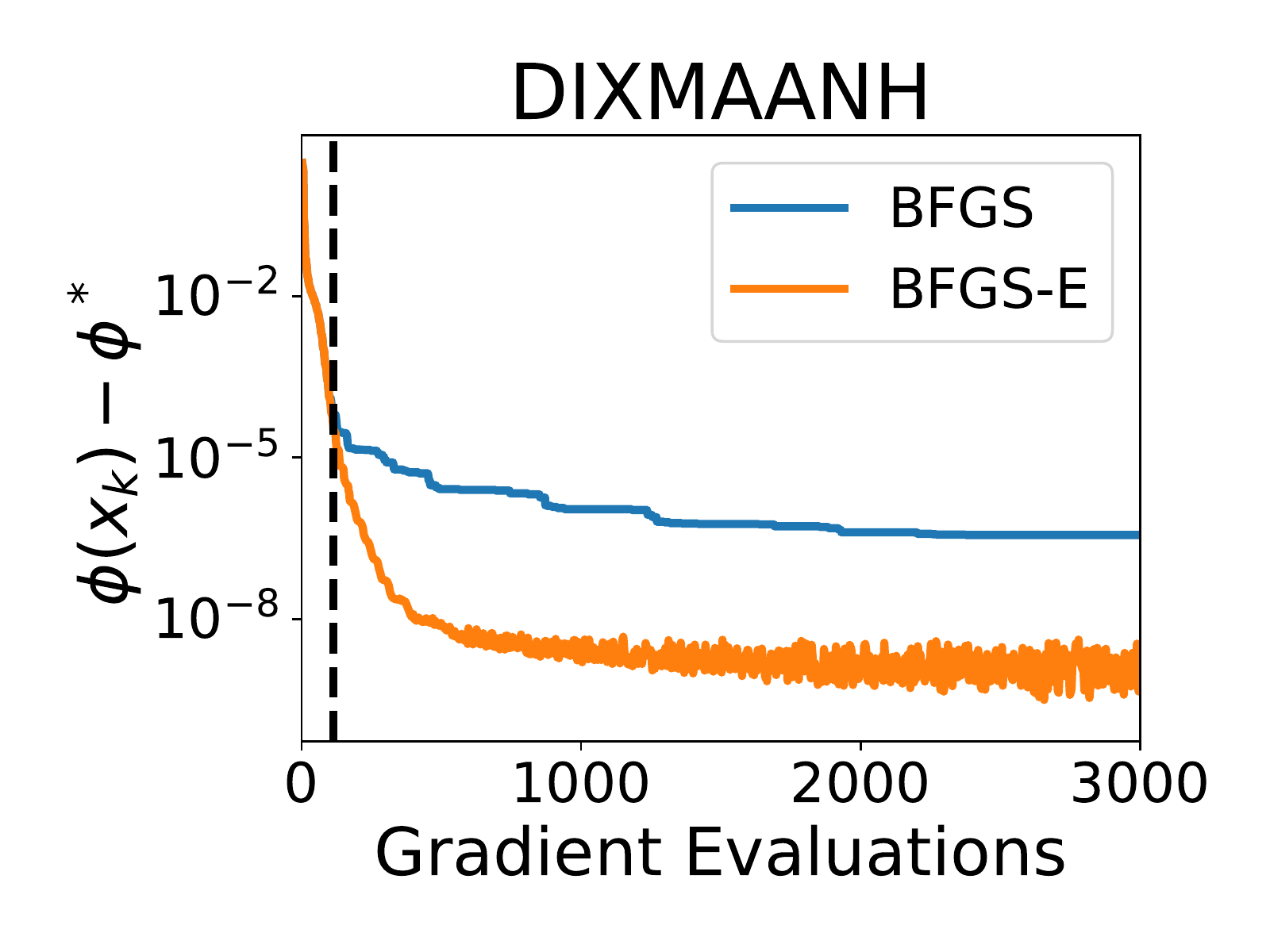} \\
    \includegraphics[width=0.32\linewidth]{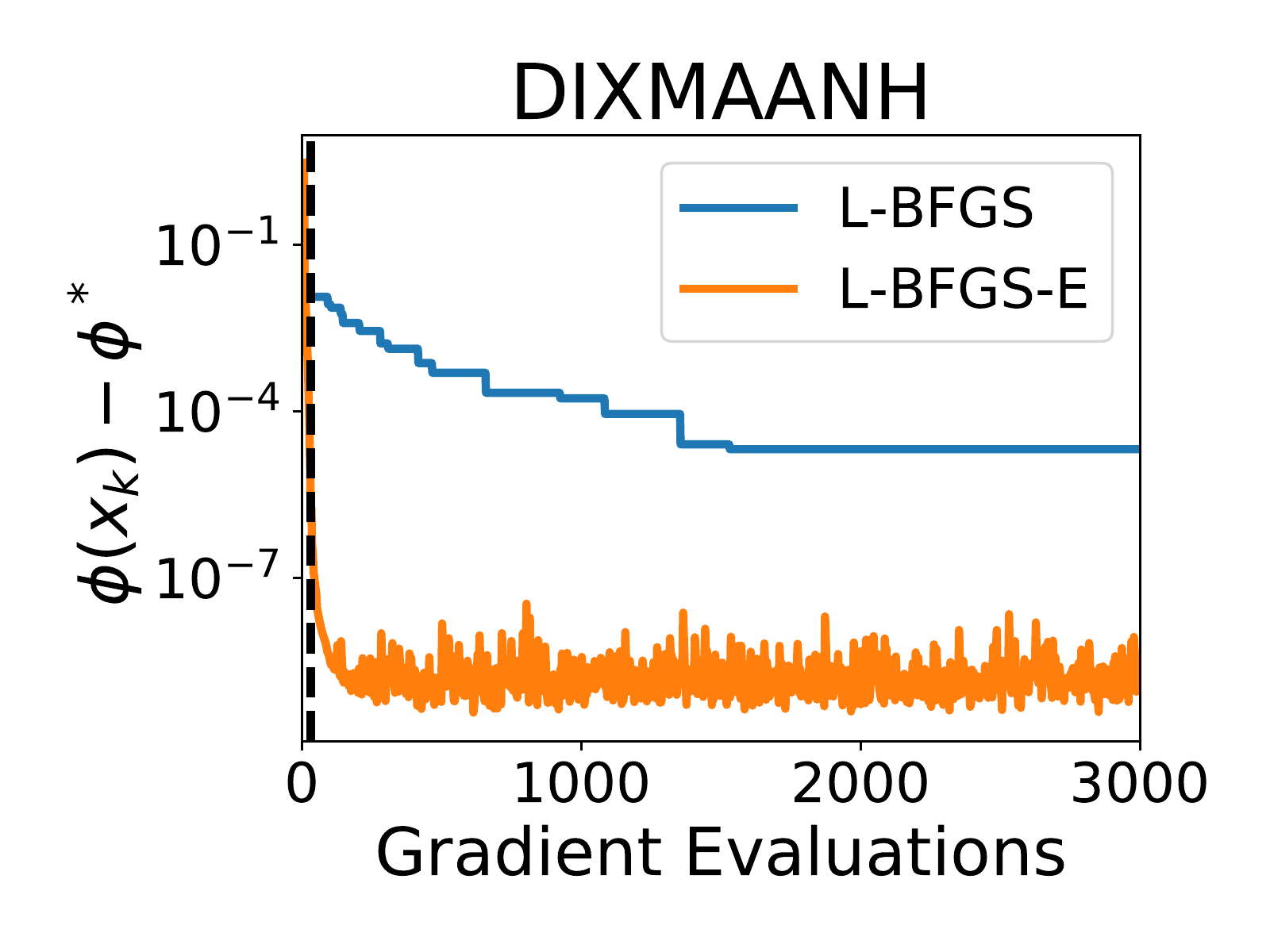}
    \includegraphics[width=0.32\linewidth]{final_plots/fg_noise/L-BFGS/obj_ts_grad_evals_DIXMAANH_SUA_1e-3_1e-5_1226.pdf}
    \includegraphics[width=0.32\linewidth]{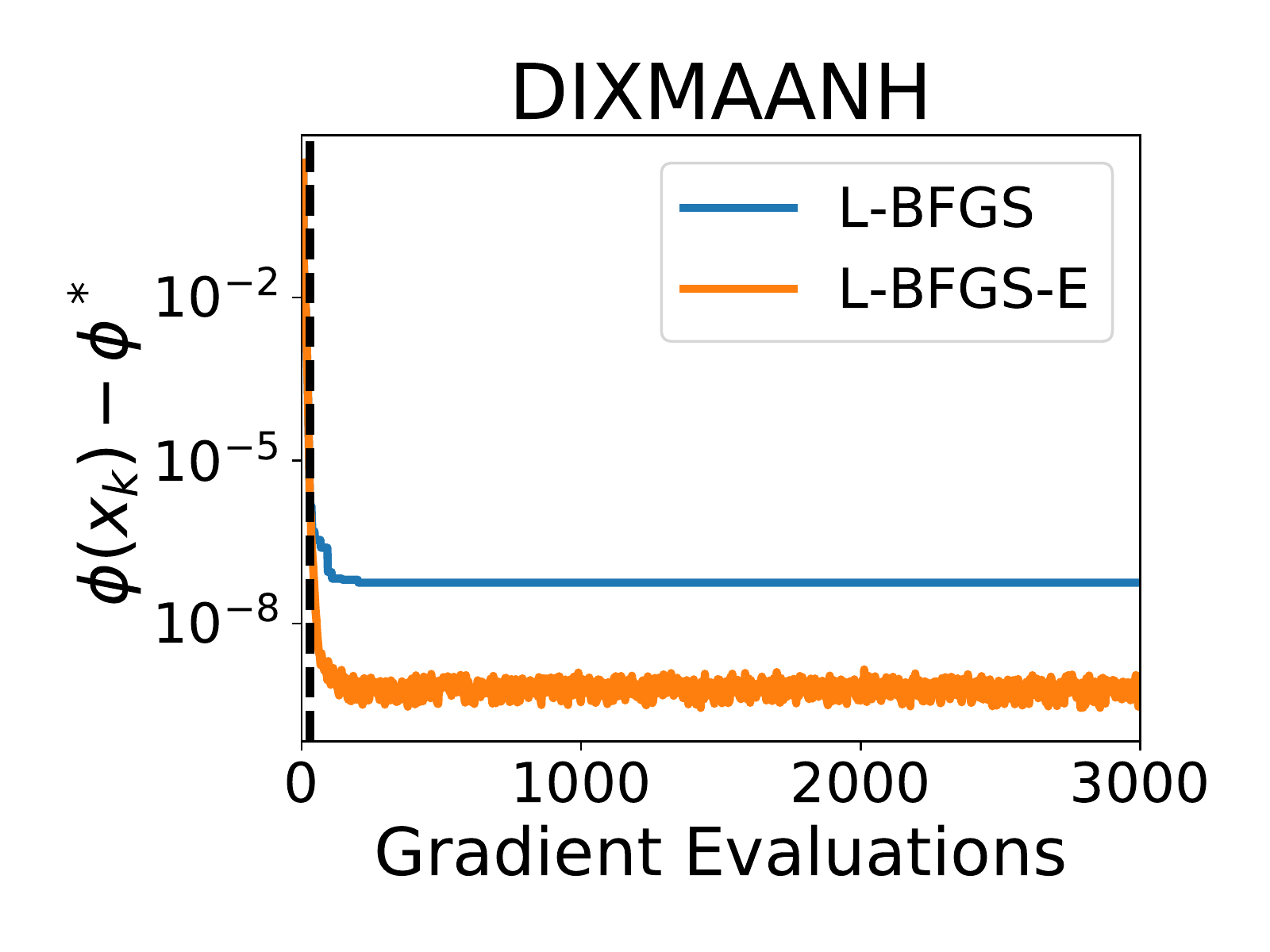}
    \caption{Optimality gap $\phi(x_k) - \phi^*$ against the number of gradient evaluations on problem \texttt{DIXMAANH}  with $\xi_g = 10^{-5}$ on all six plots, and with $\xi_f=10^{-1}$ (left), $\xi_f=10^{-3}$ (middle), and $\xi_f= 10^{-5}$ (right). The black dashed line denotes the iteration before the split phase becomes active.
    }
    \label{fig:DIXMAANH fg noise 2}
\end{figure}


Lastly, we report the performance of the methods on the 41 test problems listed in Table~2, using the profiles proposed by Morales  \cite{morales2011remark}. In Figures \ref{fig:morales obj 1} and \ref{fig:morales grad evals 1} we report, respectively, the  quantities
\begin{equation}   \label{value}
    \log_2\left(\frac{\phi_{new} - \phi^*}{\phi_{old} - \phi^*} \right) ~~~ \text{and} ~~~ \log_2\left(\frac{evals_{new}}{evals_{old}} \right),
\end{equation}
for each problem. Here 
 $\phi_{new}$ and $\phi_{old}$ denote the true objective value of the noise-tolerant and standard methods after 3000 iterations, and $evals_{new}$ and $evals_{old}$ denote the total number of gradient evaluations required to achieve one of the conditions:
\begin{equation}  \label{twoterm}
    \phi(x_k) - \phi^* \leq \epsilon_f ~~~ \text{or} ~~~ \|\nabla \phi(x_k)\| \leq \epsilon_g .
\end{equation}
All quantities are averaged over 5 runs with different seeds. In Figures \ref{fig:morales obj 1} and \ref{fig:morales grad evals 1} the problems are ordered in increasing value of the quantities given in \eqref{value}. One can thus gauge the success of a method by the area of the graph on its side of the half-space: the larger the area, the more successful the method.

\begin{figure}[htp]
    \centering
    \includegraphics[width=0.45\linewidth]{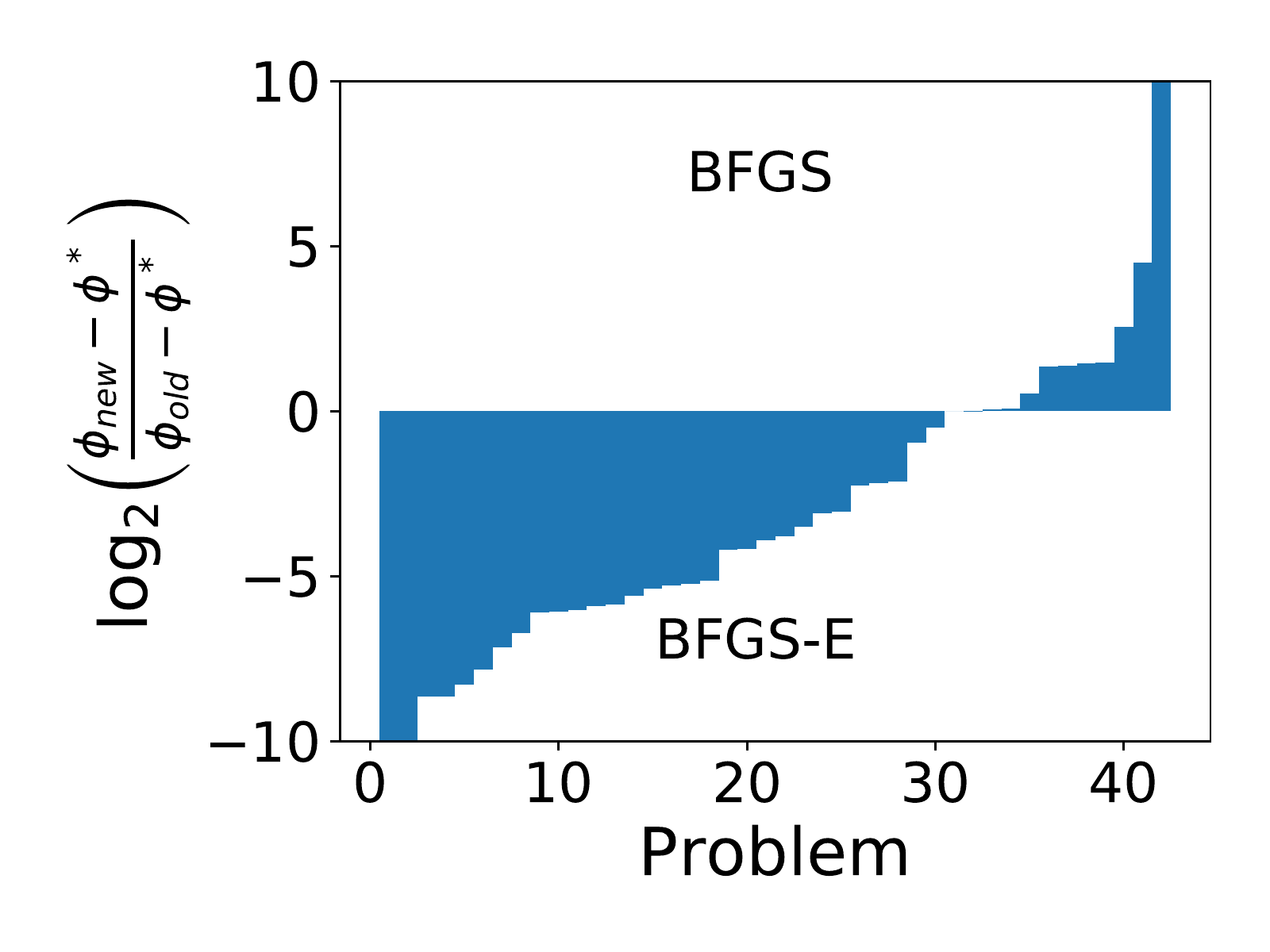}
    \includegraphics[width=0.45\linewidth]{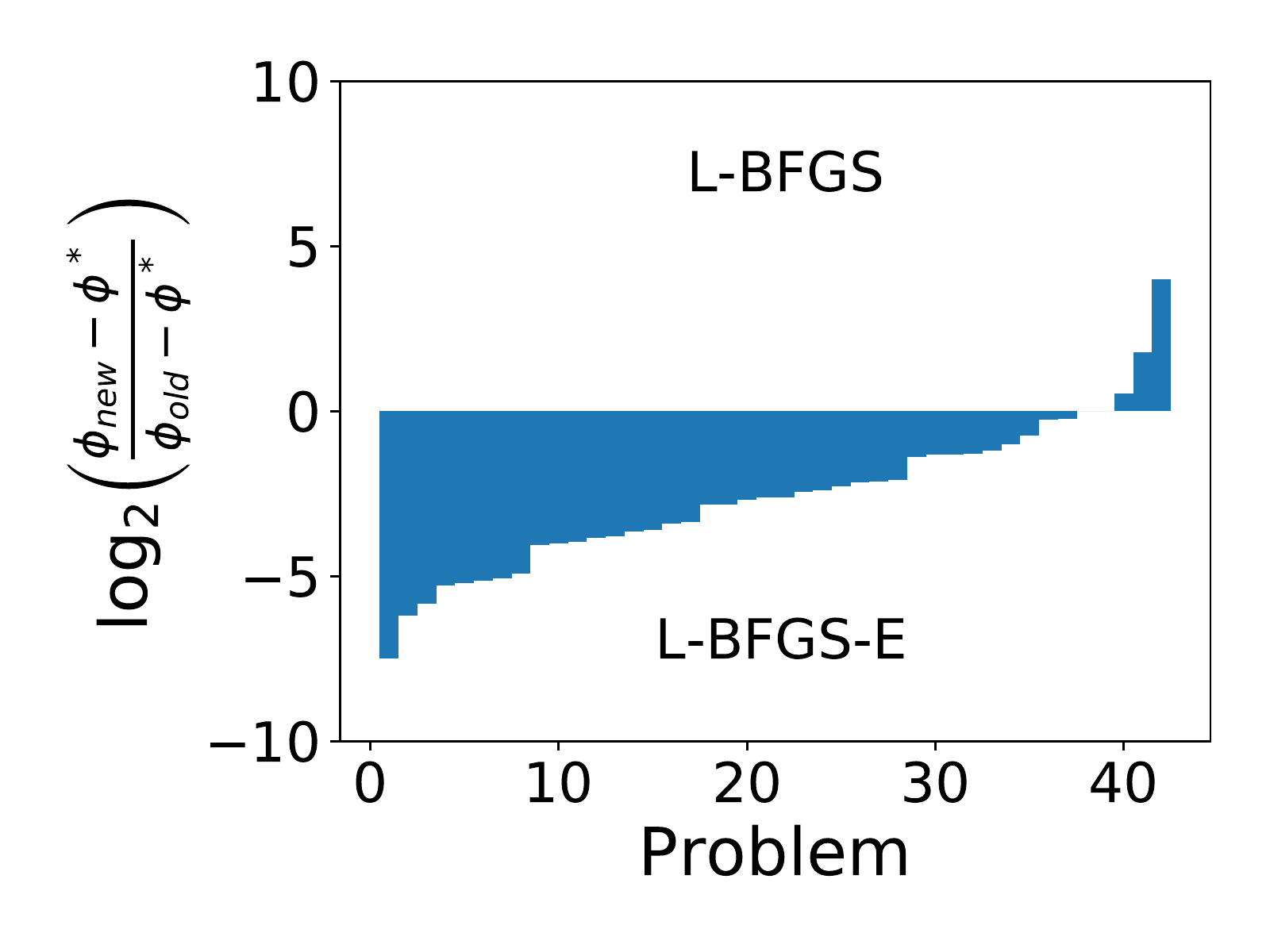}
    \caption{Morales profiles for the optimality gap $\phi(x_k) - \phi^*$ across 41 unconstrained CUTEst problems with $\xi_f = 10^{-3}$ and $\xi_g = 10^{-3}$. Results are averaged over 5 runs. The left figure compares BFGS against BFGS-E while the right figure compares L-BFGS against L-BFGS-E.}
    \label{fig:morales obj 1}
\end{figure}


\begin{figure}[htp]
    \centering
    \includegraphics[width=0.45\linewidth]{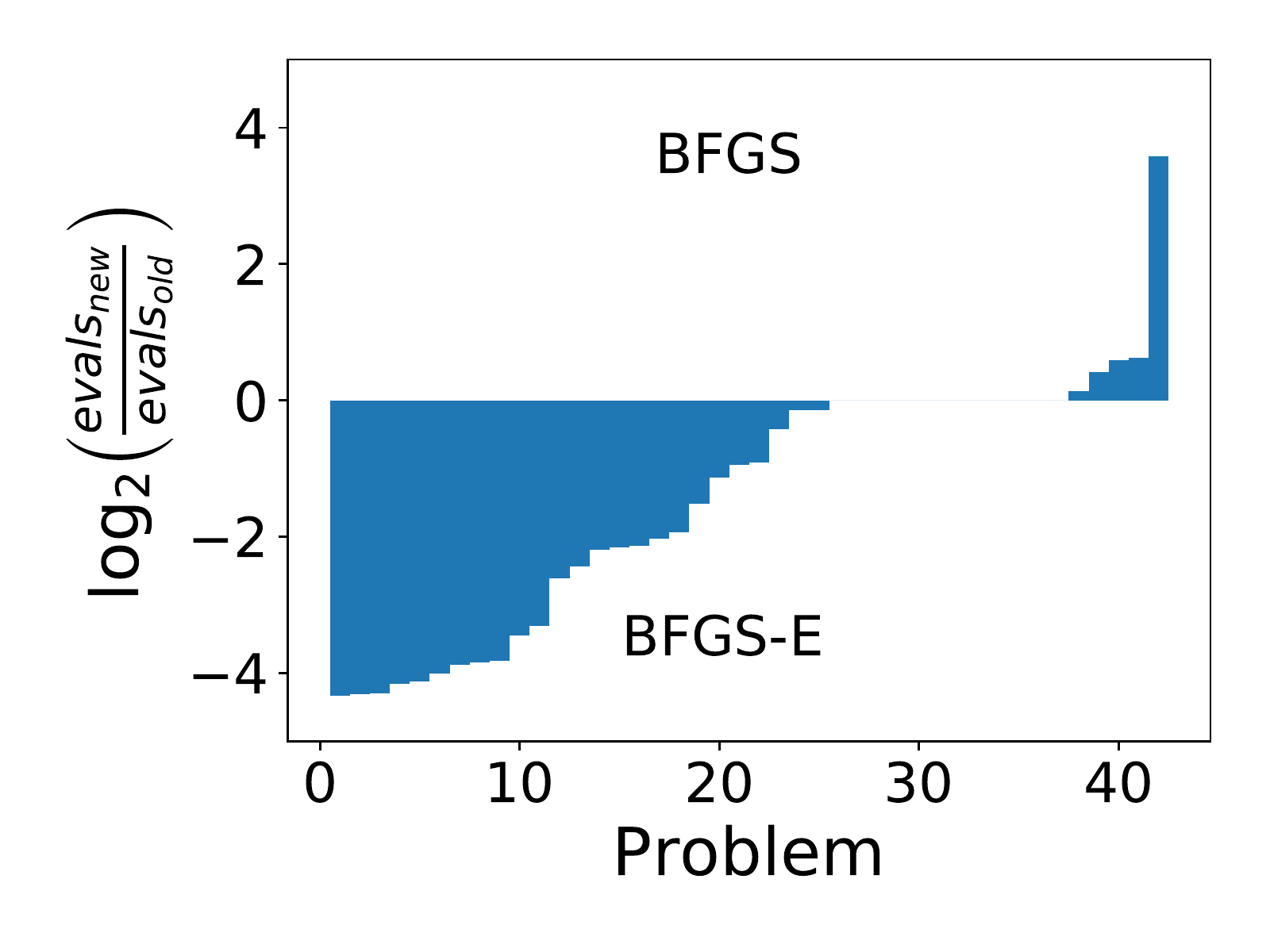}
    \includegraphics[width=0.45\linewidth]{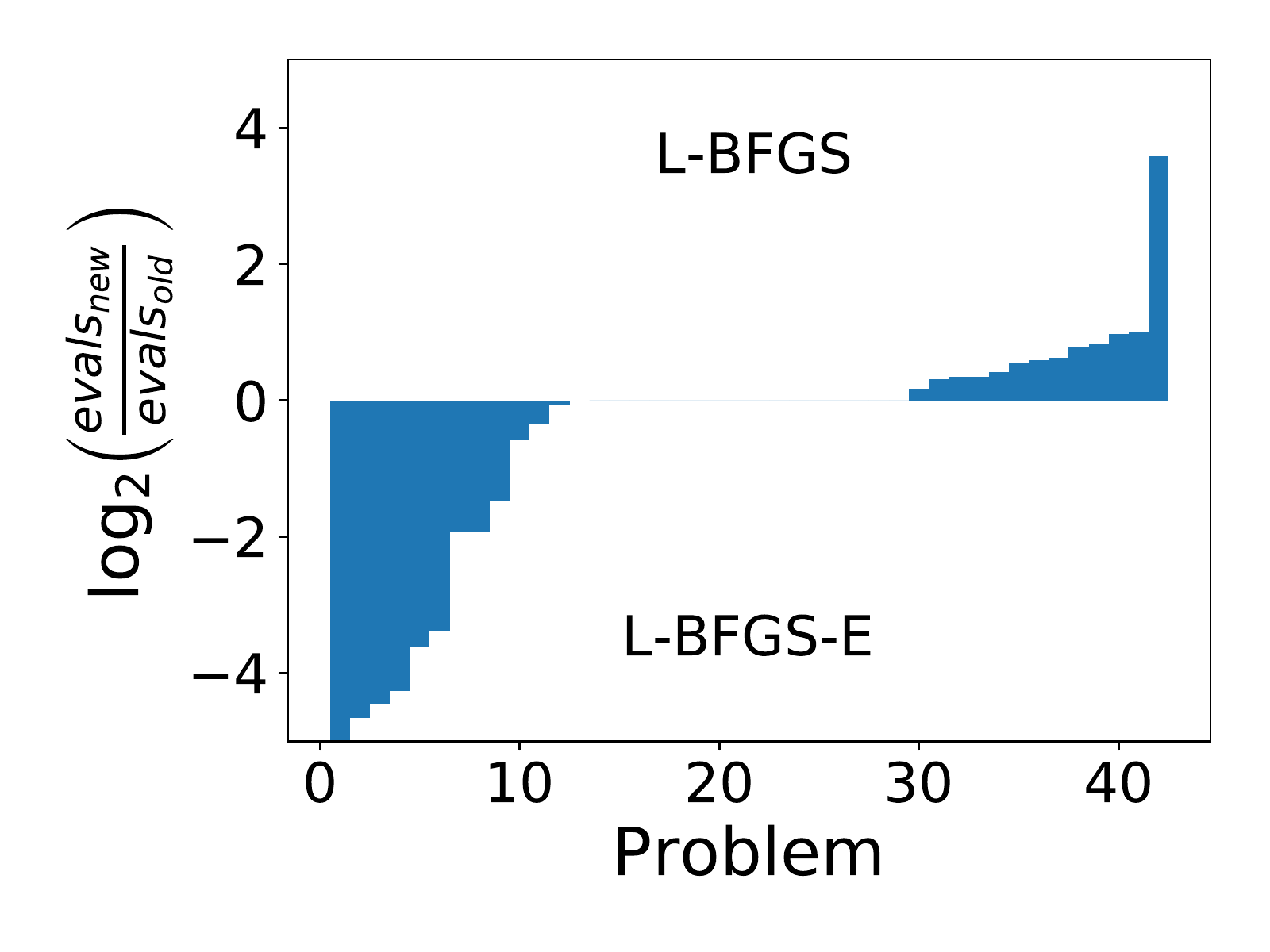}
    \caption{Morales profiles for the total number of gradient evaluations to achieve \eqref{twoterm} across 41 unconstrained CUTEst problems with $\xi_f = 10^{-3}$ and $\xi_g = 10^{-3}$. Results are averaged over 5 runs. The left figure compares BFGS against BFGS-E while the right figure compares L-BFGS against L-BFGS-E.}
    \label{fig:morales grad evals 1}
\end{figure}

Figure~\ref{fig:morales obj 1} thus compares the (long term) ability of the methods to achieve high accuracy in the function value, whereas Figure~\ref{fig:morales grad evals 1} measures the short-term cost in terms of gradient evaluations to achieve the noise level in the function or gradient. These results suggest that the noise tolerant methods often provide a real improvement in the solution of certain classes of optimization problems with noisy function and gradient evaluations.

\section{Final Remarks}

Although quasi-Newton methods are widely used in practice, the question of how to make BFGS and L-BFGS tolerant to errors in the function and gradient has not received sufficient attention in the literature. 


This paper makes two contributions. It introduces the noise control condition \eqref{eq:noise}, which can be used to determine when to skip a quasi-Newton update or adaptively lengthen the interval from which gradient differences can be employed reliably. Our proposed BFGS-E and L-BFGS-E methods utilize the latter and enjoy convergence guarantees to a neighborhood of the solution for strongly convex functions.

The second contribution of the paper is to show that the lengthening procedure based on condition \eqref{eq:noise} is successful in practice, and thus transforms the theoretical algorithm proposed in \cite{xie2020analysis} into a robust and practical procedure.  
Our numerical experiments show that quasi-Newton updating remains stable after the algorithm has reached the region where errors dominate, and this allows the noise tolerant methods to reach higher accuracy in the solution. 
Our testing also shows that the proposed algorithms are not more expensive than the standard BFGS and L-BFGS methods in the region where the latter two methods operate reliably. Once the iterates reach a neighborhood where BFGS updating is corrupted and the iteration stalls, the new algorithms invoke the lengthening procedure that typically requires $2-4$ gradient evaluations per iteration. We also tested an update skipping strategy based on the noise tolerant condition. We found that, although update skipping can be very efficient when applied to easy problems with uniform noise, the noise tolerant methods are more efficient when applied to harder problems or problems with oscillating noise.


We have made both implementations of the  BFGS-E and L-BFGS-E algorithms available on GitHub\footnote{\hyperlink{https://github.com/hjmshi/noise-tolerant-bfgs}{https://github.com/hjmshi/noise-tolerant-bfgs}}.

\section*{Acknowledgements}

We thank David Bindel, Jorge Mor\'e, Ping Tak Peter Tang, and Andreas Waechter for their valuable input on this work and for their suggestions of problems with computational noise. We also thank Shigeng Sun and Melody Qiming Xuan for their feedback on the manuscript.

\bibliographystyle{siamplain}
\bibliography{references}

\end{document}

%% file: ex_shared.tex
\usepackage{lipsum}
\usepackage{amsfonts}
\usepackage{graphicx}
\usepackage{epstopdf}
\usepackage{algorithmicx}
\usepackage{algpseudocode}
\ifpdf
  \DeclareGraphicsExtensions{.eps,.pdf,.png,.jpg}
\else
  \DeclareGraphicsExtensions{.eps}
\fi

\newsiamremark{remark}{Remark}
\newsiamremark{hypothesis}{Hypothesis}
\crefname{hypothesis}{Hypothesis}{Hypotheses}
\newsiamthm{claim}{Claim}
\newsiamthm{assumption}{Assumption}

\headers{A Noise-Tolerant Quasi-Newton Algorithm}{ H.-J.M. Shi, Y. Xie, R.H. Byrd, and J. Nocedal}

\title{A Noise-Tolerant Quasi-Newton Algorithm for Unconstrained Optimization}

\author{       
        Hao-Jun Michael Shi\thanks{Department of Industrial Engineering and Management Sciences, Northwestern University, Evanston, IL, USA. \email{hjmshi@u.northwestern.edu}.
         This author was supported by a grant from Facebook, and by National Science Foundation grant DMS-1620022.}
       \and  
       Yuchen Xie\thanks{Department of Industrial Engineering and Management Sciences, Northwestern University, Evanston, IL, USA. \email{ycxie@u.northwestern.edu}. 
         This author was supported by the Office of Naval Research grant N00014-14-1-0313 P00003.}
       \and   
       Richard Byrd\thanks{Department of Computer Science, University of Colorado,
        Boulder, CO, USA. \email{richard@cs.colorado.edu}.  This author was supported by National Science Foundation grant DMS-1620070.}
        \and
       Jorge Nocedal\thanks{Department of Industrial Engineering and Management Sciences, Northwestern University, 
       Evanston, IL, USA. \email{j-nocedal@northwestern.edu}. This author was supported by the Office of Naval Research grant N00014-14-1-0313 P00003, and by National Science Foundation grant DMS-1620022.}
      }

\usepackage{amsopn}

\makeatletter
\newcommand*{\addFileDependency}[1]{
  \typeout{(#1)}
  \@addtofilelist{#1}
  \IfFileExists{#1}{}{\typeout{No file #1.}}
}
\makeatother


%% file: ex_article.bbl
\begin{thebibliography}{10}

\bibitem{petsc}
{\sc S.~Balay, W.~D. Gropp, L.~C. McInnes, and B.~F. Smith}, {\em {PETS}c users
  manual}, Tech. Report Report ANL-95/11, Revision 2.1.1, Argonne National
  Laboratory, Argonne, Illinois, USA, 2001.

\bibitem{barton1992computing}
{\sc R.~R. Barton}, {\em Computing forward difference derivatives in
  engineering optimization}, Engineering optimization, 20 (1992), pp.~205--224.

\bibitem{berahas2019derivative}
{\sc A.~S. Berahas, R.~H. Byrd, and J.~Nocedal}, {\em Derivative-free
  optimization of noisy functions via quasi-newton methods}, SIAM Journal on
  Optimization, 29 (2019), pp.~965--993.

\bibitem{berahas2019linear}
{\sc A.~S. Berahas, L.~Cao, K.~Choromanski, and K.~Scheinberg}, {\em Linear
  interpolation gives better gradients than {G}aussian smoothing in
  derivative-free optimization}, arXiv preprint arXiv:1905.13043,  (2019).

\bibitem{berahas2019theoretical}
{\sc A.~S. Berahas, L.~Cao, K.~Choromanski, and K.~Scheinberg}, {\em A
  theoretical and empirical comparison of gradient approximations in
  derivative-free optimization}, arXiv preprint arXiv:1905.01332,  (2019).

\bibitem{berahas2019global}
{\sc A.~S. Berahas, L.~Cao, and K.~Scheinberg}, {\em Global convergence rate
  analysis of a generic line search algorithm with noise}, arXiv preprint
  arXiv:1910.04055,  (2019).

\bibitem{berahas2016multi}
{\sc A.~S. Berahas, J.~Nocedal, and M.~Tak{\'a}c}, {\em A multi-batch {L-BFGS}
  method for machine learning}, in Advances in Neural Information Processing
  Systems, 2016, pp.~1055--1063.

\bibitem{bertsekas2015convex}
{\sc D.~P. Bertsekas}, {\em Convex Optimization Algorithms}, Athena Scientific,
  2015.

\bibitem{bollapragada2018progressive}
{\sc R.~Bollapragada, D.~Mudigere, J.~Nocedal, H.-J.~M. Shi, and P.~T.~P.
  Tang}, {\em A progressive batching {L-BFGS} method for machine learning}, in
  International Conference on Machine Learning, 2018, pp.~620--629.

\bibitem{byrd2016stochastic}
{\sc R.~H. Byrd, S.~L. Hansen, J.~Nocedal, and Y.~Singer}, {\em A stochastic
  quasi-{N}ewton method for large-scale optimization}, SIAM Journal on
  Optimization, 26 (2016), pp.~1008--1031.

\bibitem{ByNoTool}
{\sc R.~H. Byrd and J.~Nocedal}, {\em A tool for the analysis of quasi-{N}ewton
  methods with application to unconstrained minimization}, SIAM Journal on
  Numerical Analysis, 26 (1989), pp.~727--739.

\bibitem{caflisch1998monte}
{\sc R.~E. Caflisch}, {\em Monte {C}arlo and quasi-{M}onte {C}arlo methods},
  Acta numerica, 7 (1998), pp.~1--49.

\bibitem{choi2000superlinear}
{\sc T.~Choi and C.~T. Kelley}, {\em Superlinear convergence and implicit
  filtering}, SIAM Journal on Optimization, 10 (2000), pp.~1149--1162.

\bibitem{DennWalker}
{\sc J.~Dennis and H.~Walker}, {\em Inaccuracy in {quasi-Newton} methods: Local
  improvement theorems}, in Mathematical Programming Studies, R.~K. Korte~B.,
  ed., vol.~22, Springer, 1984.

\bibitem{GillMurrSaunWrig83}
{\sc P.~E. Gill, W.~Murray, M.~A. Saunders, and M.~H. Wright}, {\em Computing
  forward-difference intervals for numerical optimization}, SIAM Journal on
  Scientific and Statistical Computing, 4 (1983), pp.~310--321.

\bibitem{GillMurrWrig81}
{\sc P.~E. Gill, W.~Murray, and M.~H. Wright}, {\em Practical Optimization},
  Academic Press, London, 1981.

\bibitem{gould2015cutest}
{\sc N.~I. Gould, D.~Orban, and P.~L. Toint}, {\em {CUTEst}: a constrained and
  unconstrained testing environment with safe threads for mathematical
  optimization}, Computational Optimization and Applications, 60 (2015),
  pp.~545--557.

\bibitem{cuter}
{\sc N.~I.~M. Gould, D.~Orban, and P.~L. Toint}, {\em {\sf CUTEr} and {\sf
  {s}if{d}ec}: A {C}onstrained and {U}nconstrained {T}esting {E}nvironment,
  revisited}, ACM Trans. Math. Softw., 29 (2003), pp.~373--394.

\bibitem{gower2016stochastic}
{\sc R.~M. Gower, D.~Goldfarb, and P.~Richt{\'a}rik}, {\em Stochastic block
  {BFGS}: squeezing more curvature out of data}, in Proceedings of the 33rd
  International Conference on Machine Learning, 2016.

\bibitem{kelley2011implicit}
{\sc C.~T. Kelley}, {\em Implicit filtering}, vol.~23, SIAM, 2011.

\bibitem{LiuNocedal89}
{\sc D.~C. Liu and J.~Nocedal}, {\em On the limited memory {BFGS} method for
  large scale optimization}, Mathematical Programming, 45 (1989), pp.~503--528.

\bibitem{morales2011remark}
{\sc J.~L. Morales and J.~Nocedal}, {\em Remark on “{A}lgorithm 778:
  {L-BFGS-B}: {F}ortran subroutines for large-scale bound constrained
  optimization”}, ACM Transactions on Mathematical Software (TOMS), 38
  (2011), pp.~1--4.

\bibitem{more2011estimating}
{\sc J.~J. Mor{\'e} and S.~M. Wild}, {\em Estimating computational noise}, SIAM
  Journal on Scientific Computing, 33 (2011), pp.~1292--1314.

\bibitem{more2012estimating}
{\sc J.~J. Mor{\'e} and S.~M. Wild}, {\em Estimating derivatives of noisy
  simulations}, ACM Transactions on Mathematical Software (TOMS), 38 (2012),
  p.~19.

\bibitem{moritz2016linearly}
{\sc P.~Moritz, R.~Nishihara, and M.~Jordan}, {\em A linearly-convergent
  stochastic {L-BFGS} algorithm}, in Artificial Intelligence and Statistics,
  2016, pp.~249--258.

\bibitem{nesterov2017random}
{\sc Y.~Nesterov and V.~Spokoiny}, {\em Random gradient-free minimization of
  convex functions}, Foundations of Computational Mathematics, 17 (2017),
  pp.~527--566.

\bibitem{mybook}
{\sc J.~Nocedal and S.~Wright}, {\em Numerical {O}ptimization}, Springer New
  York, 2~ed., 1999.

\bibitem{schraudolph2007stochastic}
{\sc N.~N. Schraudolph, J.~Yu, and S.~G{\"u}nter}, {\em A stochastic
  {quasi-Newton} method for online convex optimization}, in International
  Conference on Artificial Intelligence and Statistics, 2007, pp.~436--443.

\bibitem{xie2020analysis}
{\sc Y.~Xie, R.~H. Byrd, and J.~Nocedal}, {\em Analysis of the {BFGS} method
  with errors}, SIAM Journal on Optimization, 30 (2020), pp.~182--209.

\bibitem{ypma}
{\sc T.~J. Ypma}, {\em The effect of rounding errors on {N}ewton-like methods},
  IMA Journal of Numerical Analysis, 3 (1983), pp.~109--118.

\end{thebibliography}
